\documentclass[a4paper,twoside,10pt]{amsart}

\usepackage{amsthm,amsmath,amssymb,amscd}
\usepackage[greek,english]{babel}
\usepackage{graphicx}

\newtheorem{theorem}{Theorem}[section]
\newtheorem{proposition}[theorem]{Proposition}
\newtheorem{lemma}[theorem]{Lemma}
\newtheorem{corollary}[theorem]{Corollary}

\theoremstyle{definition}
\newtheorem{definition}[theorem]{Definition}

\newtheorem{example}[theorem]{Example}
\theoremstyle{remark}
\newtheorem{remark}[theorem]{Remark}
\newtheorem{Remark}[theorem]{Remarks}

\newlength{\szer}

\renewcommand{\qedsymbol}{$\square$}

\def\Par{{\par\noindent}}
\def\A{{\mathbf A}}

\def\C{{\mathbf C}}

\def\Ff{\mathcal F}

\def\Ll{{\mathcal L}}
\def\N{{\mathbf N}}

\def\P{{\mathbf P}}
\def\Pp{{\mathcal P}}
\def\Q{{\mathbf Q}}
\def\Qq{{\mathcal Q}}
\def\R{\mathbf R}

\def\Z{{\mathbf Z}}

\def\int{\mathrm{int}}

\def\elem(#1,#2){  \{ \frac{#1}{\overline{\ #2\ }}\}  }

\begin{document}
\bibliographystyle{amsplain}
\setcounter{tocdepth}{2}
\title{OVERWEIGHT DEFORMATIONS OF AFFINE TORIC VARIETIES AND LOCAL UNIFORMIZATION}
\author{Bernard Teissier}
\address{Institut de Math\'ematiques de Jussieu - Paris Rive Gauche, UMR 7586 du CNRS, 
B\^atiment Sophie Germain, Case 7012,
75205 PARIS Cedex 13, France}

\email{teissier@math.jussieu.fr}

\keywords{Toric geometry, Valuations, Key polynomials}

\subjclass[2000]{14M25, 14E15, 14B05}

\dedicatory{ \textit {To the memory of Shreeram S. Abhyankar}}

\begin{abstract}
 Given an equicharacteristic complete noetherian local domain $R$ with algebraically closed residue field $k$, we first present a combinatorial proof of \emph{embedded} local uniformization for zero-dimensional valuations of $R$ whose associated graded ring ${\rm gr}_\nu R$ with respect to the filtration defined by the valuation is a finitely generated $k$-algebra. The main idea here is that some of the birational toric maps which provide embedded pseudo-resolutions for the affine toric variety corresponding to ${\rm gr}_\nu R$ also provide local uniformizations for $\nu$ on $R$. These valuations are necessarily Abhyankar (for zero-dimensional valuations this means that the value group is $\Z^r$ with $r={\rm dim}R$).\Par
 In a second part we show that conversely, given an excellent noetherian equicharacteristic local domain $R$ with algebraically closed residue field, if the zero-dimensional valuation $\nu$ of $R$ is Abhyankar, there are local domains $R'$ which are essentially of finite type over $R$ and dominated by the valuation ring $R_\nu$ ($\nu$-modifications of $R$) such that the semigroup of values of $\nu$ on $R'$ is finitely generated, and therefore so is the $k$-algebra ${\rm gr}_\nu R'$. Combining the two results and using the fact that Abhyankar valuations behave well under completion gives a proof of local uniformization for rational Abhyankar valuations and, by a specialization argument, for all Abhyankar valuations. \par
As a by-product we obtain a description of the valuation ring of a rational Abhyankar valuation as an inductive limit indexed by $\N$ of birational toric maps of regular local rings. One of our main tools, the valuative Cohen theorem, is then used to study the extensions of rational monomial Abhyankar valuations of the ring $k[[x_1,\ldots ,x_r]]$ to monogenous integral extensions and the nature of their key polynomials. In the conclusion we place the results in the perspective of local embedded resolution of singularities by a single toric modification after an appropriate re-embedding.
\end{abstract}

\footnote{AMS classification: Primary 14M25; Secondary 14E15, 14B05.
Keywords:
Toric geometry, Valuations, Uniformization, Key polynomials.}

\maketitle

\newpage
\tableofcontents

\section{Introduction}
After \cite{Te1}, this is the second stage of an attempt to prove that every singularity $(X,0)$ over an algebraically closed field $k$ can be embedded in an affine space $\A^N(k)$ over that field in such a way that there exist a system of coordinates, making $\A^N(k)$ a toric variety such that the intersection of $X$ with the torus is dense in $X$, and a toric proper and birational map of non singular toric varieties $Z\to \A^N(k)$ such that the strict transform of $(X,0)$ is non singular and transversal to the toric boundary. This problem of proving embedded resolution by a single toric morphism after a suitable re-embedding has recently been solved for projective varieties, assuming embedded resolution of singularities, by Jenia Tevelev (see \cite{Tev} and the conclusion of this paper).
\Par This paper is concerned with the "toric embedded local uniformization" avatar of the "toric embedded resolution" problem just mentioned and explained in \cite{Te2},  \cite{Te3}. \par
A valuation of a local domain $(R,m)$ is the datum of a valuation ring $(R_\nu, m_\nu)$ of the fraction field $K$ of $R$, containing $R$. We say that the valuation is centered in $m$ if $(R,m)$ is dominated by $(R_\nu, m_\nu)$ (i.e., $m_\nu\cap R=m$). The valuation is \emph{rational} if the corresponding residue fields extension is trivial. The problem of local uniformization of a valuation is to show the existence of a regular local ring $R'$ essentially of finite type over $R$ and dominated by $R_\nu$. Embedded local uniformization requires in addition that $R'$ should be obtained as strict transform of $R$ in a birational morphism, essentially of finite type, of regular local rings $S\to S'$ having respectively $R$ and $R'$ as quotients.\par\noindent This text deals mostly with the local uniformization of a class of valuations of equicharacteristic noetherian excellent local domains with an algebraically closed residue field. The focus is on local uniformization of rational valuations. \par
When the residue field $k$ of $R$ is algebraically closed rational valuations, which correspond to the $k$-rational points of the Riemann-Zariski manifold of valuations centered in ${\rm Spec}R$, concentrate the difficulty of local uniformization.\par
 The first purpose of this paper is to realize a part of the program for local uniformization of valuations of excellent equicharacteristic local rings with an algebraically closed residue field proposed in \cite {Te1}, in the special case of a complete local domain $(R,m)$ and a valuation whose semigroup of values is finitely generated.\par\noindent The main point is that the formal space corresponding to the local ring can be embedded in an affine space over the residue field $k=R/m$ in such a way that  its strict transform under a single toric birational modification of the ambient space is non singular at the point picked by the valuation.\par
Let us fix a valuation $\nu$ with ring $R_\nu$ on a noetherian equicharacteristic local domain $(R,m)$, and assume that it is centered in $m$.\Par
Set $k_\nu=R_\nu/m_\nu$ and recall Abhyankar's inequality $r(\nu)+{\rm tr}_kk_\nu\leq {\rm dim}R$, where $\hbox{\rm r}(\nu)$ is the rational rank  of the totally ordered abelian group $\Phi$ of values of the valuation $\nu$, and ${\rm tr}_kk_\nu$ is the transcendence degree of the residue fields extension, also called the \emph{dimension} of the valuation (see \cite{V0}, Th\'eor\`eme 9.2). We say that our valuation is \emph{rational} if $k_\nu=k$. We then have $\hbox{\rm r}(\nu)\leq{\rm dim}R$. Abhyankar valuations are defined as those for which Abhyankar's inequality is an equality and it is known (see \emph{loc.cit.}) that this implies $\Phi\cong \Z^{r(\nu)}$. In the case where $R$ is a $k$-algebra with residue field $k$ the rationality condition means that the centers of $\nu$ in all birational models proper over ${\rm Spec}R$ are closed points which are rational over $k$, hence the name. A rational valuation on $R$ is Abhyankar when $\Phi\cong \Z^{{\rm dim}R}$.\par\noindent Denote by $\Gamma=\nu(R\setminus \{0\})\subset \Phi_{\geq 0}=\Phi_+\cup\{0\}$ the semigroup of values of $\nu$ on $R$. The semigroup $\Gamma$ is well ordered since $R$ is noetherian and so we can denote by  $\Gamma=\langle (\gamma_i)_{i\in I}\rangle$ its minimal set of generators, indexed by an ordinal $I\leq \omega^h$ where $h$ is the real (or Archimedean) rank, of the valuation, also called its  height (see \cite{Te1}, corollary 3.10).\par
If we agree that $\nu(0)=+\infty$, an element larger than any element of $\Phi$, for any subring $R\subseteq R_\nu$, noetherian or not, the valuation $\nu$ defines a filtration of $R$ by the ideals $\Pp_\phi (R)=\{x\in R\vert\nu(x)\geq\phi\}, \Pp^+_\phi (R)=\{x\in R\vert\nu(x)>\phi\}$ and an associated graded ring $${\rm gr}_\nu R=\bigoplus_{\phi\in \Phi_{\geq 0}}\Pp_\phi (R)/\Pp^+_\phi(R).$$
Note that $\Pp_\phi(R)=R$ if $\phi\notin  \Phi_+$ and that the sum on the right is actually indexed by $\Gamma$. If $\Gamma$ is well ordered, the ideal $\Pp^+_\phi (R)$ is equal to $\Pp_{\phi^+} (R)$, where $\phi^+={\rm min}\{\psi\in \Gamma\vert \psi>\phi\}$ is the successor of $\phi$ in $\Gamma$.\par
It is shown in (\cite{Te1}, 2.3) that even without the assumption of rationality there a faithfully flat specialization of the ring $R$ to ${\rm gr}_\nu R$.\Par If the valuation is rational, each non zero homogeneous component of the $k$-algebra ${\rm gr}_\nu R$ is a one-dimensional vector space over $k$ and the algebra is generated by homogeneous elements $\overline \xi_i$ whose degrees $\gamma_i$ generate the semigroup $\Gamma$ (see \cite{Te1}, \S 4). In particular it is finitely generated if and only if the semigroup $\Gamma$ is.\Par Thus (see \emph{loc.cit.}), when the valuation is rational this graded ring is the quotient of a polynomial algebra (possibly in infinitely many variables) $k[(U_i)_{i\in I}]$ by a prime binomial ideal $(U^{m^\ell}-\lambda_\ell U^{n^\ell})_{\ell\in L}$, with $\lambda_\ell\in k^*$, and is isomorphic to the semigroup algebra over $k$ of the value semigroup $\Gamma$.
\Par Assuming again that $R$ is noetherian, by a theorem of Piltant (see \cite{Te1}, proposition 3.1) for zero dimensional valuations the Krull dimension of the ring ${\rm gr}_\nu R$, which is not necessarily noetherian, is equal to ${\rm r}(\nu)$. Thus, for zero dimensional valuations we have the inequality ${\rm dim}{\rm gr}_\nu R\leq {\rm dim}R$ and among them Abhyankar valuations are characterized\footnote{A typical example where strict inequality holds is due to Zariski and analyzed from the viewpoint of this paper in \cite{Te1}, Example 4.20.}by the equality ${\rm dim}{\rm gr}_\nu R={\rm dim}R$.\par
The constants $\lambda_\ell$ which appear in the binomial equations reflect the fact that the incarnation of ${\rm Specgr}_\nu R$ in the affine space with coordinates $(U_i)_{i\in I}$ is isomorphic to the closure of the orbit of a point under the action of the torus $k^{*r}$, but there is no canonical choice of the point, just as there is no canonical choice of a system of generators of the $k$-algebra ${\rm gr}_\nu R$. Once a system of generators is fixed, these constants have a geometric interpretation in terms of coordinates in an affine chart of the point picked by the valuation in a toric modification of our ring (see \cite{Te1}, remark 5.12).\par Taking as in \cite{Te1} the point of view that $R$ is a deformation of its associated graded ring ${\rm gr}_\nu R$, we start from the fact that the affine variety defined by a prime ideal generated by binomials $U^{m^\ell}-\lambda_\ell U^{n^\ell}$ of $k[(U_i)_{i\in F}]$, with constants $\lambda_\ell\in k^*$, is a toric variety and therefore has, over any algebraically closed field $k$, toric \emph{embedded} resolutions of singularities described in \cite{GP-T1} from an intrinsic and general viewpoint and in (\cite{Te1}, 6.2) from the equational viewpoint, which is more adapted here since we deform equations. Given a weight on the polynomial or power series ring which is compatible with the binomials, we define a class of deformations of such affine toric varieties and their formal completions at the origin, which is determined by equidimensionality of the fibers and weight conditions. These deformations have the property that the monomial order induced by the weight determines a unique valuation on the ring $R$ of the general fiber of the deformation. The main point here is that \emph{some} of the toric (pseudo-)resolutions of the affine toric varieties extend to local uniformizations of this valuation. It is then shown in section \ref{OW} that any complete equicharacteristic local ring with a rational valuation such that the associated graded ring is a finitely generated algebra can be obtained in this way and thus can be uniformized in this manner.\par The idea underlying \cite{Te1} is that \emph{every} rational valuation of a complete noetherian equicharacteristic local domain $R$ can be obtained by overweight deformation from a weighted affine toric variety (possibly of infinite embedding dimension) and that because $R$ itself is of finite embedding dimension the valuation can be uniformized by a birational toric map of a suitably large ambient space, corresponding to a \emph{partial} pseudo-resolution of the toric variety. \par The concept of overweight deformation is one of the avatars of a general principle in singularity theory, which is that adding to equations terms of "higher weight" tends to preserve geometric features at the origin of the zero set of these equations. Here we deal with the behavior of strict transforms under toric modifications and it is a combinatorial problem on exponents, which is somewhat more complicated when the weights take value in a totally ordered group of rank $>1$. \par\medskip
In a second part we show that when $R$ is complete with an algebraically closed residue field the condition that the value semigroup of the rational valuation $\nu$ on a $\nu$-modification\footnote{This means the ring obtained by localizing a birational modification essentially of finite type (which we may assume to be a blowing-up) of ${\rm Spec}R$ at the point picked by the valuation} of $R$ is finitely generated is equivalent to the valuation being Abhyankar. Here the key ingredients are the equational version of the smoothness over ${\rm Spec}\Z$ of the torus of the toric scheme ${\rm Spec}\Z[t^\Gamma]$, and the version of key polynomials provided for Abhyankar valuations by the valuative Cohen theorem (theorem \ref{Valco}) of section \ref{cohen}. The first ingredient allows us to find, after a suitable birational $\nu$-modification $(R,m)\to (R',m')$ of our original ring, an injection of complete local rings $k[[x'_1,\ldots, x'_r]]\subset \hat R'^{m'}$ which is finite, and such that the extension of the value groups of their fraction fields is tame, in the usual sense that its index is not divisible by the characteristic of the residue field. It may be very different from the injection corresponding to $r$ elements of a well chosen minimal set of generators of the maximal ideal of $\hat R'^{m'}$, as one sees in remark \ref{keyrem}. This puts us in position to build key polynomials thanks to the valuative Cohen theorem, and we use them to produce a contradiction from the assumption that the semigroup of our Abhyankar valuation is not finitely generated. The proof also shows that the fraction fields extension corresponding to the finite tame injection has to be separable.\par
 The good behavior of Abhyankar valuations under completion (see subsection \ref{extcomp}) allows us to deduce the same result of finite generation after a $\nu$-modification for an excellent equicharacteristic local domain $R$, and finally to obtain local uniformization of rational Abhyankar valuations, using the result of the first part. By a specialization argument, the local uniformization of rational Abhyankar valuations of $R$ implies the local uniformization of all Abhyankar valuations of $R$.\par This result is closely related to the work of Knaf-Kuhlmann (see \cite{KK1}) proving  a field-theoretic version of local uniformization for Abhyankar valuations of algebraic function fields, as well as to the work of M. Temkin (see \cite{Tem}) on inseparable local uniformization, which in particular proves local uniformization for Abhyankar valuations of algebraic function fields and also uses toroidal methods.\par
 A consequence in valuation theory is that for the excellent local domains we study, it is equivalent for a rational valuation to be Abhyankar and to be quasi monomial, a fact which was first shown to be true by Dale Cutkosky for rank one valuations of algebraic function fields assuming embedded resolution of singularities. In any characteristic and for arbitrary Abhyankar valuations of algebraic function fields separable over the base field, it is a consequence of the more recent work of Knaf-Kuhlmann mentioned above.\par
We obtain a stronger result, theorem \ref{LUTOR} which is a description of the valuation ring of a rational Abhyankar valuation as an inductive limit indexed by $\N$ of birational toric maps of regular local rings. The inductive system is rather explicitly related to the ordering of the value group by an extension of the Jacobi-Perron algorithm.\par The valuative Cohen theorem also provides a proof of the defectlessness of rational Abhyankar valuations of the field $k((x_1,\ldots ,x_r))$ and an analogue for rational Abhyankar valuations on the rings of hypersurfaces of Abhyankar's irreducibility criterion for plane curves. We also give a characterization of the semigroups of rational Abhyankar valuations of hypersurfaces which generalizes the plane branch case. \par Two sections propose comments on the use of key polynomials in the proof and in the study of Artin-Schreier extensions. In the conclusion we place the result in the perspective of local embedded resolution of singularities by a single toric modification after an appropriate re-embedding.
\section{Weighted affine toric varieties}\label{waff}
Let $k$ be a field and $\Phi$ a totally ordered abelian group of finite rational rank $r$.\par\noindent
\begin{definition}1) A \textit{weight} on the rings $k[U_1,\ldots ,U_N]$ or $k[[U_1,\ldots ,U_N]]$ is an homomorphism of groups $b\colon \Z^N\to \Phi$ which is induced by an homomorphism of semigroups $\N^N\to \Phi_{\geq 0}$. It defines a weight on monomials by $w(U^m)=b(m)$ and an additive total preorder (or monomial preorder) by $U^m\leq U^n\leftrightarrow b(m)\leq b(n)$. We shall often write also $w(m)$ instead of $b(m)$ for $w(U^m)$.\Par
2) In this text we shall assume that the weight is non trivial, which means that the homomorphism is not zero. The set $\{w(U_i)\}_{1\leq i\leq N}$ is well ordered since it is finite, and by a result of B.H. Neumann (see \cite{Ne}) the subsemigroup of $\Phi_{\geq 0}$ that it generates is also well ordered. The weight of a series is the least weight of its terms. \end{definition}
\begin{remark}
\small{\emph{A weight on $k[U_1,\ldots ,U_N]$ or $k[[U_1,\ldots ,U_N]]$ determines a monomial valuation by $$\nu_w(P)=w(P)={\rm min}_{m\in E(P)}w(U^m)\ {\rm when}\ P=\Sigma_{m\in E(P)} c_mU^m, c_m\in k^*.$$ Provided that $w(U_i)>0$ for $i=1,\ldots, N$, the valuation is rational. and conversely a monomial rational valuation on such a ring \textit{endowed with a coordinate system} determines a weight with $w(U_i)>0$ for all $i$. }}\par
\end{remark}
This rather general situation can be reduced to a more familiar one thanks to the following result: \par\noindent
\begin{proposition} \label{finapp} {\rm (see \cite{Te1}, proposition 4.12)}\footnote{In \cite{Te1}, 4.3, it is implicit that the family has a smallest element, the semigroup generated by the chosen rationally independent elements, and that the nested sequence is indexed by $\N$. It is made explicit here. Since \cite{Te1} was published, I have learnt that a version (without the nestedness and the algorithmic aspect) of this result is due to George A. Elliott; see \cite{EL}. Hagen Knaf and Franz-Viktor Kuhlmann have remarked (see \cite{Ku}, Lemma 15.4 and \cite{KK1}, Lemma 4.2 and the lines above it) that Theorem 1 in Zariski's paper \cite{Z} can also be interpreted as a result of the same nature for groups of rank one. A version of Elliott's result also appears in the more recent paper (\cite{Tem}, Appendix) of M. Temkin.} Let $\Phi_{\geq 0}$ be the positive semigroup  of a totally ordered abelian group of finite rational rank $r$. Choose $r$ rationally independent elements in $\Phi_+$ and let $\N_{(0)}^r\subset\Phi_{\geq 0}$ be the free subsemigroup they generate. Then $\Phi_{\geq 0}$ is the union of a family $(\N^r_{(h)})_{h\in \N}$, of nested  free subsemigroups of rank $r$: $$\N_{(0)}^r\subset \cdots\subset \N_{(h)}^r\subset \N_{(h+1)}^r\subset\cdots\subset \Phi_{\geq 0},$$ the inclusions being semigroup maps.\end{proposition}
As explained in \textit{loc.cit.}, this result and the algorithmic aspect of its proof can be viewed as an extension of the Jacobi-Perron algorithm for approximating directions of vectors in $\R^N$ by directions of integral vectors.\par\medskip
Since the image $b(\N^N)$ of the weight map is a finitely generated subsemigroup of $\Phi_{\geq 0}$ it is contained in some free subsemigroup $\N_{(h)}^r\subset \Phi_{\geq 0}$ and the subgroup of $\Phi$ which this image generates is finitely generated and free of rank $r'\leq r$, totally ordered by the order of $\Phi$. Renaming $r'$ into $r$ \emph{we may and will assume in the sequel that $\Phi=\Z^r$ with a total order and that the map $b\colon \Z^N\to \Z^r$ is surjective}.\par Of course we do not need proposition \ref{finapp} to obtain this, since by construction the image of the map $b$ is finitely generated and torsion-free. The placement of the semigroup $b(\N^N)$ in the sequence of the $\N^r_{(h)}$ will be useful later. \par
\par\medskip\noindent
Recall the following corollary of proposition \ref{finapp}; here a \textit{term} is the product of a monomial by a nonzero constant:
\begin{corollary}\label{grLU}{\rm (see \cite{Te1}, proposition 4.15)} Let $R_\nu$ be the valuation ring of a valuation with value group $\Phi$, of finite rational rank $r$. Choose $r$ homogeneous elements $x_1^{(0)},\ldots ,x_r^{(0)}\in \hbox{\rm gr}_\nu R_\nu$ whose valuations are rationally independent. The graded algebra $\hbox{\rm gr}_\nu R_\nu$ is the union of a nested family of polynomial algebras in $r$ variables over $k_\nu=R_\nu/m_\nu$ with maps between them sending each variable to a term:
$$k_\nu[x_1^{(0)},\ldots ,x_r^{(0)}]\subset\ldots \subset k_\nu[x_1^{(h)},\ldots ,x_r^{(h)}]\subset k_\nu[x_1^{(h+1)},\ldots ,x_r^{(h+1)}]\subset\ldots\subset \hbox{\rm gr}_\nu R_\nu.$$
\end{corollary} These subalgebras are the semigroup algebras of nested free subsemigroups of the non negative part $\Phi_{\geq 0}$ of $\Phi$.\par As was observed in \cite{Te1}, 4.3, this corollary can be viewed as implying a graded version of local uniformization: a finitely generated graded $k_\nu$-subalgebra of $\hbox{\rm gr}_\nu R_\nu$ is contained in a polynomial (i.e., graded regular) subalgebra.
\begin{remark}\label{unimod} \small{\emph{In the case which will interest us most here, the group $\Phi$ is $\Z^r$ and then there exists an integer $h_0$ such that for $h\geq h_0$ the maps $\N_{(h)}^r\subset \Z^r$ are unimodular, since the indices of the images in $\Phi$ of the groups $\Z_{(h)}^r$ decrease as $h$ increases and must eventually become one. Of course if we can take the $r$ rationally independent elements to be a basis of $\Z^r$ we have $h_0=0$}.\par
\emph{ In the language of corollary \ref{grLU} unimodularity means the corresponding inclusions of polynomial rings are, up to a homothety on each variable $x_k^{(h)}$, birational toric maps.}}
\end{remark}
\par\medskip

We now recall some facts which can be found in \cite{Stu} and \cite{Ei-S}. Denoting by $e_i$ the $i$-th basis vector of $\Z^N$ and setting $\gamma_i=b (e_i)\in \Z^r$, we may consider the semigroup $\Gamma$ generated by $\gamma_1,\ldots ,\gamma_N$ and the toric variety $X_0={\rm Spec}k[t^\Gamma]$ where $k[t^\Gamma]$ is the semigroup algebra with coefficients in $k$. It is the closure of the orbit of the point $(1,1,\ldots ,1)\in \A^N(k)$ under the action of the torus $k^{*r}$ determined by $(t,z_1,\ldots ,z_N)\mapsto (t^{\gamma_1}z_1,\ldots ,t ^{\gamma_N}z_N)$ with $t=(t_1,\ldots ,t_r)\in k^{*r}$ and $t^{\gamma_j}=t_1^{\gamma_{j1}}...t_r^{\gamma_{jr}}$\Par
Denoting by $\Ll$ the lattice which is the kernel of $b$, we may choose a system of generators $(m^\ell-n^\ell)_{\ell\in L}$ for $\Ll$, where $m^\ell$ and $-n^\ell$ are respectively the non negative and negative part of the vector $m^\ell-n^\ell$ so that the entries of $m^\ell$ and $n^\ell$ are non negative, and which is such that the ideal $F_0$ of $k[U_1,\ldots ,U_N]$ generated by the $(U^{m^\ell}-U^{n^\ell})_{\ell\in L}$ is a prime binomial ideal defining the embedding $X_0\subset \A^N(k)$.\Par
\emph{We note that the vectors $m^\ell-n^\ell$ do not 
 constitute a minimal system of generators of $\Ll$}. (See \cite{Ei-S}, corollary 2.3.)\par In this way a surjective weight map $b\colon\Z^N\to\Z^r$ determines an affine toric variety corresponding to the affine semigroup $\Gamma=b(\N^N)$. This is what we call a weighted affine toric variety. If $k$ is algebraically closed, any reduced and irreducible affine toric variety in $\A^N(k)$ can be obtained in this way; its prime binomial ideal corresponds to a system of generators of the saturated lattice $\Ll\subset\Z^N$ (see \cite{Ei-S}, Theorem 2.1) and so to a surjective map $b\colon\Z^N\to\Z^r=\Z^N/\Ll$ and it suffices to choose a total monomial order on $\Z^r$ such that $b(\N^N)\subseteq \Z^r_{\geq 0}$  to obtain a weight.\Par\emph{Thus, the datum of a weight is equivalent to the datum of a total monomial preorder on $\Z^N$ such that $\N^N\subseteq \Z^N_{\geq 0}$, the lattice $\Ll$ appearing as the lattice of elements preorder-equivalent to $0$.}\Par
The Krull dimension of $k[t^\Gamma]$ is equal to $r$; it is the dimension of $X_0$.

\section{Overweight deformations of prime binomial ideals}\label{OD}

Let $w$ be a weight on a polynomial or power series ring over a field $k$, with values in the positive part $\Phi_{\geq 0}$ of a totally ordered group $\Phi$ of finite rational rank. \par Let us consider the power series case and the ring $S=k[[u_1,\ldots ,u_N]]$. Consider the filtration of $S$ indexed by $\Phi_{\geq 0}$ and determined by the ideals $\Qq_\phi$ of elements of weight $\geq \phi$, where the weight of a series is the minimum weight of a monomial appearing in it. Defining similarly $\Qq^+_\phi$ as the ideal of elements of weight $>\phi$, the graded ring associated to this filtration is the polynomial ring $$
\bigoplus_{\phi\in \Phi_{\geq 0}}\Qq_\phi/\Qq^+_\phi=k[U_1,\ldots ,U_N],$$
with $U_i={\rm in}_wu_i$, graded by ${\rm deg}U_i=w(u_i)$.

\begin{definition}\label{overweight}
Given a weight $w$ as above, a (finite dimensional) overweight deformation\footnote{As far as I know, overweight (and underweight) deformations of binomial ideals were considered for the first time by Henry Pinkham in his Ph.D. thesis \cite{Pi} under the name of deformations of negative (resp.  positive) weight with respect to a natural action of the multiplicative group $k^*$ on the basis of a miniversal deformation of a monomial curve.}   is the datum of a prime binomial ideal $(u^{m^\ell}-\lambda_\ell u^{n^\ell})_{1\leq \ell \leq s},\ \lambda_\ell\in k^*$, of $S=k[[u_1,\ldots ,u_N]]$ such that the vectors $m^\ell-n^\ell\in \Z^N$ generate the lattice of relations between the $\gamma_i=w(u_i)$, and of series
$$\begin{array}{lr}
F_1=u^{m^1}-\lambda_1 u^{n^1}+\Sigma_{w(p)>w(m^1)} c^{(1)}_pu^p\\
F_2=u^{m^2}-\lambda_2 u^{n^2}+\Sigma_{w(p)>w(m^2)} c^{(2)}_pu^p\\
.....\\
F_\ell=u^{m^\ell}-\lambda_\ell u^{n^\ell}+\Sigma_{w(p)>w(m^\ell)} c^{(\ell)}_pu^p\\
.....\\
F_s=u^{m^s}-\lambda_s u^{n^s}+\Sigma_{w(p)>w(m^s)} c^{(s)}_pu^p\\
\end{array}\leqno{(OD)}
$$ in $k[[u_1,\ldots ,u_N]]$ such that, with respect to the monomial order determined by $w$, they form a standard basis for the ideal which they generate: their initial forms generate the ideal of initial forms of elements of that ideal.\par\noindent
Here we have written $w(p)$ for $w(u^p)$ and the coefficients $c^{(\ell)}_p$ are in $k$.
\end{definition}
Let us denote by $X$ the formal subspace of $\A^N(k)$ defined by the ideal $F=(F_1,\ldots ,F_s)$.
\begin{remark} \small{\emph{A prime binomial ideal $F_0$ in $k[u_1,\ldots ,u_N]$ remains prime after extension to $k[[u_1,\ldots ,u_N]]$ because the completion of $k[u_1,\ldots ,u_N]/F_0$ is the completion of the local ring of the toric variety at the point picked by the weight, which is the origin (closed orbit). This completion is an integral domain because its associated graded ring with respect to the filtration determined by the weight is an integral domain, the ring of the affine toric variety. See also \cite{GP1}, Lemma 1.}}\end{remark}
Recall from (\cite{Te1}, 6.2) and \cite{GP-T1} that the difference between a \emph{pseudo-resolution} and a resolution of singularities is that the pseudo-resolution is not required to induce an isomorphism outside of the singular locus.
\begin{proposition}\label{resoverwght}\par\noindent
a) Assuming that $w(u_i)>0$ for $i=1,\ldots , N$, an overweight deformation determines a rational valuation $\nu$ of the ring $R=S/(F_1,\ldots ,F_s)$ with value group $\Phi$, and the associated graded ring of $R$ is $$\hbox{\rm gr}_\nu R=k[U_1,\ldots ,U_N]/(U^{m^1}-\lambda_1 U^{n^1},\ldots ,U^{m^s}-\lambda_s U^{n^s}).$$
The associated graded map associated to the surjective map $\pi\colon S\to R$, the filtration by weight on $S$ and the filtration associated to the valuation $\nu$ on $R$, is the surjective map{\small
$${\rm gr}_w\pi\colon{\rm gr}_wS=k[U_1,\ldots ,U_N]\to k[U_1,\ldots ,U_N]/(U^{m^1}-\lambda_1 U^{n^1},\ldots ,U^{m^s}-\lambda_s U^{n^s})=\hbox{\rm gr}_\nu R.$$ }\noindent
b) Assuming that $k$ is algebraically closed, there exist birational toric maps \break$\pi(\Sigma)\colon Z(\Sigma)\to \A^N(k)$ which are embedded pseudo-resolutions of the irreducible affine binomial variety ${\rm Spec}\hbox{\rm gr}_\nu R$ and such that the strict transform by $\pi(\Sigma)$ of
$X\subset \A^N(k)$ is regular and transversal to the toric boundary at the point picked by $\nu$.\end{proposition}
\par\noindent
Before entering the proof, let us point out the:

\begin{corollary}\label{seeAb} The ring $R=S/F$ determined by an overweight deformation of a prime binomial ideal is an integral domain of the same dimension as the toric variety corresponding to the binomial ideal. In particular, the valuation $\nu$ is Abhyankar: the rational rank of its value group is equal to the dimension of $R$. \end{corollary}
\begin{proof}
 \textbf{Proof of a) and comments:}\par\noindent Let us consider the group of rank one $\Phi_1$ which is a quotient of $\Phi$ by its largest non trivial convex subgroup. By composition the weight $w$ gives rise to a weight $w_1$ with values in the rank one group $\Phi_1$. Let us denote by $p_1$ the ideal of elements of $S$ whose $w_1$ weight is $>0$. If the preimages in $S$ of an element $x$ have unbounded weights, there is a sequence of such pre-images whose $w_1$ weights tend to infinity, say $\tilde x_i\in \Qq_{\phi_1(i)}$ with $\phi_1(i)$ tending to infinity with $i$ in $\Phi_1$. Since the intersection $\bigcap_{\phi_1\in \Phi_{1+}}\Qq_{\phi_1}$ is zero and $S/F$ is a complete noetherian local ring, by Chevalley's theorem (see \cite{Te1}, section 5, and \cite{B3}, Chap. IV, \S 2, No. 5, Cor.4), we have a sequence of integers $T(\phi_1)$ tending to infinity with $\phi_1$ and such that $\Qq_{\phi_1(i)}\subset F+\tilde m^{T(\phi_1(i))}$ for all $i$, where $\tilde m$ is the maximal ideal of $S$. Therefore the images in $R$ of the elements of the sequence $\tilde x_i$ tend to $0$ in the $m$-adic topology which contradicts the fact that they are all equal to $x\neq 0$.\par\noindent  Therefore the inverse images in $S$ of a non zero element of $R$ have bounded weight. \par\medskip Let now $\nu(x)$ be the smallest element of the set $$\{\phi\in \Phi\vert w(\tilde x)\leq \phi\  {\rm if}\  \tilde x \ {\rm mod.} F =x\}.$$ This minimum exists since the weights of the elements of the noetherian ring $S$ form a well ordered subset of $\Phi$, and we have just proved that this set is not empty.\par
Let us prove by induction on the rank of $\Phi$ that given $x\in R$ there exists a pre-image $\tilde x\in S$ such that $w(\tilde x)=\nu(x)$: assume that the rank of $\Phi$ is one; then the number of elements of $\Phi_{\geq 0}$ that are less than $\nu(x)$ is finite and therefore so is the number of the weights of elements of $S$ that are $\leq\nu(x)$. So the maximum weight of those whose image in $R$ is $x$ is attained and then it has to be equal to $\nu(x)$. If the rank of $\Phi$ is $>1$, let $\Phi_{h-1}$ be the quotient of $\Phi$ by its smallest non zero convex subgroup. Let $w_{h-1}$ be the corresponding weight. By induction we may assume that there exists an element $\tilde x\in S$ whose weight is the image of $\nu(x)$ in $\Phi_{h-1}$. Let $w(\tilde x)\in \Phi$ be the weight of this element. Again the number of weights of elements of $S$ that are between $w(\tilde x)$ and $\nu(x)$ is finite since they have the same image in $\Phi_{h-1}$ (see the proof of Lemma 3 of \cite{Z-S}, Appendix 3, or \cite{Te1}, proposition 3.17), so the maximum is attained and then must be equal to $\nu(x)$.\par\medskip\noindent
This map $x\mapsto\nu(x)$ clearly satisfies the inequalities $\nu (x+y)\geq \hbox{\rm min}(\nu(x),\nu(y))$ and $\nu(xy)\geq \nu(x)+\nu(y)$.  Thus, we have defined an order function on $R$. To prove that it is a valuation is to prove that the second inequality is an equality, and we argue as follows:\par\noindent
The order function $\nu$ determines a filtration of $R$ by ideals just as a valuation does, and the associated graded ring is an integral domain if and only if the order function is a valuation. By construction the associated graded ring $\hbox{\rm gr}_\nu R$ of $R$ with respect to this filtration is a quotient of the associated graded ring $k[U_1,\ldots ,U_N]$ of $S$ with respect to the weight filtration. Indeed if we denote by $\Qq_\phi$ the ideal of elements of weight $\geq \phi$ in $S$, we see that by definition of $\nu$ it maps onto the ideal of elements of $R$ which are of order $\geq \phi$. \par
The ideal $F_0$ defining the quotient is the ideal of $k[U_1,\ldots ,U_N]$ generated by the initial forms of the elements of $F$ with respect to the weight filtration of $S$. Since by hypothesis this initial ideal is generated by the binomials $U^{m^\ell}-\lambda_\ell U^{n^\ell}$, the graded algebra $\hbox{\rm gr}_\nu R$ is equal to $k[U_1,\ldots ,U_N]/F_0$. It is therefore an integral domain, which shows that the order function $\nu$ is actually a valuation.
\end{proof}\noindent
\begin{Remark}\label{nofin}\small{\begin{enumerate} \item This applies in particular to the trivial overweight deformation: a positive weight which is such that the binomials correspond to generators of the relations between the $w(u_i)$ induces a rational valuation on the ring of the associated toric variety.
\item This proof does not use the fact that $S$ is noetherian, but only the fact that $R$ is and that the semigroup generated by the weights of the variables is well ordered. So corollary \ref{seeAb} holds also in that case. We shall use this below in the proof of proposition \ref{irreducible}. We will also need the following, which does not use the fact that $S$ is noetherian either:\end{enumerate}}\end{Remark}
\begin{proposition}\label{initial} With the notations of proposition \ref{resoverwght}, given $x\in R$ and a preimage $\tilde x\in S$ of $x$, the equality $\nu(x)=w(\tilde x)$ holds if and only if ${\rm in}_w\tilde x\notin F_0=\hbox{\rm  Kergr}_w\pi$.
\end{proposition}
\begin{proof} We have $w(\tilde x)<\nu(x)$ if and only if there exists $\hat x\in S$ with $\tilde x-\hat x \in F$ and $w(\hat x)>w(\tilde x)$. But then by construction ${\rm in}_w(\tilde x-\hat x)= {\rm in}_w(\tilde x)$ has to be in the binomial ideal $F_0={\rm in}_wF$. Conversely, if ${\rm in}_w\tilde x\in F_0$, it is the initial form of an element $\tilde x-\hat x\in F$ with $w(\hat x)>w(\tilde x)$.
\end{proof}

\begin{example}\label{ex1} Consider the binomial ideal generated by $u_2^2-u_1^3, u_3^2-u_1^5u_2$ in $S=k[[u_1,u_2,u_3]]$ with $w(u_1)=4,w(u_2)=6,w(u_3)=13$, and the overweight deformation $F_1=u_2^2-u_1^3-u_3, F_2=u_3^2-u_1^5u_2$. The image in $S/(F_1,F_2)$ of the element $u_2^2-u_1^3$ has a representative $u_2^2-u_1^3\in S$ of weight $12$ and another one, $u_3\in S$ which is of weight 13 and gives the maximum since its initial form is not in the binomial ideal.
\end{example}
\begin{definition} In the situation of proposition \ref{resoverwght}, we shall say that the pair $(R,\nu)$ is an overweight deformation of the ring $k[[u_1,\ldots ,u_N]]/((u^{m^\ell}-\lambda_\ell u^{n^\ell})_{1\leq \ell \leq s})$ equipped with the weight induced by $w$. By abuse of language we shall also say that it is an overweight deformation of its associated graded ring $\hbox{\rm gr}_\nu R$, of which the previous one is the completion.
\end{definition}\noindent
\begin{Remark}\label{ROD}\small{\begin{enumerate}\item In Theorem \ref{Valco} below, we shall use the concept of overweight deformation in a wider context, where we may have weights on infinitely many variables, and deform infinitely many binomial equations generating a prime ideal by adding to each one a series of terms of strictly greater weight in such a way that the $w$-initial ideal of the ideal generated by the deformed equations is equal to the original binomial ideal.
\item The notion of overweight deformation\footnote{I have just learnt that a similar definition of an order function on a quotient is used independently by Ebeling and Gusein-Zade in \cite{E-GZ} for the purpose of computing Poincar\'e series.} extends to any system of equations that are homogeneous with respect to the given weight. Whenever the ideal they generate is prime, the procedure above produces a valuation of the quotient ring of the power series ring by the ideal created by overweight deformation. This valuation is not rational in general. In this text we concentrate on the extension of a resolution of singularities and therefore make an appropriate assumption on the homogeneous ideal.
\end{enumerate}}
\end{Remark}
\par\smallskip\noindent
\textbf{Proof of b) and comments:}\par\noindent
We want to find a regular fan $\Sigma$ subdividing $\check\R_{\geq 0}^N$, compatible with the hyperplanes $H_\ell$, and a cone $\sigma\in\Sigma$ such that the strict transform of each $F_\ell$ by the corresponding monomial map $\pi^*(\sigma)\colon k[[u_1,\ldots ,u_N]]\to k[[y_1,\ldots ,y_N]]$ is a deformation of the strict transform of its initial form at the point of the strict transform $X'$ of our formal space $X$ determined by the valuation. As we shall see, this will ensure the nonsingularity of $X'$ at that point.\par\medskip\noindent
The idea to prove b) is very simple to explain in the case where the value group $\Phi$ (or the order on $\Z^r$) is of rank one and we consider the case of a single equation $$F=u^m-\lambda u^n+\sum_{w(p)>w(m)}c_pu^p.$$
Let $$E'=\langle\{p-n/c_p\neq 0\},m-n\rangle\subset \R^N,$$
where as above $\langle a,b,...\rangle$ denotes the cone generated by $a,b,...$. Since there may be infinitely many exponents $p$, the smallest closed convex cone containing $E'$ may not be rational. However, the power series ring being noetherian, there exist finitely many exponents $(p_f-n)_{f\in F}$ as above, with $F$ finite, such that $E'$ is contained in the rational cone $E$ generated by $m-n$, the vectors $(p_f-n)_{f\in F}$ and the basis vectors of $\R^N$. This cone is strictly convex because all its elements have a strictly positive weight except the positive multiples of $m-n$.\Par
Since the order is of rank one, we can fix an ordered embedding of $\Z^r$ in $\check\R$ and define\footnote{The dual is there to conform to toric tradition.} the \emph{weight vector} ${\mathbf w}=(w(u_1),\ldots ,w(u_N))\in \check\R^N$. The weight of a monomial $u^m$ is then the evaluation, or scalar product, $\langle {\mathbf w},m\rangle$. We note that, by construction, we have ${\mathbf w}\in\check E$.
\par\medskip\noindent
Given a regular cone $\sigma=\langle a^1,\ldots ,a^N\rangle \subset \check\R_{\geq 0}^N$, set $Z(\sigma)={\rm Spec}k[\check\sigma\cap M]$.\par\noindent
The map $Z(\sigma)\to \A^N(k)$ corresponding to the inclusion $k[\R^N_{\geq 0}\cap M]\subset k[\check\sigma\cap M]$ is monomial and birational, and we write it as:
$$u_i\longmapsto y_1^{a^1_i}\ldots y_N^{a^N_i},\ 1\leq i\leq N,$$
\noindent where the $y_j$ are generators of the polynomial algebra $k[\check\sigma\cap M]$ and the matrix with column vectors $a^k$ corresponds to the expression of the basis vectors of $\R^N$ in the basis given by the generating vectors of the cone $\check\sigma$.\Par
More precisely:\Par
To a regular fan $\Sigma$ with support $\check\R_{\geq 0}^N$ corresponds a proper and birational toric map of non singular toric varieties $\pi(\Sigma)\colon Z(\Sigma)\to \A^N(k)$ . To each cone of maximal dimension $\sigma=\langle a^1,\ldots , a^N \rangle$ corresponds a chart $Z(\sigma)$ of $Z(\Sigma)$ which is isomorphic to $\A^N(k)$. If we choose adapted coordinates $y_1,\ldots ,y_N$ in that chart, the restriction
$$\pi (\sigma)\colon Z(\sigma)\to \A^N(k)$$ is described by monomials as above and so for each monomial $u^m$ we have
 $$u^m\longmapsto
y_1^{\langle a^1,m\rangle}\cdots y_N^{\langle a^N,m\rangle},$$ where $\langle a^i,m\rangle =\sum_{j=1}^Na^i_jm_j$. \par\noindent
Remarking that the monomial map tells us that we have the equality of vectors
$${\mathbf w}=\sum_{i=1}^N w(y_i)a^i,$$ we see that the center of the valuation $\nu$ of $S$ determined by the weight $w$ is in $Z(\sigma)$ if and only if the weights $w(y_i)$, which are uniquely determined by the monomial map since the $a^j$ form a basis, are all $\geq 0$, which is equivalent to the condition that ${\mathbf w}\in \sigma$.\par The vector ${\mathbf w}$ is in the hyperplane $H$ of $\check\R^N$ dual to the vector $m-n$; this implies that the weight $w(y_i)$ is $>0$ only if $\langle a^i,m-n\rangle=0$.
It also implies in view of our overweight hypothesis that ${\mathbf w}$ lies in the interior of the intersection with $H$ of the convex dual $\check E$ of $E$, which is of dimension $N$ as we saw above. The intersection $\check E\cap H$ is of dimension $N-1$ because the only vector space which can be contained in the cone $E+\R(m-n)=\langle \{p-n/c_p\neq 0\}, m-n,n-m\rangle$ is $\R(m-n)$. By convex duality, this means that $\check E\cap H$, which is the convex dual of $E+\R(m-n)$ cannot be contained in a linear subspace smaller that $H$.\par\noindent
Recall that a fan $\Sigma$ is said to be compatible with a rational convex cone $C\subset\check\R^N$ if the intersection of any cone $\sigma\in \Sigma$ with $C$ is a face of $\sigma$ (which may be $\{0\}$ or $\sigma$).
So if $\Sigma$ is a regular subdivision of $\check\R_{\geq 0}^N$ which is compatible with $H$ and $\check E$ it will contain a regular cone $\sigma$ of dimension $N$ whose intersection with $H$ is of dimension $N-1$, which contains ${\mathbf w}$ and is contained in $\check E$. By the resolution theorem for normal toric varieties (see \cite{KKMS}, Chap. III or \cite{Ewald}, Chap. VI), since $\check E$ is a rational convex cone and $H$ a rational hyperplane, we know that there exist such regular fans.\par As a first step, let us examine the transforms in the charts $Z(\sigma)$ corresponding to cones $\sigma=\langle a^1,\ldots ,a^N\rangle$ which contain ${\mathbf w}$ and are compatible with $\check E$ and $H$.
\begin{itemize} \item \textit{Because our fan is compatible with $H$, the convex cone $\sigma$ has to be entirely on one side of $H$ and its intersection with $H$ is a face}. We may assume that
$a^1,\ldots ,a^t$ are those among the $a^j$ which lie in the hyperplane $H$ and all the other $\langle a^j,m-n\rangle$ are of the same sign, say
$\langle a^j,m-n\rangle>0$.
We have then
$$u^m-\lambda u^n\longmapsto y_1^{\langle a^1,n\rangle} \cdots y_N^{\langle a^N,n\rangle}(y_{t+1}^{
\langle
a^{t+1},m-n\rangle}\cdots y_N^{\langle a^N,m-n\rangle}-\lambda).$$
 \item \textit{By compatibility with $\check E$ and since it contains ${\mathbf w}$ which is in $\check E$}, the cone $\sigma$ is contained in $\check E\subseteq \check E'$ so that \emph{all $\langle a^i,p-n\rangle$ are $\geq 0$}. After perhaps re-subdividing $\sigma$ and choosing a smaller regular cone containing ${\mathbf w}$ and whose intersection with $H$ does not meet the boundary of $\check E$, we have that the $\langle a^i,p-n\rangle$ are $> 0$ at least for those $i$ such that $a^i\in H$.
\end{itemize}  In the corresponding chart $Z(\sigma)$ the transform of our equation $F$ by the monomial map can then be written:
$$ y_1^{\langle a^1, n\rangle}\ldots y_N^{\langle a^N, n\rangle}\bigl( y_{t+1}^{\langle a^{t+1}, m-n\rangle}\ldots y_N^{\langle a^N,m-n\rangle}-\lambda+\sum_pc_py_1^{\langle a^1,p- n\rangle}\ldots y_N^{\langle a^N,p- n\rangle}\bigr).$$
Since $\sum_{i=1}^N\langle a^i,p-n\rangle w(y_i)=\langle {\mathbf w},p-n\rangle$,  this shows that the strict transform $F'$ of $F$ by the monomial map $Z(\sigma)\to \A^N(k)$, which is the quantity between parenthesis, is an overweight deformation of the strict transform, of weight zero, of the initial part of $F$.\par This implies the result we seek since the hypersurface defined by the initial part of $F'$ is non singular because $\lambda\neq 0$ and the $\langle a^i,m-n\rangle$ which are $\neq 0$ cannot be all divisible by the characteristic of $k$. The reason\footnote{The generalization of this to prime binomial ideals (see \cite{Te1}, Lemma 6.3 and the proof of proposition \ref{sepex} below) will play an important role in what follows.} is that the $\langle a^i,m-n\rangle$ must generate the group $\Z$ since the $a^i$ form a basis of $\check\Z^N$ and $m-n$ is a primitive vector of $\Z^N$ because, by our assumptions, the field $k$ is algebraically closed and the ideal generated by $u^m-\lambda u^n$ is prime (see \cite{Ei-S}, Theorem 2.1, c)). So $m-n$ generates a (one dimensional) lattice of relations, the kernel of a linear map $\Z^N\to \Z^{N-1}$. In short, the vector of $\Z^N$ with coordinates $\langle a^i,m-n\rangle$ must be primitive, and this is equivalent to saying that $y_{t+1}^{
\langle
a^{t+1},m-n\rangle}\cdots y_N^{\langle a^N,m-n\rangle}-\lambda=0$ is non singular, whatever the characteristic of $k$ is. Let us show that the hypersurface $F'=0$ itself, being an overweight deformation of its initial part, has to be non singular at the point picked by the valuation.\par
In the charts we consider, not only does the valuation have a center since ${\mathbf w}\in\sigma$, but also the strict transform of our hypersurface by the map $Z(\sigma)\to \A^N(k)$ meets the toric boundary (= the complement of the torus) because some of the vectors $a^i$ are in $H$. The other charts are of no import as far as smoothness of the strict transform at the point picked by the valuation is concerned.\par
 The charts where the strict transform intersects the maximal number of components of the toric boundary are obtained by choosing the regular cone $\sigma\in\Sigma$ in such a way that its intersection with the hyperplane $H$ is of maximal dimension $N-1$, which means that $N-1$ of the vectors $a^i$ are in $H$. The $N-1$ corresponding coordinates $y_i$ will be of positive value and provide a system of local coordinates for the strict transform of our hypersurface at the point picked by the valuation. In fact, if $a^N$ is the vector which is not in $H$, according to what we saw above we must have $\langle a^N,m-n\rangle=1$ and our local equation becomes $$F'=y_N-\lambda+\sum_{w(p)>w(n)}c_py_1^{\langle a^1,p- n\rangle}\ldots y_N^{\langle a^N,p- n\rangle}.$$ Since the weight of $y_N$ is zero and since this is an overweight deformation, we see immediately that $F'$ is a power series in  $y_1,\dots ,y_{N-1}$ and $w_N=y_N-\lambda$ and the hypersurface $F'=0$ is non singular and transversal to the toric boundary at the point $y_1=\dots=y_{N-1}=0,y_N=\lambda$, with local coordinates $y_1,\dots ,y_{N-1}$. This point is the point picked by the valuation because on $F'=0$ the valuation of $y_N-\lambda$ has to be positive.\par\medskip\noindent
The proof of b) in the general case follows the same general line: consider the system of hyperplanes $H_\ell =H_{m^\ell-n^\ell}$ of $\check \R^N$ dual to the vectors $m^\ell-n^\ell$ and remember from (\cite{Te1}, 6.2)-or see below- that if $\Sigma$ is a regular fan subdividing the first quadrant of $\check \R^N$ and compatible with all the $H_\ell$, then the strict transform by the toric modification $\pi(\Sigma)\colon Z(\Sigma)\to \A^N(k)$ of $\A^N(k)$ determined by $\Sigma$ of the toric variety $X_0$ corresponding to the binomial ideal is non-singular and transversal to the toric boundary. In other words, the toric modification $Z(\Sigma)\to \A^N(k)$ gives an embedded pseudo-resolution of singularities of the toric variety $X_0$. If the fan $\Sigma$ contains the regular faces of the \emph{weight cone} $\check\R_{\geq 0}^N\cap W$ with $W=\bigcap_{\ell=1}^sH_\ell$, the pseudo-resolution is a resolution. Remember also that the charts of $Z$ where the strict transform of $X_0$ meets the toric boundary are those corresponding to cones $\sigma=\langle a^1,\ldots ,a^N\rangle$ which meet the weight cone outside of the origin, i.e., at least one of the $a^i$ is in $W$.\par\noindent

Furthermore, if the intersection of $\sigma$ with $W$ is $d$-dimensional, then the equations of the strict transform depend on exactly $N-d$ variables, say $y_{d+1},\ldots ,y_N$ and in view of their binomial nature their only solutions in the charts $Z(\sigma)$ are given by $y_{d+j}=c_{d+j}\in k^*$.
\par\medskip
We are going to show that one can choose regular fans refining such a fan $\Sigma$ so that the corresponding toric modification resolves the strict transform of $X$ at the point picked by the valuation $\nu$. The problem is to generalize the weight vector and the convex set $\check E$ which we used in the simple case above to determine which regular fans would be adapted to the overweight deformation and to show the existence of such fans.\par
 Let us examine how overweight deformations behave with respect to toric modifications.\Par
 Let $w$ be a weight on the variables $u_1,\ldots , u_N$ with values in a well ordered subsemigroup of the positive semigroup $\Phi_{\geq 0}$ of a totally ordered group $\Phi$ of finite rational rank. Let us say that a regular cone $\sigma=\langle a^1,\ldots , a^N \rangle\subset \check\R^N$ is $w$-\textit{centering} if the monomial valuation on $k(u_1,\ldots , u_N)$ determined by $w$ has a center\footnote{As I found recently, if we consider the additive preorder determined on the group $M$ of Laurent monomials in $u_1,\ldots ,u_N$ by the weight $w$, this notion is related to what G. Ewald and M. Ishida define in \cite{Ewald-Ishida} as the \textit{domination} of $\sigma$ by that preorder: their definition of domination chooses in a given toric variety the smallest torus-invariant affine open set in which the "center" of the preorder is visible as an orbit. See also \cite{GP-T2}.} in $k[\check\sigma\cap M]$. This means that in the monomial map determined by $\sigma$ as written above, we have that the $w(y_i)$ are $\geq 0$. \par\noindent
 We remark that since the matrix of the $a^j_i$ is unimodular, the weights of the $y_j$ are in $\Phi$ and uniquely determined by the $w(u_i)$. \par\noindent
 If the weight $w$ is of rank one and we identify its value group $\Phi$ with a subgroup of $\R$, we can consider the vector ${\mathbf w}=(w(u_1),\ldots ,w(u_N))\in \check\R_{\geq 0}^N$. Then the positivity of the $w(y_j)$ is equivalent to the fact that ${\mathbf w}$ is in $\sigma$ as we have already noted; a regular convex cone $\sigma$ is $w$-centering if and only if it contains the vector ${\mathbf w}$. \par If the rank ${\rm h}$ of $\Phi$ is greater than one, we consider the sequence of convex subgroups, with the convention that $\Psi_0=\Phi$:

 $$(0)=\Psi_h\subset \Psi_{h-1}\subset\ldots \Psi_1\subset \Phi,$$ and we notice that since $\Phi$ has no torsion and we are interested only in inequalities we can work in the divisible hull $\Phi\otimes_\Z\Q$ of $\Phi$ with the natural extension of the ordering on $\Phi$ and use the fact that it is the lexicographic product of groups of rank one:
 $$\Phi\otimes_\Z\Q=\Xi_1\times \ldots\times\Xi_h$$
 with $\Psi_j\otimes_\Z\Q=\{0\}\times\cdots\times\{0\}\times\Xi_{j+1}\times \ldots\times\Xi_h$.\par\medskip\noindent
  Let us denote by $r_j$ the rational rank of the group $\Xi_j$.\par\noindent
 Now let us choose an ordered embedding of $\Phi\otimes_\Z\Q$ in $(\R^h)_{\hbox{\rm lex}}$. For each $j,\ 1\leq j\leq h$, we can define a vector ${\mathbf w}(j)\in \check\R^N$; it is the vector whose coordinates are the projections in $\Xi_j$ of the $w(u_i),\ 1\leq i\leq N$, viewed as real numbers through the chosen embedding.\par\noindent
 \begin{lemma}\label{rat} Given $N$ elements $w_1,\ldots ,w_N$ of $\R$, the rational rank of the subgroup of  $\R$ generated by the $w_i$ is equal to the dimension of the smallest vector subspace $\langle w\rangle_\Q$ of $\R^N$ defined over $\Q$ containing the vector ${\mathbf w}=(w_1,\ldots ,w_N)$.
 \end{lemma}
 \begin{proof} Let $\Ll$ be the kernel of the $\Z$-linear map $b\colon \Z^N\to \R$ sending the $i$-th basis vector to $w_i$. The image of the map $b$ is the subgroup generated by the $w_i$ and by construction the rank of $\Ll$ is the number of independent $\Z$-linear forms vanishing on the vector ${\mathbf w}$, which is the codimension of $\langle w\rangle_\Q$ in $\R^N$.
 \end{proof}

 \begin{lemma}\label{linind} The vectors ${\mathbf w}(j),\ 1\leq j\leq h$, of $\check\R^N$ are linearly independent.
 \end{lemma}
 \begin{proof}  Consider the $\Z$-linear map $B\colon (\Z^N)^h\to \R^h$ which is the product for $1\leq k\leq h$ of the $h$ maps $b_k\colon \Z^N\to \R$ sending the basis vector $e_i^{(k)}$ of $\R^N$ to ${\mathbf w}(k)_i$ for $1\leq i\leq N$. Since each $\Xi_k\subset \R$ is generated by the ${\mathbf w}(k)_i,1\leq i\leq N $ the subgroup $\Phi=\Xi_1\times\cdots\times \Xi_h\subset \R^h$ is the image of this map.\par
 The kernel of $B$ is the lattice ${\mathcal M}$ in $(\Z^N)^h$ which is the product of the kernels of the maps $b_k$.\par The rank of the lattice ${\mathcal M}$ is $N-r_1+\cdots +N-r_h=hN-(r_1+\cdots +r_h)$. \par\noindent
 Remembering that for each $k,\ 1\leq k\leq h$ at least one of the elements $w(u_i)$ has a non zero image ${\mathbf w}(k)_i$ in $\Xi_k$, we see that the images by the map $B$ of the $N$ vectors $F_i=(e_i^{(1)},\ldots e_i^{(h)}),\ 1\leq i\leq N,$ of $(\Z^N)^h$ generate the $\R$-vector space $\R^h$.
 If there was a linear relation $\sum_{k=1}^hs_k{\mathbf w}(k)=0$, each image $B(F_i)=({\mathbf w}(1)_i,\ldots ,{\mathbf w}(h)_i)$  would lie in the hyperplane $\sum_{k=1}^hs_ky_k=0$ of $\R^h$. This contradiction ends the proof.\end{proof}
 For each $j,\ 1\leq j\leq h$, let us denote by $S_j$ the smallest vector subspace of $\check \R^N$ defined over $\Q$ and containing the ${\mathbf w}(k),\ 1\leq k\leq j$.\par\noindent
 Notice that all the vector spaces $S_j$ meet the first quadrant $\check \R_{\geq 0}^N$ outside of the origin since ${\mathbf w}(1)$ is in it. By the properties of the lexicographic order, the vector space $S_h$ meets the interior of $\check \R_{\geq 0}^N$, so that we have $\hbox{\rm dim}(S_h\cap \R_{> 0}^N)=\hbox{\rm dim}S_h$.\par

 \begin{lemma}\label{ratrank} a) For each $j,\  1\leq j\leq h$ the dimension of the vector space $S_j$ is $\sum_{k=1}^jr_j$.\par\noindent b) Let $\Sigma$ be a regular fan with support $\check\R_{\geq 0}^N$ which is compatible with the vector spaces $S_j$. There exist $N$-dimensional cones $\sigma$ of $\Sigma$ such that ${\mathbf w}(1)\in \sigma$ and for all $j$ the face $\sigma\cap S_j$ of $\sigma$ is a cone of maximal dimension in $S_j$.\par\noindent
 c) For such cones $\sigma$, for each $j$ the number of support hyperplanes of $\sigma$ which contain $S_{j-1}\cap\sigma$ but do not contain ${\mathbf w}(j)$ is equal to $r_j$.
 \end{lemma}
 \begin{proof} Assertion a) follows directly from Lemmas \ref{rat} and \ref{linind}.\par\noindent
 To prove b) it suffices to remark that since $\Sigma$ is compatible with the $S_j$, for each $j$ the intersection $S_j\cap \Sigma$ is a fan of $S_j\cap \check\R_{\geq 0}^N$, which is of the same dimension as $S_j$ as we saw above. Since the support of $\Sigma$ is $\check\R_{\geq 0}^N$ this intersection must contain cones of the maximal dimension, which are faces of cones of $\Sigma$. These cones have the required property. \par\noindent Assertion c) is a reformulation of b).
 \end{proof}
 Assume that $\sigma$ is a regular convex cone of dimension $N$ belonging to a regular fan with support $\R^N_{\geq 0}$ which is compatible with the rational vector spaces $S_j,\ 1\leq j\leq h$. Assume that ${\mathbf w}(1)\in \sigma$. Let us denote by $(L_s)_{1\leq s\leq N}$ the hyperplanes  bounding $\sigma$. For each $j$ there is a largest subset $I(j)\subset \{1,\ldots ,N\}$ such that $S_j\subseteq \bigcap_{s\in I(j)}L_s$. By convention we set $I(0)=\{1,\ldots ,N\}$.\par Let us denote by $L_s^{\geq 0}$ the closed half space of $\R^N$ determined by $L_s$ which contains $\sigma$.\par
 \begin{lemma}\label{centering}
 In this situation, the $N$-dimensional regular convex cone $\sigma$ is $w$-centering if and only if the following holds:\par\noindent
 For each $j,\ 0\leq j\leq h-1$, we have ${\mathbf w}(j+1)\in \bigcap_{s\in I(j)}L_s^{\geq 0}$.
 \end{lemma}
 \begin{proof} Since $\sigma$ is regular, the determinant of its generating vectors is $\pm 1$. According to the description $(*)$ of the monomial map associated to $\sigma$, the weights of the $u_i$ uniquely determine the weights of the $y_i$ in $\Phi$ since the determinant is $\neq 0$. Now writing that $w(y_i)$ is $\geq 0$ in the lexicographic product $\Xi_1\times\ldots\times \Xi_h$ reduces exactly to the expression given in the lemma. We observe that the projections in $\Xi_k$ of the valuations of the $(y_i)_{1\leq i\leq N}$ are the barycentric coordinates of the vector ${\mathbf w}(k)$ with respect to the generators of $\sigma$. If all the barycentric coordinates of ${\mathbf w}(1)$ are positive, then all the $w(y_i)$ are also positive and $\sigma$ is $w$-centering. If some of these barycentric coordinates are zero, it means that ${\mathbf w}(1)$ is in a face of $\sigma$ whose linear span is the intersection of the $L_s$ for $s\in I(1)$, by the definition of $I(1)$. Then, in order for the corresponding $w(y_j)$ to be non-negative in $\R^h$, it is necessary that the corresponding barycentric coordinates of ${\mathbf w}(2)$ are $\geq 0$, which is equivalent to the inclusion ${\mathbf w}(2)\in \bigcap_{s\in I(1)}L_s^{\geq 0}$, and so on. The proof of the converse statement is obtained in the same way.
 \end{proof}\noindent
\begin{Remark}\small{\begin{enumerate}\item If the group $\Phi$ is of rank one, the condition is simply that the vector ${\mathbf w}(1)$ is in $\sigma$, as we have noted above.
 \item The argument uses only the fact that $\sigma$ is simplicial and $N$-dimensional. \end{enumerate}}\end{Remark}
 \begin{lemma}\label{quadrant} Keep the notations introduced before Lemma \ref{rat}.\par Let ${\mathbf w}(1),{\mathbf w}(2),\ldots ,{\mathbf w}(h)$ be rationally independent vectors in an $r$-dimensional rational vector subspace $W\subset \R^N$, all lying in $W\cap \R_{\geq 0}^N$. Let $\sigma_0\subset \R_{\geq 0}^N$ be an $N$-dimensional regular cone of a fan $\Sigma$ supported in $\R_{\geq 0}^N$ which is compatible with $W$ and the vector spaces $S_j$ defined above and containing ${\mathbf w}(1)$. Assume that $\sum_{i=1}^h r_i=r$ and that $\sigma_0\cap W$ is of dimension $r$. Then there exists a regular $N$-dimensional cone $\sigma\subset \R_{\geq 0}^N$ of $\Sigma$ satisfying the conditions of Lemma \ref{centering} and whose intersection with $W$ is of dimension $r$.
 \end{lemma}
 \begin{proof}let $(L_s)_{s\in I}, I=\{1,\ldots ,N\}$ be the collection of the supporting hyperplanes of $\sigma_0$. By construction there is a largest subset $I(1)\subset I$ such that $S_1= \bigcap_{s\in I(1)}L_s$.  In view of Lemma \ref{ratrank}, c) and the compatibility of $\Sigma$ with $S_1$, it is of cardinality $N-r_1$. By Lemma \ref{linind} we know that ${\mathbf w}(2)$ is not in $\bigcap_{s\in I(1)}L_s$. Let us denote by $I(2)\subset I(1)$ the set $\{s\in I(1)\vert {\mathbf w}(2)\in L_s\}$. For each $s\in I(1)\setminus I(2)$, which is of cardinality $r_2$ by Lemma \ref{ratrank}, c),  we denote by $L_s^{\geq 0}$ the closed half space determined by $L_s$ which contains ${\mathbf w}(2)$. Again by Lemma \ref{linind} we know that ${\mathbf w}(3)\notin \bigcap_{s\in I(2)}L_s$ so we can define a subset $I(3)\subset I(2)$ by the condition that ${\mathbf w}(2)\in \bigcap_{s\in I(3)}L_s$ and a closed half space $L_s^{\geq 0}$ for each $s\in I(2)\setminus I(3)$, and so on. \par In the end we have built a sequence of subsets $$\{1,\ldots ,N\}\supset I(1)\supset I(2)\supset\cdots \supset I(h)$$ such that $S_t=\bigcap_{s\in I(t)}L_s$ and we have determined half spaces $L_s^{\geq 0}$ corresponding to all the hyperplanes $L_s$ for $s\in \{1,\ldots ,N\}\setminus I(h)$ in such a way that each ${\mathbf w}(k)$ is always in the half space $L^{\geq 0}_s$ if $L_s$ vanishes on ${\mathbf w}(k-1)$. According to Lemma \ref{ratrank}, c), at step $i$ we define $r_i$ half-spaces. Since $\sum_{i=1}^hr_i=r$, the set $I(h)$ is the set of those hyperplanes $L_s$ which contain $W$. Thus $\bigcap_{s\in \{1,\ldots ,N\}\setminus I(h)}L_s^{\geq 0}$ is a rational cone which is the intersection of $d$ half spaces in $\R^N$. If now we define for $s\in I(h)$ the half-space $L_s^{\geq 0}$ as the one containing $\sigma$, we see that $\sigma'=\bigcap_{s\in \{1,\ldots ,N\}}L_s^{\geq 0}$ is a rational regular cone which satisfies the conditions of Lemma \ref{centering}; by construction it satisfies the positivity conditions with respect to the ${\mathbf w}(k)$, and its intersection with $W$ is the intersection with $\R^N_{\geq 0}$ of the $r$ half-spaces $(L_s^{\geq 0}\cap W)_{s\in \{1,\ldots ,N\}\setminus I(h)}$ of $W$ and is therefore a regular $r$-dimensional cone of the fan $\Sigma\cap W$. \noindent
\end{proof}
\begin{remark}\small{\emph{ In the case where $h=1$ the cone $\sigma$ is equal to $\sigma_0$ but if $h>1$ and we transform $\sigma_0$ into the first quadrant by a unimodular transformation, then $\sigma$ becomes a possibly different quadrant.}} \end{remark}

  \begin{proposition}\label{exist}
 Keeping the same notations, let $\Sigma$ be a regular fan with support $\check \R_{\geq 0}^N$. Assume that it is compatible with the $S_j$ and the $H_\ell$. Then there is a cone $\sigma$ of dimension $N$ of $\Sigma$ which is $w$-centering and whose intersection with $W$ is of dimension $r$.  \par\noindent
  \end{proposition}
 \begin{proof} Let $\sigma_0$ be a cone of $\Sigma$ containing ${\mathbf w}(1)$ and whose intersection with $W$ is of dimension $d$. Piltant's theorem tells us that the dimension $d$ of the toric variety defined by the initial binomial ideal, which is also the dimension of the vector space $W$, is equal to the rational rank of the group $\Phi$:
 $$r_1+\cdots +r_h=\hbox{\rm rat.rk.} \Phi =r.$$
 It suffices now to apply Lemma \ref{quadrant}.
 \end{proof}
Given an overweight deformation as in $(OD)$ above, let us define for each $\ell$ indexing the binomial $u^{m^\ell}-\lambda_\ell u^{n^\ell}$ the following cones in $\R^N$:
 $$\begin{array}{lr}
E_\ell^{(1)}(j)=& \langle \{p-n^\ell/ \vert w (u^{p-n^\ell})\in \Psi_j\setminus\Psi_{j+1}\vert c_p^{(\ell)}\neq 0,\}, m^\ell-n^\ell\rangle\\ E_\ell^{(2)}(j)= &\langle \{p-m^\ell/\vert w(u^{p-m^\ell})\in \Psi_j\setminus\Psi_{j+1}\vert c_p^{(\ell)}\neq 0,\}, n^\ell-m^\ell\rangle\end{array}$$
  \begin{lemma}\label{strictcon} For all $\ell$ and $0\leq j\leq h-1$ the cone $E_\ell^{(1)}(j)$ (respectively $E_\ell^{(2)}(j)$) is contained in a strictly convex polyhedral rational cone whose elements satisfy  $\langle{\mathbf w}(k),q\rangle= 0$ for $1\leq k\leq j-1$ and $\langle{\mathbf w}(j),q\rangle\geq 0$, with $\langle{\mathbf w}(j),q\rangle=0$ if and only if $q$ is on the half-line generated by $m^\ell-n^\ell$ (respectively $n^\ell-m^\ell$).
\end{lemma}
\begin{proof} Since the ring $k[[u_1,\ldots , u_N]]$ is noetherian, for each $\ell$ the ideal generated by the monomials $u^p$ appearing in the $\ell$-th series is generated by finitely many of them, say $u^{p^\ell_1},\ldots u^{p^\ell_{s_\ell}}$. In view of the convexity of the subgroups $\Psi_j$ the cones $ E_\ell^{(1)}(j)$ and $E_\ell^{(2)}(j)$  are contained respectively in the convex cone generated by $m^\ell-n^\ell$ and the $p^\ell_k-n^\ell+\R_{\geq 0}^N$ and in the convex cone generated by $n^\ell-m^\ell$ and the $p^\ell_k-m^\ell+\R_{\geq 0}^N$  (for $1\leq k\leq s_\ell$). These cones are rational since the $p^\ell_k$ are finite in  number and they are strictly convex since they can be defined using strict inequalities and thus cannot contain a vector subspace. The second part of the statement follows from the definition of the vectors ${\mathbf w}(k)$.
\end{proof}
 Since what we want in the end is to find regular convex cones contained in the convex duals $ \check E_\ell^{(i)}(j)$ of the cones $ E_\ell^{(i)}(j)$, we may in view of this lemma assume that the cones $ E_\ell^{(i)}(j)$ themselves are polyhedral rational strictly convex cones, which we shall do henceforth.
\begin{lemma}\label{position}Still denoting by $H_\ell$ the hyperplane of $\check\R^N$ dual to $m^\ell-n^\ell$ and by $W$ the intersection of the $H_\ell$, for each $j$ and each $\ell$ we have :
\begin{itemize}
\item The cones $\check E_\ell^{(1)}(j)$ and $\check E_\ell^{(2)}(j)$ are $N$-dimensional, and their intersection\break $\check E_\ell^{(1)}(j)\cap\check E_\ell^{(2)}(j)$ is equal to $\check E_\ell^{(1)}(j)\cap H_\ell=\check E_\ell^{(2)}(j)\cap H_\ell$.
\item For $i=1,2$ the dimension of $\check E_\ell^{(i)}(j)\cap H_\ell$ is $N-1$.
\item The interior in $H_\ell$ of $\check E_\ell^{(1)}(j)\cap\check E_\ell^{(2)}(j)$ is contained in the interior of\break $\check E_\ell^{(1)}(j)\bigcup\check E_\ell^{(2)}(j)$.
\item For each $k$ the cone $\R {\mathbf w}(1)+\cdots +\R {\mathbf w}(k-1)+\R_{\geq 0}{\mathbf w}(k)$ is contained in $\check E_\ell^{(1)}(k)\bigcap\check E_\ell^{(2)}(k)$ and meets its relative interior in $H_\ell$.
\item The same statements are true if one replaces each $\check E_\ell^{(i)}(k)$ by $\bigcap_\ell \check E_\ell^{(i)}(k)$ and $H_\ell$ by $W$.
\end{itemize}

\end{lemma}
\begin{proof} The dimensionality statement is nothing but the fact that the $ E_\ell^{(i)}(j)$ are strictly convex. An element $a\in \check\R^N$ which is in $\check E_\ell^{(1)}(j)\cap\check E_\ell^{(2)}(j)$ has to be both $\geq 0$ and $\leq 0$ on $m^\ell-n^\ell$, so it is in $H_\ell$. Since $p-n^\ell=p-m^\ell +m^\ell-n^\ell$ an element of $H_\ell$ which is $\geq 0$ on $ E_\ell^{(1)}(j)$ is $\geq 0$ on $ E_\ell^{(2)}(j)$ and conversely.\par\noindent
The second statement follows by convex duality from Lemma \ref{strictcon} which implies that $\R \langle m^\ell-n^\ell\rangle$ is the largest vector space contained in $E_\ell^{(i)}(j)+\R \langle m^\ell-n^\ell\rangle$. \par\noindent
The third statement is true because we join two convex cones along a common face of codimension 1; the boundary of the union does not meet the interior of the face.\par\noindent
The fourth statement also follows by convex duality from Lemma \ref{strictcon} if one observes that convex duality, which reverses inclusions, transforms intersection into Minkowski sum. Furthermore, if we add to a vector $a{\mathbf w}(k)$ with $a>0$ a vector $\epsilon {\mathbf b}$ with $b\in H_\ell$, for $\vert \epsilon\vert$ small enough the vector $a{\mathbf w}(k)+\epsilon {\mathbf b}$ will still take positive values on the $p^\ell_k-n^\ell$, where the $p^\ell_k$ are defined in the proof of Lemma \ref{strictcon}, and therefore will belong to $\check E_\ell^{(1)}(k)\cap\check E_\ell^{(2)}(k)$.\par\noindent
Finally, the same arguments apply to $W$.
\end{proof}
\begin{proposition}\label{approx}
Given an overweight deformation as in $(OD)$,
there exist regular fans $\Sigma$ with support $\check\R_{\geq 0}^N$ compatible with the hyperplanes $H_\ell$, the vector spaces $S_j$ and all the cones $\check E^{(c)}_\ell(k), c=1,2,\ k=1,\ldots ,h$.\par\noindent
There exist such fans which contain  $w$-centering regular cones $\sigma =\langle a^1,\ldots,a^N\rangle$  contained for all $\ell$ and $k$ in one of the two cones $\check E_\ell^{(c)}(k)\ \ c=1,2$ and such that none of the vectors $a^i$ which are in $H_\ell$ is in the boundary. For such a cone we have $\langle a^i,p-m^\ell\rangle> 0$ if $\langle a^i, n^\ell-m^\ell\rangle\geq 0$ (resp. $\langle a^i,p-n^\ell\rangle> 0$ if $\langle a^i, m^\ell-n^\ell\rangle\geq 0$) for all monomials $p$ with $w(p)>w(m^\ell)=w(n^\ell)$ appearing in the overweight deformation.
\end{proposition} \noindent\textit{proof}:
By the resolution theorem for normal toric varieties (see \cite{KKMS}, Chap. III or \cite{Ewald}, Chap. VI) we know that there exist regular fans $\Sigma$ with support $\check\R_{\geq 0}^N$ compatible with the $H_\ell$, the $S_j$ and the $\check E^{(c)}_\ell(k)$, all of which determine rational cones in $\check\R_{\geq 0}^N$. According to Lemma  \ref{exist} such fans contain $w$-centering cones. Let us show that such a cone is contained in every $\check E^{(1)}_\ell(k)\bigcup \check E^{(2)}_\ell(k)$. In view of the compatibility, it suffices to show that $\sigma$ meets the interior of $\check E^{(1)}_\ell(k)\bigcup \check E^{(2)}_\ell(k)$. To prove that, in view of Lemma \ref{position} it is enough to check that $\sigma$ meets the interior of $\check E^{(1)}_\ell(k)\cap H_\ell$ in $H_\ell$.\par\noindent
We see that by construction the cone $\sigma$ contains points of the cone $\R {\mathbf w}(1)+\cdots +\R {\mathbf w}(k-1)+\R_{\geq 0}{\mathbf w}(k)$ which are in the interior of $\check E^{(1)}_\ell (k)\cap W$ in $W$. Thus the cone $\sigma$ meets the interior of each $\check E^{(1)}_\ell (k)\cap \check E^{(2)}_\ell (k)$ and by compatibility of $T$ with the $\check E^{(c)}_\ell (k)$ it is contained in the union $\check E^{(1)}_\ell (k)\bigcup \check E^{(2)}_\ell (k)$. By compatibility with the $H_\ell$ the cone $\sigma$ has to be entirely on one side of $H_\ell$, which means that it must be in
$\check E^{(1)}_\ell (k)$ or $\check E^{(2)}_\ell (k)$. But this is decided for each $\ell$ by the fact that one generating vector of $\sigma$ is on one side of $H_\ell$. Since $\sigma$ meets the interior of each $\check E^{(1)}_\ell (k)\cap \check E^{(2)}_\ell (k)$, after perhaps refining the fan into another regular fan we may assume that none of the vectors $a^i$ which are in $H_\ell$ is in the boundary of $\check E^{(1)}_\ell (k)\bigcup \check E^{(2)}_\ell (k)$.\par
If we write $\sigma=\langle a^1,\ldots ,a^N\rangle$ we see that it has the property that whenever for a given $\ell$ we have that $\langle a^i,m^\ell -n^\ell\rangle$ is $> 0$ for some $i$ then for all $j,\ 1\leq j\leq N$ we have $\langle a^j,m^\ell -n^\ell\rangle\geq 0$, and if $a^j\in W$ we have $\langle a^j,p-n^\ell\rangle > 0$ for all $p$ appearing in the $\ell$-th equation, and if $\langle a^i,n^\ell -m^\ell\rangle$ is $> 0$ for some $i$, then for all $j$ we have $\langle a^j,n^\ell -m^\ell\rangle\geq 0$ and  if $a^j\in W$ we have $\langle a^j,p-m^\ell\rangle> 0$ for those $p$. For each given $\ell$ there has to be an index $i$ for which $\langle a^i,m^\ell -n^\ell\rangle\neq 0$.\qed   \par\noindent
\emph{Let us now finish the proof of proposition \ref{resoverwght}}. Take a regular fan $T$ with support $\check\R^N_{\geq 0}$ and compatible with the $H_\ell$, the $S_j$ and the $\check E_\ell^{(i)}\ \ i=1,2$ (and so depending on the deformation), and a $w$-centering cone $\sigma=\langle a^1,\ldots, a^N\rangle$ of that fan as above. Let us write the transforms of the equations $F_1,\ldots, F_s$, with the convention that $y^{\langle a, m\rangle}=y_1^{\langle a^1, m\rangle}\ldots y_N^{\langle a^N, m\rangle}$.
 $$\begin{array}{lr}
\tilde F_1=y^{\langle a, m^1\rangle}-\lambda_1 y^{\langle a, n^1\rangle}+\Sigma_{w(p)>w(m^1)} c^{(1)}_py^{\langle a, p\rangle}\\
\tilde F_2=y^{\langle a, m^2\rangle}-\lambda_2 y^{\langle a, n^2\rangle}+\Sigma_{w(p)>w(m^2)} c^{(2)}_py^{\langle a, p\rangle}\\
.....\\
\tilde F_\ell=y^{\langle a, m^\ell\rangle}-\lambda_\ell y^{\langle a, n^\ell\rangle}+\Sigma_{w(p)>w(m^\ell)} c^{(\ell)}_py^{\langle a, p\rangle}\\
.....\\
\tilde  F_s=y^{\langle a, m^s\rangle}-\lambda_s y^{\langle a, n^s\rangle}+\Sigma_{w(p)>w(m^s)} c^{(s)}_py^{\langle a, p\rangle}\\
\end{array}
$$
Thanks to the properties of our cone $\sigma$ we may factor out of each $\tilde F_\ell$ either $y^{\langle a, m^\ell\rangle}$ or $y^{\langle a, n^\ell\rangle}$. This leaves us with a deformation of the strict transform of the toric variety, which is regular in the chart corresponding to $\sigma$. More precisely, if for convenience of notation we rearrange the binomials in such a way that all $\langle a^i,m^\ell-n^\ell\rangle$ are $\geq 0$, by writing $n^\ell -\lambda_\ell ^{-1}m^\ell$ if $ \langle a^k, m^\ell -n^\ell\rangle<0$ we can write after the corresponding change of notations
$$
\begin{array}{lr}
\tilde F_1=y^{\langle a,n^1\rangle} \bigl(y^{\langle a, m^1-n^1\rangle}-\lambda_1 +\Sigma_{w(p-n^1)>0} c^{(1)}_py^{\langle a, p-n^1\rangle}\bigr)\\
\tilde F_2=y^{\langle a,n^2\rangle} \bigl(y^{\langle a, m^2-n^2\rangle}-\lambda_2 +\Sigma_{w(p-n^2)>0} c^{(2)}_py^{\langle a, p-n^2\rangle}\bigr)\\
.....\\
\tilde F_\ell=y^{\langle a,n^\ell\rangle} \bigl(y^{\langle a, m^\ell-n^\ell\rangle}-\lambda_\ell +\Sigma_{w(p-n^\ell)>0} c^{(\ell)}_py^{\langle a, p-n^\ell\rangle}\bigr)\\
.....\\
\tilde  F_s=y^{\langle a,n^s\rangle} \bigl(y^{\langle a, m^s-n^s\rangle}-\lambda_s +\Sigma_{w(p-n^s)>0} c^{(s)}_py^{\langle a, p-n^s\rangle}\bigr).\\
\end{array}
$$

Since our binomial ideal is prime and $k$ is algebraically closed, by \cite{Ei-S}, Theorem 2.1, c), we know that the lattice $\Ll$ which is the kernel of the map $b\colon \Z^N\to \Z^r$ associated to our affine toric variety is saturated. The vectors $m^\ell-n^\ell$ are primitive as generators of a saturated lattice. By the argument we saw above in the case of a single equation, each of the "strict transform" series $$F'_\ell=y^{-\langle a,n^\ell\rangle}\tilde F_\ell=y^{\langle a, m^\ell-n^\ell\rangle}-\lambda_\ell +\Sigma_{w(p-n^\ell)>0} c^{(\ell)}_py^{\langle a, p-n^\ell\rangle},\ 1\leq \ell\leq s $$ defines a non singular hypersurface transversal to the toric boundary at the point picked by the valuation. Taken together, they define an overweight deformation of an irreducible binomial variety $X'_0$ corresponding to a prime binomial ideal generated by binomials of weight zero. In fact, by theorem 2.1 of \cite{Ei-S}, the strict transform $X'_0$ of $X_0$ is a reduced complete intersection, and its equations are the strict transforms of binomials corresponding to $N-r$ generators of the lattice $\Ll$.  Let us recall the proof of its non singularity at the point picked by the valuation, following proposition 6.2 of \cite{Te1}:\par\noindent
According to \textit{loc.cit.}, Lemma 6.3 (see also the proof of proposition \ref{sepex} below), since the lattice $\Ll$ is saturated the $(N-r)\times (N-r)$ minors of the matrix $(\langle a^j, m^\ell-n^\ell\rangle);\ 1\leq j\leq N,\ 1\leq \ell\leq s$ are not all zero and have no common factor. Up to renumbering the equations corresponding to the $N-r$ generators of $\Ll$ are are $F_{s-N+r+1},\ldots ,F_s$. All the $(N-r)\times (N-r)$ minors mentioned above are linear combinations of those corresponding to the last $N-r$ equations. Therefore, some of the the $(N-r)\times (N-r)$ minors of the matrix $A=(\langle a^j, m^\ell-n^\ell\rangle;\ 1\leq j\leq N,\ s-N+r+1\leq \ell\leq s$) are not divisible by the characteristic of the base field. Moreover, the $r$ variables $y_i$ corresponding to indices $i$ such that $a^i\in\bigcap_\ell H_\ell$ are the only ones with positive weight. We can number them $y_1,\ldots ,y_r$.\Par
For the same reason as in the case of one equation, the point picked by the valuation is the unique solution of the system of equations
$$
\begin{array}{lr}
y^{\langle a, m^{s-N+r+1}-n^{s-N+r+1}\rangle}-\lambda_{s-N+r+1} =\ldots =y^{\langle a, m^s-n^s\rangle}-\lambda_s =0\\
y_1=\cdots=y_r=0$$
\end{array}
$$ The first set of equations does not involve the variables $y_1,\ldots ,y_r$. The intersction point is unique because it is the transversal intersection in the chart $Z(\sigma)$ of a stratum of codimension $r$ of the toric boundary with an $r$-dimensional orbit of the torus action corresponding to the saturated lattice $\Ll$, a torus action which is trivial on the affine space with coordinates $y_{N-r+1},\ldots ,y_N$. \Par The non vanishing of the image in the base field of one of the $(N-r)\times (N-r)$ minors of $A$ implies the non vanishing at the point picked by the valuation of the corresponding jacobian minor of the equations defining $X'_0$ (see \cite{Te1}, 6.2). \Par This in turn shows that $X'_0$ is non singular at the point picked by the valuation and is transversal to the toric boundary at this point. Since the strict transform $X'$ of $X$ is an overweight deformation of $X'_0$ this non singularity extends to it as the consideration of the same jacobian minor for the deformed equations shows. Indeed, the jacobian minor of the deformed equations has the nonzero jacobian minor of binomials as its initial form, and therefore cannot vanish when the coordinates $y_j$ with $w(y_j)>0$ vanish. The reader may also look at a closely related argument in the proof of proposition \ref{sepex} below, and the paragraphs following it, especially the remark \ref{tame} and equation $(Jac)$, using the fact that the variables $U_k$ in equation $(Jac)$ correspond in our situation to the $y_i$ such that $a^i\notin\bigcap_\ell H_\ell$, so that $w(y_i)=0$.  \par Consider the dual $\check b\colon\check\Z^r\to \check\Z^N$ of the surjective map $b\colon\Z^N\to\Z^r$ obtained from the map $\N^N\to\Z^r$ describing the generators of the semigroup $\Gamma$. The image of the transform of the injection $\check b$ by the map $\check\Z^N\to \check\Z^N$ corresponding to the transpose of the unimodular matrix of the vectors $a^1,\ldots ,a^N$ is generated by the $r$ vectors $v_i=(\nu(y_1)_i,\ldots ,\nu(y_N)_i)$. Therefore the valuations of the $y_i$ such that $\nu(y_i)\neq 0$ must be rationally independent.
This end the proof of proposition \ref{approx} and therefore of proposition \ref{resoverwght}.\qed

And with the last paragraph we have also proved:
\begin{proposition} \label{quasim} At the point of the strict transform $X'$ of $X$ picked by the valuation $\nu$, the $r$ variables $y_i$ such that $\langle a^i, m^\ell-n^\ell\rangle =0$ for $1\leq \ell\leq L$ form a system of local coordinates and their valuations are rationally independent.\hfill\qedsymbol
\end{proposition}
\begin{remark} \small{\emph{We know from \cite{Ei-S} that the intersection of an affine toric subvariety of $\A^N(k)$ with the ambient torus is a complete intersection and from \cite{CCD} that its ideal in $k[U_1,\ldots ,U_N]$ in general does not even contain a regular sequence of binomials of length equal to its codimension. In the proof we have just given, the strict transform of our toric variety, in the chart containing the center of the valuation, is the product with an $r$-dimensional affine space of its intersection with the torus of an invariant $N-r$- dimensional affine space, and this intersection is the center of the valuation, which can be defined by the strict transforms of $N-r$ binomials as we saw above.}}
\end{remark}
\section{The valuative Cohen Theorem}\label{cohen}
If one gives it its full strength, in our equicharacteristic framework the classical Cohen structure theorem (see \cite{B1}) can be deemed to present any complete noetherian local ring $R$ as the ring of a formal subspace of an affine space over its residue field, and also to allow the encoding by equations of the specialization of this formal subspace to its tangent cone, which lives in an affine space of the same dimension.\par The purpose of the valuative Cohen theorem is to provide an analogous affine embedding (possibly in an infinite-dimensional space) of the formal space $(X,0)$ corresponding to $R$ such that $(X,0)$ can be specialized, within that affine ambient space, to its generalized tangent cone corresponding to ${\rm gr}_\nu R$, even though this algebra may not be finitely generated. The specialization can then be encoded by a (possibly infinite) system of equations corresponding to an overweight deformation of a prime binomial ideal. This is a necessary step if one wants to use (partial) toric embedded pseudo-resolutions of the toric variety ${\rm Spec}{\rm gr}_\nu R$ to obtain local uniformizations of $\nu$ on $R$.\Par For the convenience of the reader, we revisit the statement, referring to \cite{Te1}, \S 5, especially 5.3, for other developments, in particular the equational description of the faithfully flat specialization of $R$ to (the completion of) ${\rm gr}_\nu R$. \par\medskip
Given a field $k$ and a totally ordered abelian group $\Phi$, the Hahn ring of $\Phi_{\geq 0}$ with coefficients in $k$ is the $k$-vector space of all formal power series $ \sum_{\phi\in\Phi_{\geq 0}} c_\phi t^\phi$ with $c_\phi\in k$ and exponents in $\Phi_{\geq 0}$ which satisfy the condition that the set $\{\phi\in\Phi_{\geq 0}\vert c_\phi\neq 0\}$ is well ordered. This condition implies that we can multiply two such series, and this multiplication gives our $k$-vector space the structure of a $k$-algebra, which we denote by $k[[t^{\Phi_{\geq 0}}]]$. It is an interesting completion of the semigroup algebra $k[t^{\Phi_{\geq 0}}]$ with coefficients in $k$, which plays a significant role in valuation theory; it is the valuation ring of a maximal valued field $k[[t^\Phi]]$ \textit{\`a la} Krull (\cite{Kr}).\par
Consider the semigroup $\Gamma$ of values taken on a noetherian local domain $R$ by a valuation with value group $\Phi$ whose ring dominates $R$. Since the semigroup $\Gamma$ is well ordered, the ring of series with coefficients in $k$ and exponents in $\Gamma$ is a subalgebra $k[[t^\Gamma]]\subset k[[t^{\Phi_{\geq 0}}]]$.\par\medskip
Let $(u_i)_{i\in I}$ be variables indexed by the elements of the minimal system of generators $(\gamma_i)_{i\in I}$ of the semigroup $\Gamma$. Give each $u_i$ the weight $w(u_i)=\gamma_i$ and let us consider the set of power series $\sum_{e\in E} d_e u^e$ where $(u^e)_{e\in E}$ is any set of monomials in the variables $u_i$ and $d_e\in k$. By a theorem of Campillo-Galindo (see \cite{C-G}, \S 2), the semigroup $\Gamma$ being well ordered is combinatorially finite, which means that for any $\phi\in\Gamma$ the number of different ways of writing $\phi$ as a sum of elements of $\Gamma$ is finite. This is equivalent to the fact that the set of exponents $e$ such that $w(u^e)=\phi$ is finite: for any given series the map $w\colon E\to \Gamma,\ e\mapsto w(u^e)$ has finite fibers. Each of these fibers is a finite set of monomials in variables indexed by a totally ordered set, and so can be given the lexicographical order and order-embedded into an interval $1\leq i\leq n$ of $\N$. This defines an injection of the set $E$ into $\Gamma\times\N$ equipped with the lexicographical order and thus induces a total order on $E$, for which it is well ordered. When $E$ is the set of all monomials, this gives a total monomial order. \par  Just as for $k[[t^\Gamma]]$, the combinatorial finiteness implies that this set of series is a $k$-algebra, which we denote  by $\widehat{k[(u_i)_{i\in I}]}$. Since the weights of the elements of a series form a well ordered set and only a finite number of terms of the series have minimum weight, the associated graded ring of $\widehat{k[(u_i)_{i\in I}]}$ with respect to the filtration by weights is the polynomial ring $k[(U_i)_{i\in I}]$. The $k$-algebra $\widehat{k[(u_i)_{i\in I}]}$ is endowed with a topology coming from the same filtration, for which it is complete in a sense we shall see below. The combinatorial finiteness also implies that we can rearrange the terms of a series, for example\footnote{This makes the link with the construction of \cite{Te1}.} (using the notations introduced in the proof of b) of proposition \ref{resoverwght}) in order to write it $\sum_AD_Au^A$ where each monomial $u^A$ involves only variables of weight belonging to $\Phi\setminus\Psi_1$ and the $D_A$ belong to $\widehat{ k[(u_i)_{\{i\vert\gamma_i\in\Psi_1\}}]}$. The subalgebra $k[(u_i)_{i\in I}]\subset \widehat{k[(u_i)_{i\in I}]}$ corresponds to series indexed by a finite set, and it is dense for this topology in the scalewise sense of \cite{Te1}.\par
The application $u_i\mapsto t^{\gamma_i}$ defines a continuous and surjective map of topological $k$-algebras
$$\widehat{k[(u_i)_{i\in I}]}\longrightarrow k[[t^\Gamma]],\eqno{(N)}$$
which we can think of as corresponding to the natural affine embedding of the "formal toric variety" (possibly of infinite embedding dimension) associated to the semigroup $\Gamma$ and its ordering (which determines the minimal system of generators $\gamma_i$). Its kernel is the closure of the prime binomial ideal $(u^{m^\ell}-u^{n^\ell})_{\ell\in L}$ encoding the relations between the generators of $\Gamma$. The associated graded map is the usual presentation $k[(U_i)_{i\in I}]\longrightarrow k[t^\Gamma]$, $U_i\mapsto t^{\gamma_i}$, of the semigroup algebra. We can call the surjection $(N)$ the canonical presentation of the Hahn ring of $\Gamma$ over $k$ and call the ring $\widehat{k[(u_i)_{i\in I}]}$ the \emph{scalewise $w$-completion} (or just scalewise completion when there is no ambiguity) of the polynomial ring. Here the term scalewise refers to the transfinite character of the series rather that to their mode of construction.\Par For a sum of elements of $\widehat{k[(u_i)_{i\in I}]}$ to belong to that ring, it is necessary and sufficient that each monomial $u^a$ appears at most finitely many times in the sum.
\begin{proposition} Let $\Gamma =\langle(\gamma_i)_{i\in I}\rangle$ be the semigroup of values taken on a noetherian local domain $R$ by a valuation. The scalewise completion $S=\widehat{k[(u_i)_{i\in I}]}$ of the corresponding polynomial ring with respect to the weight $w(u_i)=\gamma_i$ is a formally smooth $k$-algebra.
\end{proposition}\begin{proof} This follows from (\cite{B1}, \S 7, No. 7, Prop. 7), applied to $k_0=k$ and $A=S$. Since there is no relation between the variables $u_i$, the graded ring ${\rm gr}_mS$ is equal to $k[({\rm in}_m u_i)_{i\in I}]$, which is the symmetric algebra over $k$ of $m/m^2$, and the second condition of the proposition is satisfied because $\Omega_{k_0/k_0}=(0)$.
\end{proof}
A (closed) \emph{ball} in $\widehat{k[(u_i)_{i\in I}]}$ equipped with the weight $w$ is a subset of the form $B(x,\gamma)=\{y\vert w(y-x)\geq \gamma\}$. Following F-V. Kuhlmann in (\cite{Ku2}, \S 1), with the ultrametric\footnote{This is an abuse of language, adapted to general valuation theory. Assuming that the valuation is of rank one, the ultrametric in the usual sense is of course $e^{-u(x,y)}$.}  $u(x,y)=w(y-x)$, we say that $\widehat{k[(u_i)_{i\in I}]}$ is \emph{spherically complete} with respect to the weight $w$ if every nested sequence of non empty balls has a non-empty intersection. As in \emph{loc.cit.}, it is convenient to denote by $B(x,y)$ the smallest closed ball $B(x,w(y-x))$ with center $x$ containing $y$, which is also the smallest closed ball containing $x$ and $y$.\par\medskip\noindent
The following result will enable us, in sequels to this paper, to use various forms of the implicit function theorem, and in particular those due to F.-V. Kuhlmann in \cite{Ku2}.
\begin{theorem}\label{sph} The ring $\widehat{k[(u_i)_{i\in I}]}$ is spherically complete with respect to the ultrametric $u(x,y)=w(y-x)$.
\end{theorem}
\begin{proof} Let $B(x_\iota,y_\iota),\ \iota\in H$, be a nested sequence of balls indexed by a well ordered set $H$. We denote by $1$ the smallest element of the set $H$, by $\iota +1$ the successor of $\iota$. The inclusions $B(x_{\iota +1},y_{\iota +1})\subseteq B(x_\iota,y_\iota)$ mean that $w(x_{\iota +1}-x_\iota)\geq w(y_\iota -x_\iota)$ and $w(y_{\iota +1}-x_{\iota +1})\geq w(y_\iota -x_\iota)$. Since $I$ is well ordered, by choosing an appropriate subset we may assume that the inclusions are strict without changing the intersection. The ordinal $H$ is thus order-embedded in $\Gamma$ because the radii of the balls must decrease ($w(y_{\iota +1}-x_{\iota +1})> w(y_\iota -x_\iota)$), and we may assume that the $x_\iota$ and $y_\iota$ are chosen in such a way that $x_{\iota +1}\in B(x_\iota,y_\iota)$ and $y_\iota\notin B(x_{\iota +1},y_{\iota +1})$. Then we have the inequalities $w(x_{\iota +1}-x_\iota)\geq w(y_\iota-x_\iota)$ and $w(y_\iota-x_{\iota +1})<w(y_{\iota +1} -x_{\iota +1})$. From this we deduce that $w(y_{\iota +1}-y_\iota)=w(y_\iota-x_{\iota +1})$. The inclusion $x_{\iota +1}\in B(x_\iota,y_\iota)$, gives us $w(y_\iota-x_{\iota +1})\geq w(y_\iota-x_\iota)$. Finally we have
$$w(y_{\iota +1}-x_{\iota +1})> w(y_{\iota +1}-y_\iota)\geq w(y_\iota-x_\iota).$$
These inequalities imply that $w(y_{\iota +1}-y_\iota)$ increases strictly with $\iota$. Thus, if the rank of $\Phi$ is one, the sum $y=y_1+\sum_{\iota\geq 1}(y_{\iota +1}-y_\iota)$ is an element of $\widehat{k[(u_i)_{i\in I}]}$ and we have  $y\in B(x_\iota,y_\iota)$ for all $\iota\in I$ since $y-y_\iota =\sum_{\kappa\geq \iota}(y_{\kappa +1}-y_\kappa)$.\Par
Here we have used the fact that for the sum to make sense it suffices that the set of elements of a given weight which appear in it should be finite, and this is true because the ordinal of $\Gamma$ in this case is $\omega$ (see \cite{Te1}, proposition 3.9).\par
 For higher ranks, we must avoid the possibility that infinitely many of the $y_{\iota +1}-y_\iota$ have some terms in common.\Par Let us write $y_\iota =\sum_ed^{(\iota)}_eu^e$ for each $\iota\in H$ and define $\tilde y_\iota =\sum_{e\vert w(e)\leq w(y_{\iota +1}-y_\iota)}d^{(\iota)}_eu^e$. We have by construction $w(\tilde y_{\iota +1}-\tilde y_\iota)=w(y_{\iota +1}-y_\iota)$ and $w(\tilde y_\iota-y_\iota)>w(y_{\iota +1}-y_\iota)\geq w(y_\iota -x_\iota)$, and therefore $w(\tilde y_\iota-x_\iota) =w(y_\iota-x_\iota)$. Moreover, we have $w(\tilde y_\iota -x_{\iota +1})=w( \tilde y_\iota -y _\iota +y_\iota -x_{\iota +1}) = w(y_\iota-x_{\iota +1})<w(y_{\iota +1}-x_{\iota +1})$ in view of the inequalities we have just seen. So we have $B(x_\iota, \tilde y_\iota)=B(x_\iota, y_\iota)$ and $\tilde y_\iota\notin B(x_{\iota +1}, y_{\iota +1})$. In our description of the balls we can replace each $y_\iota$ by $\tilde y_\iota$. By construction $\tilde y_{\iota +1}-\tilde y_\iota$ does not contain any term of weight $>w(y_{\iota +2}-y_{\iota +1})=w(\tilde y_{\iota +2}-\tilde y_{\iota +1})$ or of weight $<w(\tilde y_{\iota +1}-\tilde y_\iota)$, so that the terms of   $\tilde y_{\iota +1}-\tilde y_\iota$ which can appear also in  $\tilde y_{\iota +2}-\tilde y_{\iota +1}$ are those of weight $w(\tilde y_{\iota +2}-\tilde y_{\iota +1})$, which are finite in number, and they cannot appear in $\tilde y_{\iota' +1}-\tilde y_{\iota'}$ for $\iota'>\iota +1$.\Par
 The (possibly transfinite) sum $y=\tilde y_1+\sum_{\iota\in H}(\tilde y_{\iota +1}-\tilde y_\iota)$ is an element of $\widehat{k[(u_i)_{i\in I}]}$ which is in the intersection of all the balls since for each $\iota\in H$ we have $y-\tilde y_\iota=\sum_{\kappa\geq \iota}(\tilde y_{\kappa +1}-\tilde y_\kappa)$. \end{proof}\noindent
 \begin{Remark}\small{\begin{enumerate} \item A \emph{pseudo-convergent} sequence\footnote{They are also known as pseudo-Cauchy sequences. This concept is due to Ostrowski; see \cite{Ri}, \cite{Roq}.} of elements of $\widehat{k[(u_i)_{i\in I}]}$ is a sequence $(y_\tau)_{\tau\in T}$ indexed by a well ordered set $T$ without last element, which satisfies the condition that whenever $\tau<\tau' <\tau"$ we have $w(y_{\tau'}-y_\tau)<w(y_{\tau"}-y_{\tau'})$ and an element $y$ is said to be a \emph{pseudo-limit} of this pseudo-convergent sequence if $w(y_{\tau'}-y_\tau)\leq w(y-y_\tau)$ for $\tau ,\tau'\in T,\ \tau<\tau'$. One observes that if $(y_\tau)$ is pseudo-convergent, for each $\tau\in T$ the weight $w(y_{\tau'}-y_\tau)$ is independent of $\tau'>\tau$ and can be denoted by $w_\tau$. The balls $B(y_\tau, w_\tau)$ then form a strictly nested sequence of balls and their intersection is the set of pseudo-limits of the sequence. In particular, in our ring every pseudo-convergent sequence has a pseudo-limit. For example, assume that our value group $\Phi$  has rank $>1$ and let $\Psi_j$ be a non trivial convex subgroup. Denote by $T\subset I$ the set of indices $i\in I$ such that $\gamma_i\in\Psi_j$ and assume that it has no last element. Then the sum $y=\sum_{i\in T}u_i$ is a pseudo-limit of the pseudo-convergent sequence $y_\tau=\sum_{i\leq\tau} u_i$, but if we now take any series $z$ of terms involving the variables $u_k$ such that $\gamma_k\notin \Psi_j$, then $y+z$ is another pseudo-limit. If we assume that $\Psi_j$ is of rank one and consider the sequence  $\tilde y_\tau=\sum_{i\leq\tau} (u_i +z)$ it is still pseudo-convergent and certainly cannot have a limit in our ring\, but it still has $\sum_{i\in T}u_i$ as a pseudo-limit.
\item If we consider, for $0\leq s\leq h-1$, the ideal $P_s$ of elements if $\widehat{k[(u_i)_{i\in I}]}$ whose weight does not belong to the convex subgroup $\Psi_s$ of $\Phi$, the quotient map $\widehat{k[(u_i)_{i\in I}]}\to \widehat{k[(u_i)_{i\in I}]}/P_s$ induces an isomorphism from the $k$-subalgebras $\widehat{k[(u_i)_{i\vert \gamma_i\in \Psi_s}]}$ associated to the semigroup generated by the $\gamma_i$ which are in $\Psi_s$ onto the image $\widehat{k[(u_i)_{i\in I}]}/P_s$.
 \end{enumerate}}
 \end{Remark}
\noindent\textbf{The rank one case of the valuative Cohen Theorem:} Let us assume that $\nu$ is a rational valuation of rank one on the complete equicharacteristic noetherian local domain $R$ and pick a field of representatives $k\subset R$. Since the valuation is of rank one, we may fix an ordered embedding $\Gamma\subset \R$. If the set of generators of the semigroup $\Gamma$ is finite, and the variables $u_i$ correspond to a minimal set of generators as above, the scalewise completion of $k[u_1,\ldots ,u_N]$ with respect to the weight coincides with the usual completion as we shall see below in remark \ref{power}. If we choose representatives $\xi_i\in R$ of the generators of the $k$-algebra ${\rm gr}_\nu R$, they must generate the maximal ideal of $R$ since $R$ is complete for the $\nu$-adic topology as well as for the $m$-adic (\cite{Te1}, proposition 5.10). The map $u_i\mapsto\xi_i$ then induces a continuous surjection of topological $k$-algebras $k[[u_1,\ldots ,u_N]]\to R$ whose associated graded map is the surjection $k[U_1,\ldots ,U_N]\to {\rm gr}_\nu R$.\par
  If the set of generators is infinite, the semigroup $\Gamma$ is of ordinal $\omega$ (see \cite{Te1}, proposition 3.9) and cofinal in $\R_+$ since it has no accumulation point in $\R$ (see \cite{C-T}, Theorem 3.2). Since $\bigcap_{\phi\in\Gamma}\Pp_\phi (R)=(0)$, by Chevalley's theorem there exists an application $\beta\colon\Gamma\to \N$, whose value tends to infinity with $\gamma$ and such that $\Pp_\gamma (R)\subset m^{\beta(\gamma)}$. In particular, for each $i\in I$, we can write $\xi_i\in m^{\beta (\gamma_i)}$.\par
  Given an infinite series $\sum_{e\in E}d_eu^e\in\widehat{k[(u_i)_{i\in I}]}$, we order its terms as explained above, and their weights must increase indefinitely since there are only finitely many terms of a given weight.\par\noindent
The image in $R$ of a monomial $u^e$ of weight $\gamma$ is the monomial $\xi^e$ of valuation $\gamma$ in $R$, which is contained in $m^{\beta(\gamma)}$.\par
 This shows that the sum $\sum_{e\in E}d_e\xi^e$ must converge in $R$ since it is complete for the $m$-adic topology. Thus, the application $u_i\mapsto \xi_i$ extends to a continuous map $\widehat{k[(u_i)_{i\in I}]}\to R$ of topological $k$-algebras. Let us show that it is surjective. Given $x\in R$, its initial form in ${\rm gr}_\nu R$ is a term $d_{a_0}\overline\xi^{a_0}$ with $d_{a_0}\in k^*$. Let $x_1=x-d_{a_0}\xi^{a_0}\in R$; we have $\nu (x-x_1)>\nu (x)$ and applying the same treatment to $x_1$ and continuing in this manner we build a series $\sum_{k=0}^\infty d_{a_k}\xi^{a_k}$ which converges to $x$ in the $\nu$-adic topology and therefore, again by Chevalley's Theorem, in the $m$-adic topology (see also \cite{Te1}, proposition 5.10), and is the image of an element $\sum_{k=0}^\infty d_{a_k}u^{a_k}\in\widehat{k[(u_i)_{i\in I}]}$.\par
 So we have a continuous surjection of topological $k$-algebras $\widehat{k[(u_i)_{i\in I}]}\to R$ whose associated graded map is the surjection $k[(U_i)_{i\in I}]\to {\rm gr}_\nu R$. Topological generators of the kernel are then overweight deformations (see remark \ref{ROD}) of binomials generating the kernel of the associated graded map. This is the nature of the valuative Cohen Theorem.
    \begin{remark}\label{growth} \small{\emph{Taking a finite system of generators of the maximal ideal of $R$, say $\underline \xi=(\xi_{i_1},\ldots ,\xi_{i_n})$, by Chevalley's theorem we can for each $\xi_i$ choose a power series expression $\xi_i=\sum a^{(i)}_e\underline \xi^e$ of order $\geq \beta(\gamma_i)$. The convergence of the series $\sum_{e\in E}d_e\xi^e$ comes from substituting for each $\xi_i$ with $i\notin\{i_1,\ldots ,i_n\}$ the chosen expression.} }    \end{remark}
    \par\medskip\noindent
Before we deal with rational valuations of arbitrary rank we need some preliminaries, contained in the next subsection.
\subsection{More on the structure of $\hbox{\rm gr}_\nu R$ in the case where $R$ is complete}\label{morestruc}
 Recall the notations of section \ref{cohen}: Let $\Phi$ be a group of height (or rank) $h$ associated to a valuation of a complete local domain $R$ centered at the maximal ideal $m$ of $R$; set $k=R/m$. Let
$$(0)=\Psi_h\subset \Psi_{h-1}\subset\cdots\subset \Psi_1\subset \Psi_0=\Phi$$
be the sequence of isolated subgroups of $\Phi$ (including the trivial ones) and let $$(0)\subseteq p_1\subseteq p_2\ldots\subseteq p_{h-1}\subseteq p_h=m$$ be the corresponding sequence of the centers in $R$ of the valuations with which $\nu$ is composed, with $p_i=\{x\in R\vert\nu(x)\notin\Psi_i\}$. Let $\Gamma$ be
the semigroup of $\nu$ on $R$ and set
$\Gamma_i=\Gamma\cap \Psi_i$. Considering the valuation $\nu'$ of height $h-1$ with
center $p_{h-1}$ with which $\nu$ is composed, we have seen in subsection 3.3 of \cite{Te1}
that we could identify $\hbox{\rm gr}_{\overline \nu}\overline R_{h-1}$, where $\overline R_{h-1}=R/
p_{h-1}$ and $\overline\nu$ is the residual valuation induced by $\nu$, with the subalgebra $\bigoplus_{\psi\in \Psi_{h-1+}\cup \{0\}}(\hbox{\rm gr}_\nu
R)_\psi$ of $\hbox{\rm gr}_\nu R$. The generators of this subalgebra are the $(U_i)_{i\in I_1}$, where $I_1\subset I$,
is the set of indices of the $U_i$ whose degree lies in $\Psi_{h-1}$. Then we have:
\begin{proposition}\label{strucgr} Assume that $R$ is a complete equicharacteristic noetherian local domain. If $R\subset R_\nu\subset R_{\nu'}$, where $R_\nu$
dominates the local ring $R$ without residual extension and $\nu'$ is of height one less than $\nu$. Fix a set of elements $(\xi_i\in R)_{i\in I}$ whose initial forms generate the $k$-algebra ${\rm gr}_\nu R$. The map of $\overline R_{h-1}$-algebras
$$\overline R_{h-1}[(u_i)_{i\in I\setminus I_1}]\rightarrow {\rm gr}_{\nu'}R,\ u_i \mapsto {\rm in}_{\nu'}\xi_i$$
is surjective.
\end{proposition}
\begin{proof}: Let us first recall that if we denote by $\lambda\colon \Phi\to\Phi_{h-1}=\Phi/\Psi_{h-1}$ the natural map, the valuation $\nu'$ is defined as $\lambda\circ\nu$. If we denote by $\Pp_{\phi_{h-1}}$ the ideal $\{x/\nu'(x)\geq \phi_{h-1}\}$, for all $\phi\in\lambda^{-1}(\phi_{h-1})$ we have the inclusions
$$\Pp^+_{\phi_{h-1}}\subseteq\Pp_\phi\subseteq \Pp_{\phi_{h-1}},$$
and these $\Pp_\phi$ induce a filtration $\overline\Pp(\phi_{h-1})$ on the quotient $\Pp_{\phi_{h-1}}/\Pp^+_{\phi_{h-1}}$ which is a finitely generated $R/
p_1$-module.\par\noindent
We need the following Lemma:
\begin{lemma}\label{bargen} In the situation of the proposition, given a set of representatives $\xi_i \in R$ of the generators of ${\rm gr}_\nu R$, setting $p=p_{h-1}=m_{\nu'}\cap R$, for each $\phi_{h-1}\in \Phi_{h-1}$ the $R/p$-module $\Pp_{\phi_{h-1}}/\Pp^+_{\phi_{h-1}}$ is generated by finitely many monomials in the ${\rm in}_{\nu'} \xi_i$.
\end{lemma}
\begin{proof} Let $e_1,\ldots ,e_s$ be a minimal set of generators for the finitely generated $R/p$-module $\Pp_{\phi_{h-1}}/\Pp^+_{\phi_{h-1}}$. Up to reordering, we may assume that their orders for the  $\overline\Pp(\phi_{h-1})$ filtration, which we denote by $\nu(e_i)$ since they coincide with the $\nu$ valuation of a representative in $R$, satisfy $\nu(e_1)\leq\cdots \leq \nu(e_s)$. We may further assume that for each $i$ the initial form ${\rm in}_\nu e_i$ for the $\overline\Pp(\phi_{h-1})$-filtration of $e_i$ is not the initial form of a linear combination of the $e_j$ for $j<i$, and in particular that $\nu(e_1)<\cdots <\nu(e_s)$. Indeed, if ${\rm in}_\nu e_i$ does not satisfy the condition, there are $\mu_\ell^{(1)}\in R/p,\ 1\leq \ell\leq i-1$, such that $\nu(e_i-\sum_{\ell=1}^{i-1}\mu_\ell^{(1)}e_\ell)>\nu (e_i)$, and we replace $e_i$ by $e_i-\sum_{k=1}^{i-1}\mu_k^{(1)}e_k$. If we have to continue indefinitely, we build a sequence $(\sum_{k=1}^t\sum_{\ell=1}^{i-1}\mu_\ell^{(k)}e_\ell)_{t\geq 1}$ of elements of $\Pp_{\phi_{h-1}}/\Pp^+_{\phi_{h-1}}$ which is a Cauchy sequence for the $\overline\Pp(\phi_{h-1})$ filtration, such that the $\nu$ value of the $e_i-\sum_{k=1}^t\sum_{\ell=1}^{i-1}\mu_\ell^{(k)}e_\ell$ increases indefinitely with $t$. Since $R$ is complete, so is $R/p$, and since the valuation $\overline\nu$ is of rank one, the sequence of the $\sum_{k=1}^t\sum_{\ell=1}^{i-1}\mu_\ell^{(k)}e_\ell$ converges thanks to Chevalley's theorem (see \cite{Te1}, section 5, and \cite{B3}, Chap. IV, \S 2, No. 5, Cor.4) to an element of the  submodule of $\Pp_{\phi_{h-1}}/\Pp^+_{\phi_{h-1}}$ generated by the $e_j,\ j<i$ which is closed since it is finitely generated, and this shows that $e_i$ belongs to the submodule generated by the $e_j,\ j<i$ and gives us a contradiction with the minimality of our set of generators. So after replacing $e_i$ with some $e_i-\sum_{k=1}^t\sum_{\ell=1}^{i-1}\mu_\ell^{(k)}e_\ell$, we may assume that its $\overline\Pp(\phi_{h-1})$-initial form is not  the initial form of a linear combination of the $e_j$ for $j<i$.\par In view of Lemma 3.16 of \cite{Te1}, up to multiplication by an element of $k^*$ the initial forms of the $e_i$ with respect to the $\overline\Pp(\phi_{h-1})$ filtration are monomials $\overline \xi^{\alpha_i}$ in the initial forms $\overline\xi_i\in{\rm gr}_\nu R$ of the $\xi_i$. We now prove that the ${\rm in}_{\nu'} \xi^{\alpha_i}$ also generate the module. We can write
$$e_1={\rm in}_{\nu'}\xi^{\alpha_1}+\sum_{j=1}^s\lambda^{(1)}_je_j,$$
with $\lambda^{(1)}_j\in R/p$. The element $1-\lambda^{(1)}_1$ is a unit in $R/p$ since otherwise the $\overline\Pp(\phi_{h-1})$ initial form of $e_1$ cannot be $\overline\xi^{\alpha_1}$. Now if we write $$e_2={\rm in}_{\nu'}\xi^{\alpha_2}+\sum_{j=1}^s\lambda^{(2)}_je_j,$$ we see that $\overline\nu(\lambda^{(2)}_1)>0$ because otherwise the right hand side has $\overline\Pp(\phi_1)$ order $\nu(e_1)$ which we have excluded, Moreover, we must have $\nu(\lambda^{(2)}_1e_1)\geq \nu(e_2)$ and equality is impossible because otherwise the ${\rm in}_{\overline\Pp(\phi_1)}$ initial form of $e_2$ would be a multiple of $\overline\xi^{\alpha_1}$ which we have also excluded. This implies that $1-\lambda_2^{(2)}$ is invertible. We continue like this and finally we see that the elements ${\rm in}_{\nu'} \xi^{\alpha_i}$ are expressed in terms of the $e_i$ by a matrix whose diagonal entries are invertible in $R/p$ and all entries below the diagonal are in the maximal ideal of $R/p$.
\end{proof}
 The proposition follows directly from the Lemma.\end{proof}
\begin{definition}\label{part} We define a finite partition of the index set $I$ by defining $I_t$ to be the set of indices in $I$ such that $\gamma_i$ lies in $\Psi_{h-t}\setminus\Psi_{h-t+1}$.
\end{definition}
Let us define a filtration of $\overline R_{h-1}[(u_i)_{i\in I\setminus I_1}]$ by $\Qq_\phi =\{\sum_{\alpha}a_\alpha u^\alpha | \nu(\sum_{\alpha}a_\alpha \xi^\alpha )\geq \phi\}$, and keep the $\overline\Pp(\phi)$ filtration on ${\rm gr}_{\nu'}R$. Passing to the associated graded rings recovers a weaker result (see \cite{Te1}, proposition 4.8, corollary 4.9) which does not needs the completeness assumption:
\begin{corollary}\label{goodgen} the natural map
$$\hbox{\rm gr}_{\overline \nu}\overline R_{h-1}[(U_i)_{i\in I\setminus I_1}]\to \hbox{\rm
gr}_\nu R$$ mapping $U_i$ to the generator $\overline\xi_i$ is surjective and its kernel is generated by the images in $\hbox{\rm gr}_{\overline \nu}\overline R_{h-1}[(U_i)_{i\in I\setminus I_1}]$ of those
binomials
$U^n-\lambda_{mn}U^n\in k[(U_i)_{i\in I}]$ which involve at least one
variable $U_i$ with
$i\in I\setminus I_1$. Applying this result successively to the quotients of $R$ by the prime ideals $p_{h-2}, p_{h-3},\ldots ,p_1$ and the corresponding residual valuations, we find in particular that if $\nu_1$ is the valuation of height one with which $\nu$ is composed, setting $\overline R =R/p_1$, equipped with the residual valuation $\overline\nu$, we have a similar presentation
$$\hbox{\rm gr}_{\overline \nu}\overline R[(U_i)_{i\in I_h}]\to \hbox{\rm
gr}_\nu R.$$ \end{corollary}
\noindent\begin{Remark}\small{\begin{enumerate}\item This statement complements \S 3.3 of \cite{Te1}.\par\noindent
\item In the statement of the corollary, if a binomial contains a variable $U_i$ with $i\in I\setminus I_1$ it must contain at least two, otherwise the weight of that variable would have to be in $\Psi_1$ and hence in $I_1$.
\item The fact that  the $R/p$-module $\Pp_{\phi_{h-1}}/\Pp^+_{\phi_{h-1}}$ is generated by finitely many monomials in the ${\rm in}_{\nu'} \xi_i$ does not imply that its associated graded module with respect to the $\overline \Pp(\phi_{h-1})$-filtration is a finitely generated module over ${\rm gr}_{\overline\nu}\overline R$. There are counterexamples in \cite{C-T}.
\end{enumerate}}
\end{Remark}
Using Lemma \ref{bargen} we can prove a stronger result than corollary \ref{goodgen}, which is the key to the valuative Cohen theorem in higher rank:
\begin{proposition}\label{fingen} Fix a set of representatives $(\xi_i)_{i\in I},\ \xi_i\in R$, of the generators of ${\rm gr}_\nu R$. Let $\nu_s$ be the valuation taking values in $\Phi_s=\Phi/\Psi_s$ with which $\nu$ is composed. For each $\phi_s\in \Phi_s$ the $R/p_s$-module $\Pp_{\phi_s}(R)/\Pp^+_{\phi_s}(R)$ is generated by finitely many initial forms of monomials in the elements $\xi_i$.
\end{proposition}
\begin{proof} Lemma \ref{bargen} gives the result for $s=h-1$. Let us assume the result is true for $s\geq k$ and prove that it holds for $s=k-1$. Let $\lambda_k$ denote the natural surjection $\Phi_k\to \Phi_{k-1}$ and let $\phi_{k-1}$ be an element of $\Phi_{k-1}$. For all elements $\phi_k\in\lambda_k^{-1}(\phi_{k-1})$ we have inclusions
$$\Pp^+_{\phi_{k-1}}\subset\Pp^+_{\phi_k}  \subset\Pp_{\phi_k}  \subset\Pp_{\phi_{k-1}},$$
where the indices indicate the valuation to which the valuation ideals are attached. For $\phi_k\in\lambda_k^{-1}(\phi_{k-1})$ this induces a filtration $\overline\Pp(\phi_{k-1})$ of the quotient $\Pp_{\phi_{k-1}}/\Pp^+_{\phi_{k-1}}$ indexed by elements of the rank one group $\Psi_{k-1}/\Psi_k$. This filtration may be finite if the centers in $R$ of $\nu_k$ and $\nu_{k-1}$ coincide (see \cite{Te1}, Section 3.3 and proposition 3.17). The associated graded $R/p_k$-module $\bigoplus_{\phi_k\in\lambda_k^{-1}(\phi_{k-1})}\Pp_{\phi_k}/\Pp^+_{\phi_k}$ is a sum of components of ${\rm gr}_{\nu_k}R$. \par If $(e_\ell)_{1\leq\ell\leq t}$ is a system of generators of the finitely generated $R/p_{k-1}$-module\break $\Pp_{\phi_{k-1}}/\Pp^+_{\phi_{k-1}}$, each has an initial form with respect to the $\overline\Pp(\phi_{k-1})$ filtration which is, by our inductive assumption, a linear combination $\sum_u \overline a^{(\ell)}_u{\rm in}_{\nu_k}\xi^u$ with coefficients $\overline a^{(\ell)}_u\in R/p_k$. If we take representatives $a^{(\ell)}_u\in R/p_{k-1}$ of the $\overline a^{(\ell)}_u$, we see that $\nu_k(e_\ell-\sum_u a^{(\ell)}_u{\rm in}_{\nu_{k-1}}\xi^u)>\nu_k(e_\ell)$. If we apply the same procedure again to $e_\ell-\sum_u a^{(\ell)}_u{\rm in}_{\nu_{k-1}}\xi^u$ and iterate, we build a series of elements that are combinations of ${\rm in}_{\nu_{k-1}}\xi^u$ and whose terms have increasing $\nu_k$ value. In view of Chevalley's theorem (see \cite{Te1},
Section 5, and \cite{B3}, Chap. IV, \S 2, No. 5, Cor.4), this series converges to $e_\ell$ because $R/p_{k-1}$ is a complete local ring and $\Pp_{\phi_{k-1}}/\Pp^+_{\phi_{k-1}}$ is a finitely generated module. Each $e_\ell$ being a combination of ${\rm in}_{\nu_{k-1}}\xi^u$, the $R/p_{k-1}$-module $\Pp_{\phi_{k-1}}/\Pp^+_{\phi_{k-1}}$ is generated by such "monomials", and since it is finitely generated, it is generated by finitely many of them.
\end{proof}
\begin{remark}\small{\emph{We could have used this argument in the proof of Lemma \ref{bargen}, but the approach used there may be useful in a subsequent work.}}
\end{remark}

\subsection{Statement and proof}
We keep the notations introduced in subsection \ref{morestruc}, use remark \ref{ROD} and still assume that $\nu$ is a rational valuation of $R$ and that the $(\overline \xi_i)_{i\in I}$ are a minimal system of homogeneous generators of the $k$-algebra ${\rm gr}_\nu R$.\par
\begin{theorem}\label{Valco}{\rm (Valuative Cohen theorem; compare with \cite{Te1}, 5.29)} Assuming that the local noetherian equicharacteristic domain $R$ is complete, and fixing a field of representatives $k\subset R$, there exist choices of representatives $\xi_i\in R$ of the $\overline{\xi_i}$ such that the surjective map of $k$-algebras $k[(U_i)_{i\in I}]\to {\rm gr}_\nu R,\ U_i\mapsto\overline{\xi_i}$, is the associated graded map of a continuous surjective map $\widehat{k[(u_i)_{i\in I}]}\to R,\ u_i\mapsto\xi_i$, of topological $k$-algebras, with respect to the weight and valuation filtrations respectively. The kernel of this map is generated up to closure by overweight deformations of binomials generating the kernel of $k[(U_i)_{i\in I}]\to {\rm gr}_\nu R$.
\end{theorem}
\begin{proof} We need the:

\begin{lemma}\label{prepa} A finite set $(\xi_j)_{j\in F}$ of the elements $\xi_i$ generate the maximal ideal of $R$. We can choose the representatives $\xi_i$ of the $\overline\xi_i$ in such a way that if $\nu(\xi_i)\in\Psi_k$ then it has an expression in terms  of the $(\xi_j)_{j\in F}$ which does not involve any term whose valuation is not in $\Psi_k$. If we assume that the image of $\xi_i$ in $R/p_k$ belongs to a high power of the maximal ideal, we may choose $\xi_i$ so that it belongs to the same power of the maximal ideal of $R$.
\end{lemma}
\begin{proof} By descending induction on the index $k$ of convex subgroups, using the fact that the initial form in ${\rm gr}_\nu R$ of a generator $\xi_i$ whose value is in $\Psi_k$ depends only upon its image in $R/p_k$, by corollary \ref{goodgen}. We can take a representative of $\overline\xi_i$ in $R/p_k$ and write it as a series in the images in $R/p_k$ of $(\xi_j)_{j\in F}$, and take as representative of $\overline\xi_i$ in $R$ the same expression in terms of the $(\xi_j)_{j\in F}$.
\end{proof}\noindent
In what follows we start from an initial choice of the $\xi_i$ done in this way. However, this is not sufficient to ensure that the $\xi_i$ belong to powers of the maximal ideal which increase with $i$. In order to achieve this, we have to improve our choice of representatives thanks to Chevalley's theorem, as follows: \par\medskip\noindent
Let $(\gamma_i)_{i\in I}$ be the minimal set of generators of the semigroup $\Gamma$ of $\nu$ on $R$, and $(\overline{\xi_i})_{i\in I}$ a minimal set of generators of the $k$-algebra ${\rm gr}_\nu R$, giving rise to a surjective map of $k$-algebras $k[(U_i)_{i\in I}]\to {\rm gr}_\nu R$. After definition \ref{part}, the polynomial algebra $k[(U_i)_{i\in I}]$ can be written $k[(U_i)_{i\in I_1}][(U_i)_{i\in I_2}]....[(U_i)_{i\in I_h}]$.\par
In a similar manner, since the subsemigroups of $\Gamma$ are well ordered and so combinatorially finite, the series in $\widehat{k[(u_i)_{i\in I}]}$ can be reorganized according to the decomposition of $I$ induced by the subgroups $\Psi_k$: a series in $\widehat{k[(u_i)_{i\in\bigcup_{k=1}^sI_k}]}$ can be written $\sum_AD_Au^A$ with $D_A\in\widehat{k[(u_i)_{i\in \bigcup_{k=1}^{s-1}I_k}]}$ and $u^A$ involving only variables with indices in $I_s$.\par
The weight $w$ gives rise to a monomial valuation on $\widehat{k[(u_i)_{i\in I}]}$ with associated graded ring $k[(U_i)_{i\in I}]$. The centers of the valuations with which it is composed are the ideals $q_t$ generated by $(u_i)_{\{i\vert\gamma_i\notin \Psi_t\}}$.\par
We inductively build a map of $k$-algebras $\widehat{k[(u_i)_{i\in I}]}\to R$ as follows: \par\noindent
From what we saw in the rank one case, we have a continuous surjective map $\pi_{h-1}\colon\widehat{k[(u_i)_{i\in I_1}]}\to R/p_{h-1}$ mapping $u_i$ to the image of $\xi_i$ in $R/p_{h-1}$. The image of a series $z\in\widehat{k[(u_i)_{i\in I_1}]}$ converges because the images of monomials of high weight belong to high powers of the maximal ideal of $R/p_{h-1}$ and so can be represented by series of high order in the images in $R/p_{h-1}$ of a finite number of generators $(\xi_j)_{j\in F}$ of the maximal ideal of $R$. Since by corollary \ref{goodgen} the initial forms of elements of $R$ (or $R/p_{h-2}$) whose valuation is in $\Psi_{h-1}$ depend only on the images of those elements in $R/p_{h-1}$,  as we lift to $R$ or $R/p_{h-2}$ the same series in the $(\xi_j)_{j\in F}$ or their images we obtain representatives in $R$ or $R/p_{h-2}$ of the $(\overline\xi_i)_{i\in I_1}$ which still belong to high powers of the maximal ideal of $R$ or $R/p_{h-2}$ (\emph{cf.} Lemma \ref{prepa}). Let us now seek representatives of the $\overline\xi_i$ for $i\in I_2$. Let us denote by $\lambda \colon \Psi_{h-2}\to\Psi_{h-2}/\Psi_{h-1}$ the canonical map and by $\nu_{h-2}$ the corresponding valuation on $R/p_{h-2}$. The residual valuation induced by $\nu$ on $R/p_{h-2}$ is  denoted by $\overline\nu_{h-2}$. \par We choose representatives in $R/p_{h-2}$ of the $\overline\xi_i$ for $i\in I_1$, which as we saw are series of increasing order in the images in $R/p_{h-2}$ of a finite number of generators $(\xi_j)_{j\in F}$ of the maximal ideal of $R$. Then we consider the smallest non zero element, say $\phi'_{h-1}$, of $\lambda(\Gamma\cap\Psi_{h-2})$. By proposition \ref{fingen} we know that given $\phi_{h-1}\in \Psi_{h-2}/\Psi_{h-1}$, the $R/p_{h-1}$-module $\Pp_{\phi_{h-1}}/\Pp^+_{\phi_{h-1}}$ attached to the valuation $\nu_{h-2}$ of $R/p_{h-2}$ is generated by the initial forms of finitely many monomials in our initial set of representatives in $R/p_{h-2}$ of the $(\overline \xi_i)_{i\in I_1\bigcup I_2}$.\par  We do not change the representatives of those finitely many elements and then the initial forms of all the other $\xi_s$ whose valuation has image $\phi'_{h-1}$ must be of the form ${\rm in}_{ \overline\nu_{h-2}}\xi_s={\rm in}_{\overline\Pp(\phi'_{h-1})}(\sum_{t=1}^n \overline B^{(s)}_{E_t}(\xi_i){\rm in}_{\nu_{h-2}} \xi^{E_t})$, where $\overline\Pp(\phi'_{h-1})$ is the filtration of $\Pp_{\phi'_{h-1}}/\Pp^+_{\phi'_{h-1}}$ induced by $\overline\nu_{h-2}$, the ${\rm in}_{\nu_{h-2}} \xi^{E_t}$ are the finitely many generating monomials and with $\overline B^{(s)}_{E_t}(\xi_i)\in R/p_{h-1}$ belonging to powers of the maximal ideal which tend to infinity with the valuation of $\xi_s$ by Chevalley's theorem. Again we can lift the $\overline B^{(s)}_{E_t}(\xi_i)$ as a series $B^{(s)}_{E_t}(\xi_i)$ in $R/p_{h-2}$ with the same property and choose as representative for $\xi_s$ the element $\sum_{t=1}^n B^{(s)}_{E_t}(\xi_i)\xi^{E_t}$, which belongs to higher and higher powers of the maximal ideal as $s$ increases.
We then repeat the same operation with the successor of  $\phi'_{h-1}$ in $\lambda(\Gamma\cap\Psi_{h-2})$ and so on. At each step, we have finitely monomials in the $\xi_i$ whose initial forms generate the corresponding $\Pp_{\phi_{h-1}}/\Pp^+_{\phi_{h-1}}$ and we keep the initial choice for the finitely many representatives which are used in these monomials and have not been chosen in the previous steps. As the values in $\lambda(\Gamma\cap\Psi_{h-2})$ increase the elements $\xi_s$ must belong to higher and higher symbolic powers of $p_{h-1}/p_{h-2}$ by proposition 5.3 of \cite{Te1} and thus to higher and higher powers of the maximal ideal by a result of Zariski (see \cite{Te1}, proposition 5.8). Using this we see that we can choose representatives $\xi_i\in R/p_{h-2}$ such that in any simple infinite sequence with increasing valuations the elements belong to powers of the maximal ideal tending to infinity, whether the valuations of the members that sequence have ultimately a constant image in $\lambda(\Gamma\cap\Psi_{h-2})$ or not. Writing the representatives as series in the images of the $(\xi_j)_{j\in F}$, we can lift them to $R$ or $R/p_{h-3}$.\Par
Continuing in this manner we can choose representatives in $R/p_{h-3}$ of the $\xi_i$ for $i\in I_3$ with the same property, and so on. At each step we use proposition \ref{fingen}, Chevalley's theorem and  Zariski's theorem on the symbolic powers.\par

Let us now assume that we have chosen representatives $\xi_i$ as above.\Par Assume that we have built a map $\widehat{k[(u_i)_{i\in \bigcup_{k=1}^{s-1}I_k}]}\to R$ which induces a surjection $\widehat{k[(u_i)_{i\in \bigcup_{k=1}^{s-1}I_k}]}\to R/p_{h-s+1}$ and take a series $\sum_AD_Au^A$, with $D_A\in\widehat{k[(u_i)_{i\in \bigcup_{k=1}^{s-1}I_k}]}$ and $u^A$ involving only variables with weight in $I_s$. We want to show that the series $\sum_AD_A(\xi)\xi^A$, where $D_A(\xi)$ is the image in $R$ of $D_A$, converges in $R$.\par
Let us first consider the case where $s=h$, the last one in the induction. Let us denote by $\nu_1$ the valuation of rank one with which $\nu$ is composed, with values in $\Phi_1=\Phi/\Psi_1$. We denote by $w_1(A)$ the image in $\Phi_1 $ of the weight of $D_Au^A$. By our inductive assumption, the $D_A(\xi)$ exist in $R$.\par\noindent If for every $\phi_1\in\Phi^+_1$ there are at most finitely many terms $D_A(\xi)\xi^A$ whose $\nu_1$ value is $\phi_1$, either the sum $\sum_AD_A(\xi)\xi^A$ is finite, or the $\nu_1$ values of the terms $D_A(\xi)\xi^A$ increase indefinitely, the series converges for the $\nu_1$ valuation, and therefore the sum exists in $R$, which is complete for the $\nu_1$-adic valuation (see \cite{Te1}, \S 5). If such is not the case, let $\phi_1\in \Phi_1$ be the least value for which there are infinitely many terms of the series whose $\nu_1$-value is $\phi_1$. As the value of $D_A(\xi)\xi ^A$ increases, at least one of three things must happen: the value of $D_A(\xi)$ increases and so it must belong to increasing powers of the maximal ideal, or the value of $\vert A\vert$ increases, with the same consequence, or the indices of the $\xi_i$ appearing in $\xi^A$ increase, and in view of our choice of representatives, again they belong to increasing powers of the maximal ideal. Therefore, the series $\sum_{w_1(A)=\phi_1} D_A(\xi)\xi ^A$ converges in $R$ to an element $\Sigma_{\phi_1}$.\Par The sum $\sum_{w_1(A)<\phi_1}D_Au^A$ is well defined by our induction hypothesis and the choice of $\phi_1$. So we have just shown that the image of the sum $\sum_{w_1(A)\leq \phi_1}D_Au^A $ is well defined in $R$.\Par Now we repeat the argument with the successor of $\phi_1$ in the image of $\Gamma$ in $\Phi/\Psi_1$, and continuing in this manner we build a series $\sum_{w_1(A)< \phi_1}D_A(\xi)\xi^A +\Sigma_{\phi_1}+\Sigma_{\phi_2}+\cdots $ of elements of $R$ indexed by the elements of the image of $\Gamma$ in the rank one group $\Phi/\Psi_1$ and where by construction the partial sum up to the index $\phi_k$ coincides with the image of $\sum_{w(A)\leq \phi_k}D_Au^a$.
\Par Either the series is a finite sum or the $\nu_1$ valuations of the images in $R$ of its terms must tend to infinity in the image of $\Gamma$ in the rank one group $\Phi_1$ and so its image converges for the $\nu$-adic topology and therefore, by corollary 5.9 of \cite{Te1} it converges for the $m$-adic topology of $R$ as well. By construction its sum is the image of our original series.\par
Let us now go back to our induction, and apply this result to $\widehat{k[(u_i)_{i\in \bigcup_{k=1}^sI_k}]}$ and $R/p_{h-s}$. It tells us that we can define a map $\widehat{k[(u_i)_{i\in \bigcup_{k=1}^sI_k}]}\to R/p_{h-s}$ where the image of an element of the first ring is the sum of a series made of terms in the images of the $u_i$, which converges in the $m/p_{h-s}$-adic topology. Lifting these terms to $R$ defines a series which converges in the $m$-adic topology and defines a map $\widehat{k[(u_i)_{i\in \bigcup_{k=1}^sI_k}]}\to R$. By induction we have now defined our map
$$\pi\colon\widehat{k[(u_i)_{i\in I}]}\to R.$$
Let is prove that it is surjective. Given $x\in R$ we follow exactly the same procedure as we did in the rank one case. The difference is that now it gives us a transfinite series since the steps of the procedure are indexed by $\Gamma$. This series determines an element of $\widehat{k[(u_i)_{i\in I}]}$. By what we have just seen this series converges to $x'\in R$. If $x-x'\neq 0$, its initial form is part of the series, which gives a contradiction. So the series converges to $x$, which proves the surjectivity.\Par By proposition \ref{initial} if $G$ is a non zero element of the kernel $F$ of $\pi$, its initial form belongs to the binomial ideal which is the kernel of ${\rm gr}_w\pi$. Set, with a slight abuse of notation, $G_1=G-{\rm in}_w G$ and iterate this process. We represent $G$ as the sum of a series of homogeneous elements of increasing weight, whose images in ${\rm gr}_w\widehat{k[(u_i)_{i\in I}]}$ belong to ${\rm Kergr}_w\pi$. This shows that if we take elements of $F$ whose initial forms generate the initial ideal ${\rm Kergr}_w\pi$, the closure of the ideal $\tilde F$ which they generate is $F$. The initial ideal of $F$ is, by construction, equal to the initial ideal of $\tilde F$.  We could also invoke the faithful flatness of the specialization of $R$ to ${\rm gr}_\nu R$ (see \cite{Te1}, proposition 2.3 and proposition 5.38).
\end{proof}
\begin{Remark}\label{power}\small{\begin{enumerate}
\item The argument given in the proof shows that when the set $I$ is finite, the scalewise completion $\widehat{k[(u_i)_{i\in I}]}$ of the polynomial ring $k[(u_i)_{i\in I}]$ coincides with the usual power series ring $k[[(u_i)_{i\in I}]]$. The point is again that if there are finitely many variables, in order for the weight to increase in a sequence of monomials, the degrees of the monomials must increase. To sum up:\Par$\bullet$\emph{ If the valuation $\nu$ is of rank one or if the semigroup $\Gamma$ is finitely generated, any choice of representatives $\xi_i\in R$ of the generators $\overline\xi_i$ of ${\rm gr}_\nu R$ will be suitable for the valuative Cohen theorem.}\item If $\nu$ has rank $>1$, some choices of the representatives $\xi_i$ can lead to sums of
$\widehat{k[(u_i)_{i\in I}]}$ having no image in $R$ by the map $u_i\to\xi_i$. The problem comes from infinite sets of representatives having value in some $\Pp_{\phi_i}$ with $\phi_i\in\Phi/\Psi_i$, but containing an uncontrollable "tail" with value in $\Pp^+_{\phi_i}$. For example if the valuation has rank two, we consider the sum $\sum_{i\in I_1}u_i$, and the corresponding $\xi_i$ are all of the form $\xi'_i +\eta$ with ${\rm in}_\nu \xi'_i={\rm in}_\nu \xi_i$, the $\xi'_i$ belonging to higher and higher powers of the maximal ideal, and $\nu(\eta)\in \Phi\setminus\Psi_1$.
\item The nature of the proof suggests that it can be extended to the non-equicharacteristic case.
\item By construction, we have surjective maps $\widehat{k[(u_i)_{i\in \bigcup_{k=1}^sI_k}]}\to R/p_{h-s}$. Their kernel is generated up to closure by the generators (up to closure) of the kernel of the map $\pi\colon\widehat{k[(u_i)_{i\in I}]}\to R$ from which one has removed all the terms containing a variable $u_i$ with $i\notin \bigcup_{k=1}^sI_k$.

 \end{enumerate}}
\end{Remark}
\begin{example}We revisit examples 3.19 and 5.27 of \cite{Te1}. Let $R$ be a complete noetherian equicharacteristic local domain with residue field $k$, and we fix a field of representatives $k\subset R$. Let $f\in R$ generate a non trivial prime ideal, and let us choose a rational valuation $\nu$ on $R/fR$ with value group $\Psi_1$. We can define a valuation $\mu$ on $R$ with value group $\Z\oplus\Psi_1$ ordered lexicographically, as follows: $\mu(x)=(\ell, \nu(f^{-\ell}x\ {\rm mod.}fR))$, where $\ell$ is the unique integer such that $x\in f^\ell R\setminus f^{\ell+1}R$. Then, by direct inspection or by invoking \textit{loc.cit.}, we have the equality ${\rm gr}_\mu R={\rm gr}_\nu( R/fR)[F]$, where $F={\rm in}_\mu f$. Let us choose elements $\xi_i\in R$ such that their images in $R/fR$ have $\nu$-initial forms which generate the $k$-algebra ${\rm gr}_\nu( R/fR)$ and for which we can apply the valuative Cohen theorem. Let us take variables $u_i$ corresponding to the $\xi_i$ and a variable $v$ corresponding to $f$. Given a series in the $u_i$ and $v$ in the ring $\widehat{k[(u_i)_{i\in I},v]}$, we can write it $\sum_AD_A(u)v^A$. By Theorem \ref{Valco} the $D_A(u)$ have images in $R$, which  we shall write $D_A(\xi , f)$, and then the series $\sum_AD_A(\xi , f)f^A$ converges in $R$.\par To prove that the map $\widehat{k[(u_i)_{i\in I},v]}\to R$ so defined is surjective it suffices to prove that any element $af^\ell$, with $a\notin fR$, is in its image. The $\mu$-initial form of $a$ is the same as the $\nu$-initial form of its image mod.$fR$. It is a term $c^{(\ell)}_1\overline\xi^{e_1}$, with $c^{(\ell)}_1\in k^*$. We consider $a_1=a-c^{(\ell)}_1\xi^{e_1}$, note that $\mu(a_1)>\mu (a)$ and iterate this procedure, obtaining a (possibly transfinite) series $\sum_jc^{(\ell)}_j\xi^{e_j}\in R$, whose image in $R/fR$  converges to the image of $a$.\par\noindent So we have $(a-\sum_jc^{(\ell)}_j\xi^{e_j})f^\ell\in f^{\ell+1}R$. Let $\ell+k$, with $k\geq 1$, be the $f$-adic value of this element and let us write it $a^{(\ell+k)}f^{\ell+k}$, with $a^{(\ell+k)}\notin fR$. We repeat the procedure, building a series $\sum_jc^{(\ell+k)}_j\xi^{e_j}$ whose image in $R/fR$ converges to the image of $a^{(\ell+k)}$, and continue in this manner. In this way we create a series $\sum_{t=\ell}^\infty(\sum_jc^{(t)}_j\xi^{e_j})f^t$ which converges to $af^\ell$ in $R$ since $\bigcap_{t=\ell}^\infty f^tR=(0)$ and is the image of the series $\sum_{t=\ell}^\infty(\sum_jc^{(t)}_ju^{e_j})v^t\in \widehat{k[(u_i)_{i\in I},v]}$.

\end{example}
 \section{Valued complete noetherian local domains as overweight deformations}\label{OW}
Let us now go back to the notations of the introduction; let $R
$ be a complete equicharacteristic noetherian local domain and $\nu$ a rational valuation of $R$ with value group $\Phi$. We assume that the residue field $k$ of $R$ is algebraically closed and choose once and for all a field of representatives $k\subset R$. We follow the notations of \cite{Te1}.\par\noindent
\textit{Let us assume that the semigroup $\Gamma$ attached to $(R,\nu)$ is finitely generated}. Let $\gamma_1,\ldots ,\gamma_N$ be a set of generators of $\Gamma$ and let $\xi_1,\ldots ,\xi_N$ be elements of $R$ with $\nu(\xi_i)=\gamma_i$ (see remark  \ref{power}. 1)). Their images $\overline\xi_i$ in ${\rm gr}_\nu R$ generate it as a $k$-algebra. The kernel of the surjective map of graded $k$-algebras $$k[U_1,\ldots ,U_N]\to \hbox{\rm gr}_\nu R$$ determined by  $U_i\mapsto\overline\xi_i$ is a prime binomial ideal $F_0$ (see \cite{Te1}, corollary 4.3). By proposition 5.49 and corollary 5.52 of \cite{Te1}, or the valuative Cohen Theorem of the previous section, since $R$ is complete and in view of remark \ref{power}, this presentation of ${\rm gr}_\nu R$ lifts to a continuous surjection of $k$-algebras
$$k[[u_1,\ldots ,u_N]]\to R,\ u_i\mapsto\xi_i$$ whose kernel is generated by an overweight deformation of the binomial ideal $F_0$ for the weight determined by $w(u_i)=\nu(\xi_i)=\gamma_i$, and which is such that the valuation $\nu$ is the valuation determined by this weight. This is summarized as follows:
\begin{proposition}\label{OWD} Let $R$ be a complete equicharacteristic noetherian local domain and let $\nu$ be a rational valuation on $R$. Assume that the semigroup $\nu(R\setminus\{0\})$ is finitely generated. Then $(R,\nu)$ is an overweight deformation of its associated graded ring $\hbox{\rm gr}_\nu R$.\hfill\qed
\end{proposition}
\begin{Remark}\small{ \begin{enumerate}\item We have seen the converse of this proposition in proposition \ref{resoverwght}, a).
 \item Since we assume that $R$ is complete, as a consequence of the valuative Cohen Theorem, a system $(\xi_i)_{i\in I}$ of elements of the maximal ideal of $R$ such that their initial forms ${\rm in}_\nu\xi_i$ generate the $k$-algebra ${\rm gr}_\nu R$ is a system of generators for the maximal ideal of $R$. The valuations of the $\xi_i$ being positive, the weights of the variables $u_i$ are all $>0$.
\end{enumerate}}
\end{Remark}
\begin{theorem}\label{LU} Let $R$ be a complete equicharacteristic noetherian local domain with algebraically closed residue field $k$ and let $\nu$ be a rational valuation on $R$. Assume that the associated graded ring ${\rm gr}_\nu R$ is finitely generated as a $k$-algebra and let $(\xi_i)_{1\leq i\leq N}$ be elements of the maximal ideal of $R$ whose initial forms generate ${\rm gr}_\nu R$. Let us denote by $X$ the formal subspace of $\A^N(k)$ corresponding to the surjection $k[[u_1,\ldots ,u_N]]\to R$ determined by $u_i\mapsto\xi_i$. There exist regular fans $\Sigma$ with support $\R^N_{\geq 0}$ such that the strict transform $X'$ of $X$ by the birational toric map $Z(\Sigma)\to\A^N(k)$ is non singular and transversal to the toric boundary at the point picked by $\nu$.
\end{theorem}
\begin{proof}This is just the translation of proposition \ref{resoverwght}, b) using proposition \ref{OWD}.
\end{proof}
\begin{remark} In view of corollary \ref{numer} below, this applies in particular to all one dimensional complete equicharacteristic noetherian local domains with an algebraically closed residue field and their unique valuation induced by the normalization. If $R$ is an excellent one dimensional equicharacteristic local domain with an algebraically closed residue field, we shall see in subsection \ref{extcomp} below that each valuation $\nu$ on $R$ is induced by the unique valuation $\hat\nu$ of a quotient of the $m$-adic completion $\hat R$ by a minimal prime $H$. In this one-dimensional case the natural inclusion ${\rm gr}_\nu R\subset {\rm gr}_{\hat\nu}\hat R/H$ is an isomorphism (see \cite{HOST}, lemma 2.2) so that this last graded $k$-algebra is generated by the initial forms of elements of the maximal ideal of $R$. This determines a map ${\rm Spec} R\to \A^N(k)$ and ${\rm Spec} R$ has a local strict transform at the center of $\nu$ under each of the birational toric maps $Z(\Sigma)\to\A^N(k)$ which give an embedded local uniformization of $\hat\nu$. These strict transforms give embedded local uniformizations of the valuation $\nu$ on $R$. This will be generalized below in corollary \ref{anyAb}.
\end{remark}
\section{Some results on semigroups}
\emph{In the rest of this paper, by {\rm affine semigroup} we mean a subsemigroup of a finitely generated free abelian group.}
\subsection{On the finite generation of affine semigroups}
\begin{proposition}\label{transl} Given an extension $\Gamma\subset \Gamma'$ of affine semigroups, assume that there exist a system of generators $(\delta_j)_{j\in J}$ of $\Gamma'$, an integer $d$ and element $\gamma\in \Gamma$ such that $\gamma+d\delta_j\in\Gamma$ for all $j\in J$. If $\Gamma$ is finitely generated, so is $\Gamma'$.
\end{proposition}
\begin{proof}: By Dickson's Lemma\footnote{Or the fact that the semigroup ring $\Z[t^\Gamma]$ is a noetherian ring and the ideal generated by the elements $t^{\gamma+d\delta_j}$ is finitely generated.}, the monoideal of $\Gamma$ generated by the $\gamma+d\delta_j$ is finitely generated, say by $(\gamma+d\delta_{j_i})_{1\leq i\leq f}$. Thus for any $\delta_j$ we can write $\gamma+d\delta_j=a^{(j)}_1(\gamma+d\delta_{j_1})+\cdots +a^{(j)}_f(\gamma+d\delta_{j_f})+\epsilon^{(j)}$ with $\epsilon^{(j)}\in \Gamma$, $a^{(j)}_i\in \N$ and some $a^{(j)}_i\neq 0$. This shows that $d\delta_j$ is in the subsemigroup $\tilde\Gamma$ of $\Gamma'$ generated by $\Gamma$ and the $(d\delta_{j_i})_{1\leq i\leq f}$. Replacing $\Gamma$ by the finitely generated semigroup $\tilde\Gamma\subset \Gamma'$ we are reduced to the case where $\gamma=0$. In that case we have the inclusions $d\Gamma\subset d\Gamma'\subset \Gamma\subset\Gamma'$ and it suffices to prove that $d\Gamma'$ is finitely generated. Let us denote by $M$ (resp. $M'$) the group generated by $\Gamma$ (resp. $\Gamma'$). By our assumption we have $dM'\subset M$ and if we denote by $\check\sigma$ the cone generated by $\Gamma$ in $M_\R=M'_\R$, it is also the closed cone generated by $\Gamma'$ or $d\Gamma'$. We can add to $d\Gamma$ finitely many elements of $d\Gamma'$ so that the resulting subsemigroup $\tilde\Gamma_1\subset d\Gamma'$ generates the same group $dM'$. By the existence of a conductor for finitely generated affine semigroups, (see Theorem 1.4 of \cite{Ka-Kh1})\footnote{Or the fact that given a finitely generated affine semigroup $\Gamma$ generating a free abelian group $M$ and a rational convex cone $\check\sigma$ in $M_\R$, the semigroup algebra $k[t^{\check\sigma\cap M}]$ is the integral closure of $k[t^\Gamma]$ in $k[t^M]$ and a finitely generated graded $k[t^\Gamma]$-module, so that there exist homogeneous elements $t^\kappa\in k[t^\Gamma]$ such that $t^\kappa k[t^{\check\sigma\cap M}]\subset k[t^\Gamma]$.} , there exists an element $\kappa\in \tilde\Gamma_1$ such that $\kappa+\check\sigma\cap dM'\subset \tilde\Gamma_1$, and in particular $\kappa+d\Gamma'\subset \tilde\Gamma_1$. We can repeat with $\tilde\Gamma_1$ and $d\Gamma'$, now taking $d=1$ and $\gamma=\kappa$  in the hypothesis of the lemma, the argument used at the beginning, to prove that $d\Gamma'$, and hence $\Gamma'$, is finitely generated.\end{proof}
\begin{corollary}\label{fincone} An affine semigroup $\Gamma$ containing a finitely generated subsemigroup which generates the same cone is finitely generated.
\end{corollary}
\begin{proof} As in the proof of the proposition, we can add to the subsemigroup finitely many elements of $\Gamma$ to obtain a finitely generated subsemigroup $\Gamma_1\subset \Gamma$ which generates the same group $M$ and the same cone $\check\sigma$ as $\Gamma$. By the existence of a conductor, there is an element $\kappa\in \Gamma_1$ such that $\kappa + \Gamma\subset\kappa+\check\sigma\cap M\subset \Gamma_1$ and we can apply the proposition.\end{proof}
This corollary, which is perhaps well known, can be seen as a natural generalization to higher dimensions of the following classical result due to Dickson (see \cite{R} and \cite{Gi}):
\begin{corollary} \label{numer} Any subsemigroup $\Gamma$ of $\N$ is finitely generated.
\end{corollary}
\begin{proof} Any non zero element of $\Gamma$ generates the same cone as $\Gamma$.\end{proof}
\subsection{Special extensions of groups}
The next result, which we will use later, is also known is special cases. For the classical case see \cite{GB-P}, Lemma 1-1 and the references therein. One finds rather general formulations in Lemma 3.1 of \cite{Ka} and in Lemma 2.1 of \cite{M3}.\par\noindent
 An element $i$ in a well ordered set $I$ has a \emph{predecessor} $i-1\in I$ if $i$ is the least element of $I$ which is $>i-1$. In the well ordered set $\N^2$ with the lexicographic order, the element $(1,0)$ has no predecessor.

\begin{proposition}\label{expression} Let $\Phi_0$ be a commutative torsion free group. Let $(\delta_i)_{i\in I}$ be a family, indexed by an ordinal $I$, of elements of a torsion free commutative group $\Phi$ containing $\Phi_0$ as a subgroup, and assume that for each $i\in I$ there exists an integer $n_i\in \N,\ n_i\geq 1$, such that $n_i\delta_i$ belongs to the subgroup $\Phi_i^-$ generated by $\Phi_0$ and the elements $\delta_j, j<i$. Assume also that for each $i\in I$ the set $E(i)=\{k\in I,k\leq i, n_k>1\}$ is finite. Then for each element $\phi$ of the subgroup $\Phi_i$ generated by $\Phi_0$ and the elements $\delta_j, j\leq i$, there exists  a presentation:
$$\phi=\phi_0+\sum_{k\in E(i)}t_k\delta_k,\ \phi_0\in \Phi_0,\ t_k\in\N,\ 0\leq t_k\leq n_k-1.$$
If we assume that each $n_i$ is the smallest integer such that $n_i\delta_i\in \Phi_i^-$, the presentation is unique.
\end{proposition}
\begin{proof}Let us denote by $1$ the smallest element of $I$ and set $\Phi_1^-=\Phi_0$. If $n_1=1$, then $\Phi_1=\Phi_0$ and the result is true. Assume that $n_1>1$ and let $\phi$ be an element of $\Phi_1$; we can write $\phi=\phi'_0+t'\delta_1$ with $\phi'_0\in\Phi_0$ and $t'\in\Z$. Divide $t'$ by $n_1$ in the following sense: write $t'=cn_1+t,\ c\in \Z,\ t\in\N,\ 0\leq t\leq n_1-1$. We can rewrite $\phi= \phi'_0+cn_1\delta_1+t\delta_1$ which has the required form since $n_1\delta_1\in\Phi_0$. Now we proceed by induction. If $n_i=1$ and if $i$ has a predecessor $i-1$ in $I$, by induction there is nothing to prove since $\Phi_i=\Phi_i^-=\Phi_{i-1}$. If $n_i=1$ and $i$ has no predecessor, the set $\{j\in I, j<i\}$ is infinite and by our hypothesis there is a largest element $j$ in it such that $n_j>1$. Then $\Phi_i=\Phi_j$ and we apply the induction to $\Phi_j$. If $n_i>1$, each element of $\Phi_i$ can be written as $\phi_i^-+t_i\delta_i$ with $\phi_i^-\in \Phi_i^-,\ 0\leq t_i\leq n_i-1$, and we apply the induction hypothesis to $\Phi_i^-$, which again is equal to $\Phi_{i-1}$ if $i$ has a predecessor $i-1$ or to some $\Phi_j$ with $j<i$ and $n_j>1$ if not.\par\noindent
The uniqueness under the minimality hypothesis follows from the fact that at each passage from $\Phi_i$ to $\Phi^-_i$ in the construction of the presentation a non zero difference of two presentations would produce a smaller factor than $n_i$.\end{proof}

\section{Valuations with finitely generated semigroup and \\Abhyankar valuations}\label{finiteAbh}
Let $R$ be a complete noetherian equicharacteristic local domain with an algebraically closed residue field, endowed with a rational valuation $\nu$. Assume that the semigroup $\nu(R\setminus\{0\})$ is finitely generated. Then by Theorem \ref{OWD} the ring $R$ is an overweight deformation of its associated graded ring, and in particular they have the same dimension. By a result of Piltant (see \cite{Te1}, proposition 3.1), the dimension of $\hbox{\rm gr}_\nu R$ is the rational rank of $\nu$, so that the valuation $\nu$ has to be Abhyankar. A slightly different argument was given in corollary \ref{seeAb}. Note that the semigroup of an Abhyankar valuation may be finitely generated with a rational rank $<{\rm dim}R$, but then the valuation is not rational.\par\medskip
The purpose of this section is to prove, in the situation studied here, a form of converse: if the rational valuation $\nu$ of the complete equicharacteristic local domain $R$ is Abhyankar, then after replacing $R$ by the completion of a toric $\nu$-modification of $R$, its semigroup becomes finitely generated. We know of no example where the semigroup of $R$ itself is not finitely generated.
\subsection{Composition of Abhyankar valuations}\label{Compo}
Recall that a valuation $\nu$ on a local domain $R$ is said to be \emph{zero dimensional} if $R$ is dominated by the valuation ring of $\nu$ and the residual extension is algebraic. Rational valuations are zero dimensional.
\begin{proposition}\label{compAb} Given a noetherian catenary local domain $R$ and an Abhyankar valuation $\nu$ of $R$, let $\nu'$ be a valuation with which $\nu$ is composed and $p'\subset R$ its center. The valuation $\nu'$ induces a zero dimensional Abhyankar valuation of $R_{p'}$ and $\nu$ induces an Abhyankar residual valuation of $R/p'$. Conversely, the composition of two Abhyankar valuations of a noetherian catenary local domain is Abhyankar, and it is zero dimensional if both valuations are.
\end{proposition}
\begin{proof} $k$ be the residue field of $R$ and $k_\nu$ the residue field of the ring $R_\nu$ of $\nu$. Let $\Psi'$ be the convex subgroup of $\Phi$ corresponding to $\nu'$. Since the residual valuation induced by $\nu$ on the quotient $R/p'$ has $\Psi'$ as value group, by Abhyankar's inequality we have ${\rm tr}_kk_\nu +{\rm rat.rk.}\Psi'\leq {\rm dim}R/p'$. Since $R$ is a catenary and local domain we have the equality ${\rm dim}R={\rm dim}R_{p'}+{\rm dim}R/p'$ by (\cite{EGA4.1}, 16.1.4.2). Since $\nu$ is Abhyankar, using Abhyankar's inequality for $\nu'$ on $R_{p'}$ gives {\small
$${\rm tr}_kk_\nu +{\rm rat.rk.}\Psi'\leq {\rm dim}R-{\rm dim}R_{p'}\leq {\rm rat.rk.}\Phi+{\rm tr}_kk_\nu-{\rm rat.rk.}\Phi/\Psi'-{\rm tr}_{R_{p'}/p'R_{p'}}R_{\nu'}/m_{\nu'},$$} which implies that all inequalities must be equalities and ${\rm tr}_{R_{p'}/p'R_{p'}}R_{\nu'}/m_{\nu'}=0$. This shows that $\nu'$ is a zero dimensional Abhyankar valuation of $R_{p'}$. The residual valuation of $\nu$ in $R/p'$ is zero dimensional if $\nu$ is, and the first equality above shows that it is Abhyankar. The converse follows from a similar dimension count. \end{proof}
\subsection{Extension of Abhyankar valuations to the completion}\label{extcomp}
Given a rational valuation of a local domain $R$, let us consider the inductive system of local birational $\nu$-extensions of $R$, that is, local rings $R'$ containing $R$, essentially of finite type over $R$ and dominated by $R_\nu$; it is a tree in the sense of \cite{HOST}, to which we refer for details.\par
The next proposition proves a (very) special case of Conjecture 9.1 of \cite{HOST}, to which we refer for basic facts concerning extensions of a valuation on an excellent local domain $R$ to a quotient of its formal completion ${\hat R}^m$ by an ideal $H$ such that $H\cap R=(0)$. In particular, to such an extension is associated a sequence of convex subgroups $$(0)={\hat\Psi}_{2h+1}\subset {\hat\Psi}_{2h}\subset\cdots\subset{\hat\Psi}_{2\ell+1}\subset{\hat\Psi}_{2\ell}\subset \cdots \subset {\hat\Psi}_1\subset {\hat\Psi}_0=\hat\Phi.$$ where $h$ is the rank of the valuation $\nu$ and $\hat\Phi$ is the value group of the extended valuation $\hat\nu_-$.
The study of extension of valuations to the completion is relatively straightforward in the rank one case, and was already known to Zariski at least in special cases as explained before proposition 5.19 in \cite{Te1}. It is dealt with in \cite{HOST} and also appears as Lemma 3.9 in \cite{JM} in the special case of regular local $k$-algebras essentially of finite type and quasi monomial valuations.
\begin{proposition}\label{Abext} Let $\nu$ be a rational Abhyankar valuation of an excellent equicharacteristic local domain $R$. There exist birational $\nu$-extensions $R\to R'$ in the inductive system, or tree, defined above, such that for any birational $\nu$-extension $R'\to R"$ the valuation $\nu\vert R"$ extends uniquely to a valuation $\hat\nu_-$ of a quotient of ${\hat {R"}}^{m"}$ by a minimal prime ideal determined by $\nu$, with the same semigroup of values. The minimal prime is equal to $H"=\bigcap_{\phi\in \Phi_+}\Pp_\phi (R"){\hat {R"}}^{m"}$.\end{proposition}
\begin{proof} According to \S 5 of \cite{HOST}, a valuation $\nu$ of rank one centered at the maximal ideal of $R$ extend uniquely to a valuation of ${\hat R}^m/H$ with the same semigroup of values, where $H=\bigcap_{\phi\in \Phi_+}\Pp_\phi{\hat R}^m$ (note that $H\cap R=(0)$). If the valuation $\nu$ is Abhyankar an extension to a quotient ${\hat R}^m/H$ has to be Abhyankar too, so that the ideal $H$ has to be a minimal prime of ${\hat R}^m$ and the value group $\hat\Phi$ of the extended valuation must also be equal to $\Z^r$.  Now let $h$ be the rank of our valuation and let us assume that the result is true for all valuations of lower rank. Let
$$(0)\subset \Psi_{h-1}\subset\cdots\subset \Psi_1\subset \Psi_0=\Phi$$ be the sequence of convex subgroups of $\Phi$. If $p_1$ is the center in $R$ of the rank one valuation $\nu_1$ with which $\nu$ is composed, by our induction assumption for $R'$ sufficiently far in the tree of $\nu$-modifications of $R$ we may assume that the valuation $\overline\nu$ on $\overline{R'}_1=R'/p'_1$ extends uniquely to a valuation $\hat{\overline\nu}_-$ on the quotient ${\overline{R'}_1}^{(\overline\nu)}$ of ${\hat {R'}}^{m'}/p'_1{\hat {R'}}^{m'}$ by the ideal $\bigcap_{\phi\in \Psi_{1+}}\Pp_\phi(R'){\hat {R'}}^{m'}/p'_1{\hat {R'}}^{m'}$, with the same semigroup.\par
The inclusion $\Phi\subset\hat\Phi$ of value groups is then $\Z^r\subset\Z^r$. As a consequence there is an integer $f$ such that $f\hat\Phi\subset\Phi$ so that the two groups have the same real rank. From this and Lemma 5.1 of \cite{HOST} it follows that ${\hat\Psi}_{2\ell+1}={\hat\Psi}_{2\ell}$ for $0\leq \ell\leq h$, and so the corresponding sequence of prime ideals of ${\hat R}^m$ associated in section 5 of \cite{HOST} to the extension $\hat\nu_-$  also satisfies $\tilde{H}_{2\ell}=\tilde{H}_{2\ell+1}$. By proposition 5.3 of \cite{HOST}  we have for an $R'$ sufficiently far in the tree described above the inclusions $H'_i\subset \tilde{H}'_i$ for $0\leq i\leq 2h$ and $$H'_{2\ell+1}=\bigcap_{\phi\in\Psi_\ell} \Pp'_\phi {\hat {R'}}^{m'} $$ and $H'_{2\ell}$ is the unique minimal prime of $p'_\ell {\hat {R'}}^{m'}$ contained in $H'_{2\ell+1}$. Since $\hat\nu_-$ is Abhyankar  $\tilde{H}'_{2\ell}$ has to be both a minimal prime of $p'_\ell {\hat {R'}}^{m'}$ and equal to $H'_{2\ell+1}$ and finally we must have $\tilde{H}'_i=H'_i$ for all $i$ and $H'_{2\ell}=H'_{2\ell+1}$ for all $\ell$. We can apply proposition 6.9 of \cite{HOST} which tells us that the extension of $\nu$ to ${\hat {R'}}^{m'}/H'_0$ is unique, minimal and tight (see Definition 6.1 of \cite{HOST}) so that in particular by proposition 6.7 of \emph{loc.cit.} the groups of values of $\nu$ and $\hat\nu_-$ are the same. In what follows we set $H'_0=H'$ and ${\hat {R'}}^{(\nu)}={\hat {R'}}^{m'}/H'$.\par
Now the valuation $\nu_1$ extends to a valuation $\hat\nu_{-,1}$ of ${\hat {R'}}^{(\nu)}$, after perhaps choosing an $R'$ further in the tree. This valuation is the valuation of rank one with which $\hat\nu_-$ is composed. Imitating the proof of Lemma 2.3 of \cite{HOST}, we see that for $z\in{\hat {R'}}^{(\nu)}$ we have  $\hat\nu_{-,1}(z)={\rm max}\{\phi_1\in\Phi_1\vert \tilde z\in\Pp_{\phi_1}(R'){\hat {R'}}^{m'}\}$, where $\tilde z$ is a representative in ${\hat {R'}}^{m'}$ of $z$. This makes sense as follows: by construction there exists an element $\phi^+_1\in\Phi_1$ such that the element $\tilde z$ is not in  $\Pp_{\phi^+_1}(R'){\hat {R'}}^{m'}$. Since $\nu_1$ is of rank one, the set of elements of the semigroup of values of $\nu_1$ which are $\leq\tilde\phi^+_1$ is finite (see \cite{Z-S}, Vol. II, App. 3, Lemma 3), so that there is a $\phi_1\leq \phi^+_1$ with $\tilde z\in\Pp_{\phi_1}(R') {\hat {R'}}^{m'}\setminus\Pp^+_{\phi_1}(R') {\hat {R'}}^{m'}$. In view of the definition of $H$, this $\phi_1$ is independent of the choice of the representative $\tilde z$. As a consequence, the semigroup of values $\Gamma'_1$ of $\hat\nu_{-,1}$ on ${\hat {R'}}^{(\nu)}$ is the same as that of $\nu_1$ and we have for $\phi_1\in \Gamma_1$ the equality $\Pp_{\phi_1}({\hat {R'}}^{(\nu)})=\Pp_{\phi_1}(R'){{\hat {R'}}^{(\nu)}}$. Thus, the natural map of graded $\overline{R'}_1$-algebras
$${\rm gr}_{\nu_1}R'\otimes_{\overline{R'}_1}\hat{\overline{R'}_1}^{(\overline\nu)}\rightarrow {\rm gr}_{\hat\nu_{-,1}}{{\hat {R'}}^{(\nu)}}$$
is an isomorphism. Since it is graded, it has to be also an isomorphism of $\hat{\overline{R'}_1}^{(\overline\nu)}$-algebras. We know that the value groups of $\nu$ and $\hat\nu_-$ are the same and we assume by induction that the value semigroups of $\overline\nu$ on $\overline{R'}_1$ and $\overline{\hat\nu_-}$ on $\hat{\overline{R'}_1}^{(\overline\nu)}$ are the same. By the structure result of corollary \ref{goodgen} the fact that the sets of generators as $\hat{\overline{R'}_1}^{(\overline\nu)}$-algebras of both algebras have to be the same implies that the semigroups of $R'$ and ${\hat {R'}}^{(\nu)}$ are equal. \end{proof}\noindent
\begin{Remark}\label{ir}\small{\begin{enumerate} \item According to lemma 7.3  of \cite{HOST}, by taking a smaller cofinal tree we can even assume in the statement of the proposition that $R'$ is analytically irreducible, so that $H'=(0)$.\item The results of \cite{HOST} assume that $R$ is excellent, which explains the hypothesis made in the proposition.\item One may ask whether if the semigroup of values of a rational Abhyankar valuation $\nu$ on $R$ is finitely generated, and $R$ is analytically irreducible, there is a unique extension $\hat\nu$ of $\nu$ to $\hat R^m$ and it has the same semigroup. According to \cite{Te1}, 7.11, in that case a birational toric map $R\to R'$ which induce a local uniformization $\hat R^m\to \hat R'^{m'}$ of $\hat\nu$ uniformizes $\nu$.
\end{enumerate}}
\end{Remark}
\subsection{Key polynomials for Abhyankar valuations}\label{Kpol}
 Assume that $R$ is complete and equicharacteristic with an algebraically closed residue field and that $\nu$ is Abhyankar and rational, and fix a field of representatives $k\subset R$. There are elements $x_1,\dots ,x_r$ in $R$, with $r=\hbox{\rm dim}R$, such that $\Phi_0=\Z\nu(x_1)\oplus\cdots\oplus\Z\nu(x_r)$ is a subgroup of finite index in $\Phi\simeq \Z^r$ (see \cite{V0}, Th\'eor\`eme 9.2). The $x_i$ are analytically independent so we have an injection $R_0=k[[x_1,\ldots ,x_r]]\to R$ which, with respect to the valuation $\nu_0=\nu\vert R_0$, corresponds to a finite extension of the value group and a trivial extension of the residue field. The valuation $\nu_0$ is a monomial valuation; its associated graded ring is ${\rm gr}_{\nu_0}R_0=k[X_1,\ldots ,X_r]$ with degree of $X_i$ equal to $\nu(x_i)$. In this subsection we show that after base change $(R,m)\to (R',m')$ on $R$ by a birational toric map $k[[x_1,\ldots ,x_r]]\to k[[x'_1,\ldots ,x'_r]]$ in the coordinates $x_1,\ldots ,x_r$ (followed by localization at the center of the valuation and completion), we obtain a situation where the transformed ring $\hat R'^{m'}$ is a finite $k[[x'_1,\ldots ,x'_r]]$-module. For a suitable choice of the $x'_i$ the extension of value groups is tame and the extension of fraction fields corresponding to $k[[x'_1,\ldots ,x'_r]]\subset \hat R'^{m'}$ is finite and separable. The description for each $y\in \hat R'^{m'}$ of the valuation $\nu\vert k[[x'_1,\ldots ,x'_r]][y]$ with key polynomials plays an important role. We can then apply the same description when $y\in \hat R'^{m'}$ is a primitive element of the separable field extension, and deduce that the semigroup $\nu (R'\setminus\{0\})$ is finitely generated.
 \begin{proposition}\label{sepex} Let $\nu$ be a rational Abhyankar valuation on $R$. After a $\nu$-modification $(R,m)\to (R',m')$, it is possible to choose elements $x'_1,\dots ,x'_r$ as above in such a way that $\hat R'^{m'}$ is a finite $k[[x'_1,\ldots ,x'_r]]$-module, the field $K'$ is a finite extension of $K'_0=k((x'_1,\ldots ,x'_r))$ and the index in $\Phi$ of the subgroup generated by $\nu(x'_1),\dots ,\nu(x'_r)$ is not divisible by the characteristic of $k=R/m$.
\end{proposition}
\begin{proof}
Since $\Phi$ is finitely generated, there is a finite set of generators $\gamma_1,\ldots ,\gamma_N$ of the semigroup $\Gamma=\nu(R\setminus\{0\})$ which generates $\Phi$ as a group. We consider the semigroup $\Gamma_1$ which they generate and apply to it what we recalled just before section \ref{OD}. A subset of $r$ linearly independent generators of this semigroup generates a lattice $\Phi_0$ of rank $r$ in the group $\Phi\simeq\Z^r$. These generators are the images of $r$ vectors of the canonical basis of $\Z^N$ by the surjective morphism of groups $b\colon\Z^N\to\Z^r\simeq\Phi$ which sends the basis vectors of $\Z^N$ to the $N$ generators of the semigroup.\par The kernel of $b$ is a saturated sub-lattice $\Ll$ of $\Z^N$ of rational rank $N-r$, and hence a direct factor of $\Z^N$. There are $L$ generators $m^\ell-n^\ell$ of $\Ll$, with $L\geq N-r$, which correspond to binomials generating the ideal of the affine toric variety associated to $\Gamma_1$. The $(N-r)$-th exterior power $\Lambda^{N-r} \Ll\subset\Lambda^{N-r}\Z^N$ is also a direct factor, so it must be a primitive vector which means that the $(N-r)\times (N-r)$ minors of the matrix $M(\Ll)$ whose columns are the coordinates of the system of generators $m^\ell-n^\ell$ of the lattice $\Ll$ in the canonical basis of $\Z^N$ are coprime (compare with \cite{Te1}, Prop. 6.2).\par Indeed, the lattice $\Ll$ is the image of a map $c\colon \Z^L\to\Z^N$ so that $\Lambda^{N-r}\Ll\subset \Lambda^{N-r}\Z^N$ is the image of $\Lambda^{N-r}c\colon \Lambda^{N-r}\Z^L\to\Lambda^{N-r}\Z^N$, a sublattice generated by vectors corresponding to the choices of $N-r$ basis vectors of $\Z^L$ ($N-r$ "columns") and whose coordinates are the $(N-r)\times (N-r)$ minors which involve those columns in the matrix describing $c$ which, in the canonical basis of $\Z^N$, is the matrix of the $ L$ vectors $m^\ell-n^\ell$. The fact that $\Lambda^{N-r}\Ll\simeq \Z$ implies that up to a change of the generators $m^\ell-n^\ell$ of $\Ll$, we may assume that $\Lambda^{N-r}\Ll$ is the image of a single basis vector of $\Lambda^{N-r}\Z^L$, which means that its coordinates in $\Lambda^{N-r}\Z^N$ are the minors corresponding to a single set of $N-r$ binomials and different choices of $N-r$ basis vectors of $\Z^N$.  \par\noindent Let us denote by $\check b\colon \check\Z^r\subset \check\Z^N$ the inclusion which is dual to the surjection $b$. Its image is a direct factor in $\check\Z^N$ so the image of $\Lambda^r\check b\colon \Lambda^r\check\Z^r\subset \Lambda^r\check\Z^N$ is also a direct factor, which means that it is a primitive vector; each of its coordinates corresponds to a choice $\{i_1,\ldots ,i_r\}$ of $r$ basis vectors of $\Z^N$ which by duality of the corresponding injection $\Z^r\subset \Z^N$ gives a projection $\Lambda^r\check\Z^N\to  \Lambda^r\check\Z^r\simeq\Z$. \par\noindent
The choice of $r$ basis vectors, of indices $\{i_1,\ldots ,i_r\}$, of $\Z^N$ corresponds to the choice of a basis vector in $\Lambda^r\Z^N$. By duality the  images by the canonical map $\Lambda^r\Z^N\to\Lambda^r\Z^r$ of the basis vectors of $\Lambda^r\Z^N$ are the coordinates of the vector $\Lambda^r\check\Z^r\in\Lambda^r\check\Z^N$. The image in $\Lambda^r\Z^r\simeq\Z$ by the map $\Lambda^r b$ of a basis vector of $\Lambda^r\Z^N$ is the determinant of the matrix whose columns are the images by $b$ of the corresponding $r$ basis vectors of $\Z^N$. The absolute value of each of these coordinates, then, is the index of the subgroup of $\Z^r$ generated by the images of the corresponding $r$ basis vectors of $\Z^N$.\par
The isomorphism between $\Lambda^r\check\Z^N$ and $\Lambda^{N-r}\Z^N$ (\cite{B2}, \S 11, No. 11, Prop. 12) already used in (\cite{GP-T2}, proof of Prop. 10.1) maps the primitive vector $\Lambda^r\check\Z^r\subset \Lambda^r\check\Z^N$ to the primitive vector $\Lambda^{N-r} \Ll\subset\Lambda^{N-r}\Z^N$ : each coordinate of the vector $\Lambda^r\check\Z^r\subset \Lambda^r\check\Z^N$ is equal to an $(N-r)\times (N-r)$ minor of the matrix $M(\Ll)$ and the $N-r$ coordinates of $\Z^N$ appearing in that minor are those indexed by the complementary set of $\{i_1,\ldots ,i_r\}$ in $\{1,\ldots ,N\}$.\par\medskip
 \emph{For any prime $p$ there must be such minors which are not divisible by $p$.}
\begin{example}\label{ex2} A simple example, closely related to Example \ref{ex1}, is given by the numerical semigroup $\Gamma=\langle 4,6,13\rangle\subset \N$. Each of these integers can be seen as the index of the injection $\Z\subset \Z$ corresponding to the choice of a generator of the semigroup. The lattice $\Ll\subset \Z^3$ of relations between the generators can be generated by the vectors $(-3,2,0)$ and $(-5,-1,2)$. The $2\times 2$ minors of the $3\times 2$ matrix whose columns are these vectors are, up to sign, $4,6,13$.
\end{example}
\noindent
 \begin{Remark}\label{tame}\small{\begin{enumerate} \item Recall the description found in \cite{Te1}, before Prop. 6.2 and in \cite{GP-T2}, Prop. 10.1 of the jacobian ideal of an affine toric variety defined by a prime binomial ideal $ P\subset k[U_1,\ldots
,U_N]$. The jacobian determinant
$J_{G,\mathbf L'}$ of rank
$c=N-r$ of the generators $(U^{m^\ell}-\lambda_\ell U^{n^\ell})_{\ell\in\{1,\ldots , L\}}$ of $P$, associated to a sequence $G=(k_1,\ldots , k_c)$ of distinct elements of $\{1,\ldots, N\}$ and a subset $\mathbf L'\subseteq \{1,\ldots , L\}$ of
cardinality $c$, satisfies the congruence
$$U_{k_1}\ldots U_{k_c}.J_{G,\mathbf L'}\equiv
\big(\prod_{\ell\in \mathbf L'} U^{m^\ell}\big)\hbox{\rm Det}_{G,\mathbf L'}(\langle m-n\rangle ) \ \
\hbox{\rm mod.} P,\eqno{(Jac)}$$ where $\big(\langle m-n\rangle\big)$ is the matrix of the
vectors $(m^\ell-n^\ell)_{\ell\in \{1,\ldots , L\}}$,  and $\hbox{\rm Det}_{G,\mathbf L'}$ indicates the minor
in question. If the field $k$ is of characteristic $p$, choosing a minor which is not divisible by $p$ amounts to choosing $r$ of the coordinates such that the corresponding projection to $\A^r(k)$ of a certain binomial variety containing the toric variety as one of its irreducible components (see \cite{Ei-S}) is \'etale outside of the coordinate hyperplanes.\par
 This is the equational aspect of the smoothness over ${\rm Spec}\Z$ of the torus ${\rm Spec}\Z[t^{\Z^r}]$ of the affine toric variety over $\Z$ corresponding to the subsemigroup $b(\N^N)$ of $\Z^r$; it has the advantage that it deforms with overweight deformations (see proposition \ref{sepsep} below). The gist of the linear algebra detailed above is that for a projection of a toric variety (equipped with a weight) to an affine space of the same dimension, over an algebraically closed field, separability and tameness of the corresponding valued fields extension (see remarks \ref{nofin}, 1)) go together. From this point of view, tameness in our case appears as a portable (with respect to strict transforms and immediate extensions, such as henselizations) equational version of etaleness, which ensures that after a birational toric modification, \emph{and only near the point picked by the valuation}, the map from the strict transform to a $r$-dimensional affine space is still etale outside of the coordinate hyperplanes.  \item As a special case, we have that given an algebraically closed field $k$ and an affine (toric) semigroup $\Gamma$, the affine toric variety ${\rm Spec}k[t^\Gamma ]$ can always be presented as a \textit{separable and tame} "weakly" quasi-ordinary singularity: if $r={\rm rat.rk.\Gamma}={\rm dim}k[t^\Gamma ]$ there exist $r$ rationally independent generators $\gamma_{i_1}, \ldots ,\gamma_{i_r}$ of $\Gamma$ such that the corresponding map $\varpi\colon {\rm Spec}k[t^\Gamma ]\to \A^r(k)$ is \'etale outside of the toric boundary and hence induces a separable extension $k(x_1,\ldots ,x_r)\to {\rm Frac} k[t^\Gamma ]$. For example, taking a field $k$ of characteristic $p$ and $\Gamma=\langle p-1,p\rangle\subset \N$, the inclusion $k[t^{p-1}]\subset k[t^\Gamma]$ has this property while the inclusion $k[t^p]\subset k[t^\Gamma]$ does not. We shall see a consequence of this at the end of Section \ref{AS} and a more interesting example of numerical semigroup in remark \ref{keyrem}.\par\noindent
The map $\varpi$ is finite and makes ${\rm Spec}k[t^\Gamma ]$ into a truly quasi-ordinary singularity if and only if the semigroup is contained in the cone generated by the vectors $\gamma_{i_1}, \ldots ,\gamma_{i_r}$. This is the case for example for the semigroups of the toric varieties to which an irreducible quasi-ordinary hypersurface specializes (in characteristic zero) as explained in \cite{GP2}.
\end{enumerate}}
\end{Remark}
Let us come back to the proof of proposition \ref{sepex}. Applying what we have seen above to the kernel of the corresponding surjective map $b\colon \Z^N\to\Z^r$ we see that we can choose $r$ of these generators such that taking elements $x_1,\ldots ,x_r\in R$ with these valuations gives us an injection $R_0=k[[x_1,\ldots ,x_r]]\subset R$ with the property that the index of the value group $\Phi_0$ of $\nu\vert R_0$ in the group $\Phi$ is not divisible by the characteristic of $k$. We shall denote by $\Gamma_0$ the free subsemigroup
$$\Gamma_0=\langle\nu(x_1),\ldots ,\nu(x_r)\rangle\subset \Phi_0$$ The ring $R$ is a quotient of a power series ring $k[[x_1,\ldots ,x_r,y_1,\ldots, y_t]]$ and we can apply the Hironaka flattening theorem in the formal case, which relies on the Hironaka division theorem of \cite{Hi} and \cite{Hi2}; see \cite{ACH} for an algorithmic characteristic-free proof and \cite{Hi} for the application to flattening. It is summarized in the Appendix, section \ref{appendix}.\par
Since we want to flatten at the point picked by one valuation in the strict transform, we only need the existence of a local flattener at a point, which follows from the division theorem for power series, and the fact that after blowing-up the flattener in the base, if the map was not already flat, the fiber of the strict transform of the map at the point picked by the valuation strictly decreases so that after finitely many such steps, the strict transform has to be flat at the point picked by the valuation (see \cite{Hi}, \cite{H-L-T}).\par The flattening theorem then tells us that there exists a $\nu$-blowing up of local rings $R_0\to R^e_0$, which we may assume to be the blowing-up of a monomial ideal (see \cite{Te1}, corollary 7.5) of $R_0$, such that $R^e$, defined as $R\otimes_{R_0}R^e_0$ divided by its $R^e_0$-torsion, and localized at the point picked by the valuation $\nu$, is a flat $R^e_0$-module. By construction, its field of fractions is the same as that of $R$. Then, since ${\rm dim}R_0 ={\rm dim}R$, the $R^e_0$-module $ R^e$ is finite and free so that the extension of fraction fields is algebraic.
\par By construction, (since it is the blowing-up of a monomial ideal, see \cite{GP-T2}, part 1, section 5)  and there is no relation between the values of the variables there exist an element $m \in \Gamma_0$ and a system of generators $(\delta^e_j)$ of the semigroup $ \Gamma^e_0$ of the values of $\nu_0$ on $R^e_0$ such that $m+\delta^e_j\in\Gamma_0$ for all $j$, so that by proposition \ref{transl} the semigroup $\Gamma^e_0$ is finitely generated. This property is preserved under further birational $\nu$-modifications of $R^e_0$. Since flatness is also preserved under further blowing-ups of $R^e_0$ we can assume that $R^e_0$ is analytically irreducible by lemma 7.3 of \cite{HOST} and use proposition \ref{Abext} to extend the valuation $\nu_0$ to a complete $\hat{R^e_0}$, with a finitely generated semigroup.\Par Here we also use lemma 1.1 of \cite{HOST} which states that in a $\nu$-modification $R\to R'$, the ideal $N=\hat m\otimes _R1+1\otimes_R m'$ is maximal in the $R$-algebra $\hat R^m\otimes_R R'$ and the injection $(\hat R^m\otimes_R R')_N\to \hat R'^{m'}$ is the completion homomorphism. Then we can apply Theorem \ref{LU} and assume that $\hat{R^e_0}$ is a power series ring with a system of local coordinates having rationally independent values $\tilde\gamma_1,\ldots ,\tilde\gamma_r$ generating a semigroup $\tilde\Gamma_0\simeq \N^r$.\par
Since $R^e$ is finite and free over $R^e_0$, and the power series ring $\hat{R^e_0}$ is henselian, the $\hat{R^e_0}$-module $R^e\otimes_{R^e_0}\hat{R^e_0}$ is again finitely generated and free and contains as a summand a complete local domain corresponding to the maximal ideal picked by $\nu$, which we denote by $\hat{R^e}$. The map $\hat{R^e_0}\to \hat{R^e}$ is finite and injective. After what we have just seen and the results of subsection \ref{extcomp}, throughout these birational maps and completions the value groups have not changed and both maps $R_0\to \hat{R^e_0}$ and $R\to \hat{R^e}$ are birational $\nu$-modifications of complete local domains followed by completion in the sense of subsection \ref{extcomp}, passing eventually to a quotient of the maximal-adic completion. \end{proof}
\begin{remark}\label{existab} \emph{Any complete equicharacteristic local domain $R$ admits rational Abhyankar valuations: it is in many ways a finite module over a power series ring $R_0=k[[x_1,\ldots ,x_r]]$ with the same residue field and by general facts of valuation theory rational monomial valuations of $R_0$ extend to $R$. This implies that local equicharacteristic noetherian domains also admit rational Abhyankar valuations since they are subrings of a quotient of their completion by a minimal prime.} \end{remark}

\emph{Until the end of this subsection we assume that $R$ satisfies the conclusion of proposition \ref{sepex}}: The ring $R_0$ is $k[[x_1,\ldots ,x_r]]$, the ring $R$ is a finite $R_0$-module and the extension of their valued fraction fields is tame. \par\vskip.1in\noindent
  Set $R_0=k[[x_1,\ldots ,x_r]]$ and denote by $\nu_0$ the restriction of $\nu$ to $R_0$.\par \noindent
Any element $y\in R$ has a unitary minimal polynomial $p(y)\in R_0[y]$ over $K_0$. Consider the subring $R_1=R_0[y]/(p(y))\subset R$ generated by $y$. The ring $R$ is integral over $R_1$. Since $R_0$ is complete it is henselian and since $R_1$ is an integral domain, it is a complete local ring and a free $R_0$-module with generators $1,y,\ldots, y^{e-1}$ where $e$ is the degree $[K_0(y):K_0]={\rm deg}p(y)$.\par\medskip
We fix an element $y\in R\setminus R_0$ and temporarily restrict our attention to the corresponding $R_1$ to show that the semigroup of the valuation $\nu\vert R_1$ is finitely generated. In the proof of Theorem \ref{AbhFin} below we shall see that this implies the result we seek. For the sake of simplicity, we continue to write $\nu$, $\Gamma$, and $\Phi\simeq\Z^r$ for the valuation and semigroup of $R_1$ and for the corresponding group of values, and $K$ for the field of fractions of $R_1$. This will last until proposition \ref{sepsep} where we shall see that we can take for $y$ a primitive element of the extension $K_0\subset K$, so that there should be no confusion.\Par
 Let us choose a minimal system of generators of the $k$-algebra ${\rm gr}_\nu R_1$. It is well ordered, in bijection with a minimal set of generators of the well ordered semigroup $\Gamma$ which, as we recalled in the Introduction, is of ordinal $\leq \omega^h$ where $h$ is the rank of $\nu$.\par
\begin{definition}\label{ord} Let $\Gamma_0\subset \Gamma$ be an inclusion of affine semigroups. Assume that the group $\Phi$ generated by $\Gamma$ is totally ordered with $\Gamma\subset\Phi_{\geq 0}$. Assume moreover that $\Gamma\setminus ((\Gamma_0\setminus\{0\})+\Gamma)$ is well ordered\footnote{We may have to use this construction in a situation where $\Gamma_0$ itself is not well ordered.}. Then we define a system of generators of $\Gamma$ by adding to $\Gamma_0$ the elements defined inductively as follows: $\gamma_1$ is the least nonzero element of $\Gamma$ which is not in $\Gamma_0$,...., $\gamma_{i+1}$ is the least nonzero element of $\Gamma$ which is not in $\Gamma_i=\langle\Gamma_0,(\gamma_k)_{1\leq k\leq i}\rangle$. This is a transfinite construction, and the resulting set of generators $(\gamma_i)_{i\in I}$ is minimal by construction and is indexed by an ordinal (see \cite{Te1}, corollary 3.10). In short, some $\gamma_i$ may be less that a generator of $\Gamma_0$, but $\gamma_i<\gamma_{i+1}$.\par\smallskip
\emph{We also make in the sequel a convenient abuse of notation: since the $\nu(x_i)$ are rationally independent, for each value $s\in\Phi_0$ there is a unique Laurent monomial in $k((x_1,\ldots ,x_r))$ with this value. We denote it by $x^s$. In other words, we identify $k((x_1,\ldots ,x_r))$ with a subfield of $k((t^{\Phi_0}))$.}\end{definition}
Let us denote by $(\overline\xi_i)_{i\in I}$ the generators of the $k$-algebra ${\rm gr}_\nu R_1$ which are not the initial forms $X_1,\ldots ,X_r$ of the $x_j$. An element of $R_1$ whose valuation is not in $\Gamma_0$ must involve the initial form of $y$ in its initial form and thus has a value $\geq\nu(y)$. Since $R_1$ is complete we may, up to a change of the variable $y$, assume that $\nu(y)$ is the smallest element of $\Gamma$ which is not in the semigroup $\Gamma_0=\langle\nu(x_1),\ldots ,\nu(x_r)\rangle$. Indeed, if $\nu(y)\in\Gamma_0$ there exist a monomial $x^{r_1}$ and $\rho_1\in k^*$ such that $\nu(y-\rho_1x^{r_1})>\nu (y)$. We repeat the argument with $y_1= y-\rho_1x^{r_1}$. If $\nu$ is of rank one and after finitely many steps we do not reach an $y_k$ such that $\nu(y_k)\notin \Gamma_0$, we have built an expression of $y$ as a series in $x$ and so $y\in R_0$ and $R_1=R_0$. If $\nu$ is of rank $>1$, we build a transfinite series and use the fact we saw in section \ref{cohen} that $R_0$ is complete with respect to $\nu_0$, to reach the same conclusion. Thus we may take $y=\xi_1$ and $\gamma_1=\nu(y)$.\par Except in the case where $r=1$, this does not exclude the possibility that $\nu(y)\in\Phi_0$ or that $\nu(y)<\nu(x_i)$ for some $i$. We set $x=(x_1,\ldots ,x_r)$ and $X=(X_1,\ldots ,X_r)$ and denote by $\Phi_0\simeq\Z^r$ the value group of $\nu_0=\nu\vert R_0$, and apply what is said in remark \ref{ord}. There is an integer $f$ such that $f\Phi\subset \Phi_0$ and so for each $\gamma_i$ there is a smallest positive integer $n_i$ such that $n_i\gamma_i$ is in the subgroup $\Phi_i^-$ of $\Phi$, with the notations of proposition \ref{expression}. If the $\Psi_k, \ 1\leq k\leq h-1$ are the non trivial convex subgroups of $\Phi$, there is no claim that $\Phi_0\subset\Psi_{h-1}$, but we note that if $n_i\gamma_i\in\Psi_k$, then $\gamma_i\in\Psi_k$ since $0<\gamma_i\leq n_i\gamma_i$. In particular if $\Phi_0\subset \Psi_{h-1}$ then $\Phi\subset\Psi_{h-1}$ and $h=1$.\par Since $\Phi_0$ is of finite index in $\Phi$, and this index is the product of all the $n_i$, only a finite number of the $n_i$ can be $>1$. In view of proposition \ref{expression} this implies that the kernel of the surjective map of $k$-algebras
$$k[X, (U_i)_{i\in I}]\stackrel{{\rm gr}_w\pi}\longrightarrow {\rm gr}_\nu R_1,\ {\rm determined\  by} \ \ X_j\mapsto X_j,\ U_i\mapsto\overline\xi_i$$
contains binomials $X^{s_i}U_i^{n_i}-\lambda_i X^{r_i}\prod_{k\in E(i)}U_k^{t^{(i)}_k}$, with $s_i,r_i\in \N^r$, $\lambda_i\in k^*$ and $0\leq t^{(i)}_k<n_k$. These binomials encode the presentations $n_i\gamma_i=\phi^{(i)}_0+\sum t^{(i)}_k\gamma_k$, with $\phi^{(i)}_0\in\Phi_0$ of proposition \ref{expression}. The element $\phi^{(i)}_0$ is uniquely determined and we assume that it is written $\phi^{(i)}_0=r_i-s_i$ with $r_i,s_i\in \N^r$  by separating non negative and negative coordinates. This uniquely determines $r_i,s_i$ and is equivalent to saying that the corresponding binomial is not divisible by any of the $X_j$. The minimality of $n_i$ implies that for each binomial the set of exponents of the variables $X_j, U_i, U_k$ is a set of coprime integers so that the binomial is an irreducible element of the polynomial ring by \cite{Ei-S}.\par\noindent
Recall that according to the valuative Cohen Theorem \ref{Valco}, the map ${\rm gr}_w\pi$ lifts to a continuous surjective map $\pi\colon\widehat{k[x,(u_i)_{i\in I}]} \to R_1$, where the first ring is the scalewise completion of the polynomial ring.

\begin{proposition}\label{nokey} Set $R_0=k[[x_1,\ldots , x_r]]$. Let $\nu_0$ be a rational valuation of $R_0$ such that the $\nu(x_i)$ are rationally independent. Let $R_1=R_0[y]/(p(y))$ where $p(y)\in R_0[y]$ is a unitary irreducible polynomial, and let $\nu$ be a rational Abhyankar valuation on $R_1$ which extends $\nu_0$. With the notations just introduced we have:
\begin{enumerate}
\item The binomials $(X^{s_i}U_i^{n_i}-\lambda_i X^{r_i}\prod_{k\in E(i)}U_k^{t^{(i)}_k})_{i\in I}$ generate the kernel $F_0$ of ${\rm gr}_w\pi$.
\item If the set $I$ has no largest element, up to a change of the representatives $\xi_i\in R_1$ of the $\overline\xi_i$, the kernel $F$ of the continuous surjective map $\pi\colon\widehat{k[x,(u_i)_{i\in I}]} \longrightarrow R_1$ determined by $x_j\mapsto x_j,\ u_i\mapsto \xi_i$ according to the valuative Cohen theorem is generated up to closure by elements $$H_i=x^{s_i}u_i^{n_i}-\lambda_i x^{r_i}\prod_{k\in E(i)}u_k^{t^{(i)}_k}-g_i-u_{i+1} \ \ \ \ {\rm for}\ i\in I,$$ where $i+1={\rm min}\{j\in I\vert j>i\}$, $g_i\in\widehat{k[x,(u_j)_{j\leq i}]}$ with $w(x^{r_i}\prod_{k\in E(i)}u_k^{t^{(i)}_k})< w(g_i)<\gamma_{i+1}$, each term of $g_i$ is of weight $<\gamma_{i+1}$,  and ${\rm in}_w(g_i)\notin F_0$. If $I$ has a largest element $\overline i$ all the equations $H_i$ can be assumed to be in the form above except the last one which is $$H_{\overline i}=x^{s_{\overline i}}u_{\overline i}^{n_{\overline i}}-\lambda_{\overline i} x^{r_{\overline i}}\prod_{k\in E(\overline i)}u_k^{t^{(\overline i)}_k}-g_{\overline i},$$ with $g_{\overline i}=\sum_{w(x^{m_p}u^p)>w(x^{s_{\overline i}}u_{\overline i}^{n_{\overline i}})}c^{(\overline i)}_px^{r_p}u^p,\ \ c^{(\overline i)}_p\in k$.
\end{enumerate}
\end{proposition}
\begin{proof} We have noted that the valuation of an element of $R_1$ which is not in the semigroup $\Gamma_0$ generated by the $\nu(x_j)$ has to be $\geq \nu(y)$. Let us denote by $F_0$ the kernel of the map ${\rm gr}_w\pi$ and let $n_1\geq 1$ be the smallest integer such that $n_1\nu(y)\in\Phi_0$. We have $X^{s_1}U_1^{n_1} -\lambda_1X^{r_1}\in F_0$. Let $X^{s'}U_1^{n'}-\lambda'X^{r'}\in F_0$ be any other relation. If we divide $n'$ by $n_1$ and write $n'=qn_1+r,\ 0\leq r<n_1$, we see that $r=0$ by the minimality of $n_1$. Now the product $X^{qs_1}(X^{s'}U_1^{qn_1}-\lambda'X^{r'})$ is in $F_0$ and is congruent modulo $X^{s_1}U_1^{n_1} -\lambda_1X^{r_1}$ to $\lambda_1^qX^{s'+qr_1}-\lambda'X^{r'+qs_1}$. So this last binomial has to be in $F_0\cap k[X_1,\ldots ,X_r]$ which is the zero ideal by our assumption. If we remember that whenever we write $r-s$ it is shorthand for the decomposition of a vector of $\Z^r$ according to its non negative and negative coordinates in the canonical basis so that $r, s$ both have non negative coordinates, we see that $s'+qr_1=r'+qs_1$, rewritten $r'-s'=q(r_1-s_1)$ implies, since $q>0$, that $s'=qs_1,r'=qr_1$. Thus we must have $\lambda'=\lambda_1^q$ and $r'=qr_1,s'=qs_1$ so that $X^{s'}U_1^{n'}-\lambda'X^{r'}\in F_0$ is a multiple of $X^{s_1}U_1^{n_1} -\lambda_1X^{r_1}$. This proves that $F_0\cap k[X_1,\ldots ,X_r,U_1]$ is the prime ideal generated by $X^{s_1}U_1^{n_1} -\lambda_1X^{r_1}$.\par We now proceed by transfinite induction on the largest index of a variable appearing in a binomial relation. Given a binomial relation $B\in F_0$ we denote by $i$ the largest index of a variable $U_k$ appearing in it. Again by division of the exponent of $U_i$ by $n_i$ we find that the exponent of $U_i$ in the relation has to be a multiple of $n_i$ and so our binomial is of the form $B=X^{s'_i}U_i^{qn_i}U(i)^{s'(i)}-\lambda 'X^{r'_i}U(i)^{r'(i)}$, where $U(i)$ represents the variables of index $<i$. The binomial $B$ involves only finitely many variables and our inductive assumption is that the ideal $F_0\cap k[X, (U_j)_{j< i}]$ is generated by the binomials of our list which involve only the variables $X, (U_j)_{j< i}$. Each of the variables $U_j$ other than $U_i$ which appear in our binomial $B$ is involved in one such relation with variables of lower weight. Each of these variables in turn is involved in one such relation with variables of lower weight, and so on. By induction we may assume that for each such relation involving a variable of weight $<\gamma_i$ the total number of variables appearing in this iterative process is finite. Then it is also finite for our binomial $B$ since it involves a finite number of variables. We do the same thing with the binomial relation $X^{s_i}U_i^{n_i}-\lambda_iX^{r_i}\prod_{k\in E(i)}U_k^{t^{(i)}_k}$ and add the corresponding variables to the set of variables associated to $B$.\Par Note that if the rank of our valuation is one, the set of variables of weight less than the weight of some $U_i$ is finite anyway.\par
At this stage we have a finite subset $(U_l)_{l\in A_i}$ of the variables $(U_l)_{l<i}$ which has the property that the kernel of the map $k[X,(U_l)_{l\in A_i}]\to{\rm gr}_\nu R_1$ is generated by the binomial relations in our list which involve these variables, since by our inductive assumption they generate all the relations between them. The group generated by the weights of the variables $X, (U_l)_{l\in A_i}$ is still $\Phi_{i}^-$ so that our $n_i$ is still minimal. \par Then, the binomial $X^{qs_i}B$ is equal to $X^{s'_i}(X^{s_i}U_i^{n_i})^qU(i)^{s'(i)}-\lambda 'X^{r'_i}U(i)^{r'(i)}$ and thus, modulo the relation $X^{s_i}U_i^{n_i}-\lambda_iX^{r_i}\prod_{k\in E(i)}U_k^{t^{(i)}_k}$, is equal to a relation between variables of indices which belong to our set $A_i$. By the induction hypothesis this product $X^{qs_i}B$ is in the ideal $I$ of $k[X, (U_l)_{l\in A_i},U_i]$ generated by the binomials $(X^{s_l}U_l^{n_l}-\lambda_lX^{r_l}\prod_{k\in E(l)}U_k^{t^{(l)}_k})_{l\in A_i}$ and $X^{s_i}U_i^{n_i}-\lambda_iX^{r_i}\prod_{k\in E(i)}U_k^{t^{(i)}_k}$.\Par The set of binomials which generate $I$ is a regular sequence since, as we build it, to each added variable corresponds one equation which involves that variable.\par The ideal $F_0$ can contain no monomial such as $X^{qs_i}$ since ${\rm gr}_\nu R$ is a domain and minimally generated by the $X_j$ and $\overline\xi_i$.\par By corollary 2.3 of \cite{Ei-S} and with its notations, our ideal is of the form $I=I_+(\rho)$, which means that it has no associated prime containing a variable $U_l$, or that $(I\colon (\prod_{j=1}^rX_j\prod_{l\in A_i\cup\{i\}}U_l)^\infty)=I$, an equality which already suffices for our purpose at this step. But for the induction we need to prove that the ideal $I$ is prime. In the case of a binomial ideal generated by a regular sequence of binomials and containing no monomial, by theorem 2.1 of \cite{Ei-S} it is enough to show that the lattice generated by the vectors corresponding to the binomials is saturated. Let $N_i$ denote the cardinality of the set $A_i$. By induction on $i$ we may assume that the lattice $\Ll_{i-1}$ generated in $\Z^{N_i}$ by the exponents of the binomials $(X^{s_l}U_l^{n_l}-\lambda_lX^{r_l}\prod_{k\in E(l)}U_k^{t^{(l)}_k})_{l\in A_i}$ is a direct factor, or equivalently is saturated. We have to prove that the same is true of the lattice $\Ll_i= \Ll_{i-1}+\Z v\subset \Z^{N_i+1}$ where $v$ the vector of exponents of the binomial $X^{s_i}U_i^{n_i}-\lambda_iX^{r_i}\prod_{k\in E(i)}U_k^{t^{(i)}_k}$. This vector is  primitive since $n_i$ is minimal. To prove that $\Ll_i$ is saturated we take a primitive vector $m\in \Z^{N_i+1}\setminus \Ll_i$ such that $qm\in \Ll_i$, and the least such $q$. Since $\Ll_{i-1}$ is saturated, we have $qm=q\ell_{i-1}+tv$ with $\ell_{i-1}\in\Ll_{i-1}$. The integers $q,t$ must be coprime since $q$ is minimal, but then $q$, which is $>1$ since $m\notin \Ll_i$, must divide $v$ which is primitive. This contradiction shows that $\Ll_i$ is saturated and the ideal $I$ is prime, so the binomial $B$ has to be in the ideal generated by the binomials $(X^{s_j}U_j^{n_j}-\lambda_jX^{r_j}\prod_{k\in E(j)}U_k^{t^{(j)}_k})_{j\leq i}$. This ends the proof of $(1)$.\par
To prove $(2)$ we first assume that we have chosen representatives $\xi_i$ for which the theorem is valid and recall that in view of the valuative Cohen theorem of section \ref{cohen} and of $(1)$, the ideal $F$ is generated, up to closure, by overweight deformations $$x^{s_i}u_i^{n_i}-\lambda_ix^{r_i}\prod_{k\in E(i)}u_k^{t^{(i)}_k}+\sum_{w(x^{m_p}u^p)>w(x^{s_i}u_i^{n_i})}c^{(i)}_px^{m_p}u^p.$$
In the ring $R_1$, the elements $\xi_i$ with $i>1$ must be series in $x_1,\ldots ,x_r, y$ since these generate the maximal ideal. This implies that each $u_j$ with $j>1$ must appear linearly with a non zero constant coefficient, which we may take equal to $1$, in one of the series $H_i$. Now $u_2$ cannot appear linearly in the series $H_2$ because the overweight condition would imply that $s_2=0$ and $n_2=1$ and then $\gamma_2$ would be in the semigroup generated by the previous ones, which contradicts the minimality of our set of generators. The overweight condition also prevents $u_2$ from appearing linearly in any $H_i$ with $i\geq 3$. So $u_2$ must appear in $H_1$. Of course some $u_j$ with $j>2$ might also appear linearly in $H_1$ but if it does not appear linearly in any other equation, ultimately it will not be expressible in terms of $x_1,\ldots ,x_r,y$; as we shall see below it can be eliminated from the equation by a change of the representative $\xi_i\in R_1$ of the generator $\overline\xi_i$ of the graded algebra. The same argument shows that $u_{i+1}$ must appear linearly in $H_i$ for all $i\in I$.\par
Our next step is to show that we can modify the $H_i$ into another system of equations generating the same closed ideal, and which have the form:
$$x^{s_i}u_i^{n_i}-\lambda_ix^{r_i}\prod_{k\in E(i)}u_k^{t^{(i)}_k}-g_i-u_{i+1}+\sum_{w(x^{r_p}u^p)>\gamma_{i+1}}c^{(i)}_px^{r_p}u^p.\eqno{(*)}$$
Let us write, perhaps at the price of replacing the representatives $\xi_{i+1}$ by $\rho_{i+1}\xi_{i+1}$, with $\rho_{i+1}\in k^*$, the equations $(H_i)$ above in the form:
$$x^{s_i}u_i^{n_i}-\lambda_ix^{r_i}\prod_{k\in E(i)}u_k^{t^{(i)}_k}-G_i(x,u)-u_{i+1}+\sum_{w(x^{r_p}u^p)>\gamma_{i+1}}c^{(i)}_px^{r_p}u^p.\eqno{(*+)},$$
where $G_i$ does not contain any term of weight $\geq \gamma_{i+1}$. The equality $\nu(G_i(x,u))=\gamma_{i+1}$ is impossible because of the minimality of the system of generators of $\Gamma$. If $\nu(G_i(x,u))>\gamma_{i+1}$, by adding an element of the kernel $F$ of the surjection $\pi\colon \widehat{k[x,(u_i)_{i\in I}]}\to R_1$ we can eliminate $G_i$ from the equation by incorporating it in the last sum. If that is not the case, again by adding an element of $F$ we can replace $G_i$ in the expression $(*+)$ by an element $g_i$ such that $w(g_i)=\nu(\pi(g_i))<\gamma_{i+1}$. Since we have chosen representatives $\xi_i$ which are appropriate for the valuative Cohen theorem, we can also assume that for all $i$ except possibly a finite number, the sum $\sum_{w(x^{r_p}u^p)>\gamma_{i+1}}c^{(i)}_px^{r_p}u^p$ does not contain variables $u_k$ whose weight is not in the smallest convex subgroup containing $\gamma_{i+1}$.\par
Once our equations are all in the form $(*)$ we can iteratively make the changes of variables $u'_{i+1}=u_{i+1}-\sum_{w(x^{r_p}u^p)>\gamma_{i+1}}c^{(i)}_px^{m_p}u^p$, which corresponds to a change of the representatives $\xi_i\in R_1$ of the $\overline{\xi_i}$. These new representatives still satisfy the conditions for the validity of the valuative Cohen theorem. The reason is that  in the equations $(*+)$, except for finitely many of them, we may assume that all the variables $u_k$ which appear have their indices in the smallest union $\bigcup_{k=1}^sI_k$ which contains the index $i$ of $u_i$. In that case, the terms of the last sum belong to powers of the maximal ideal which increase with $i$ as was shown in the proof of the valuative Cohen theorem, and the change of representatives is harmless. The exceptions correspond to the equations where $i+1$ is not in the smallest union $\bigcup_{k=1}^sI_k$ which contains the index $i$ of $u_i$. This happens only if there are, for the quotient $R/p_{h-s}$, a last element $\overline\iota$ in $I_s$ and a last equation $$x^{s_{\overline\iota}}u_{\overline\iota}^{n_{\overline\iota}}-\lambda_{\overline\iota}x^{s_{\overline\iota}}\prod_{k\in E(\overline\iota)}u_k^{t^{(\overline\iota)}_k}-g_{\overline\iota}=0.$$ This equation has to be the trace of an equation
$$x^{s_{\overline\iota}}u_{\overline\iota}^{n_{\overline\iota}}-\lambda_{\overline\iota}x^{s_{\overline\iota}}\prod_{k\in E(\overline\iota)}u_k^{t^{(\overline\iota)}_k}-g_{\overline\iota}-u_{\overline\iota+1}+\sum_{w(x^{r_p}u^p)>\gamma_{\overline\iota+1}}c^{(\overline\iota)}_px^{r_p}u^p$$ for $R$, in which we have lost control of the $m$-adic order of the last sum because $\gamma_{\overline\iota +1}\notin I_s$ and Chevalley's theorem cannot be brought to bear. But this concerns only finitely many of the variables $u_i$ and therefore has no consequence for the valuative Cohen theorem. \par At this point, dividing $\pi(g_i)$ by the unitary polynomial $p(y)$, we may assume that $\pi(g_i)$ is a polynomial of degree $<{\rm deg}p(y)$.\Par  If the set $I$ has no largest element, ultimately all the equations are in the form required by the proposition. If the set $I$ has a largest element $\overline i$, then the last equation is not in that form, but in the form
$$x^{s_{\overline i}}u_{\overline i}^{n_{\overline i}}-\lambda_{\overline i} x^{r_{\overline i}}\prod_{k\in E(\overline i)}u_k^{t^{(\overline i)}_k}+\sum_{w(x^{r_p}u^p)>w(x^{s_{\overline i}}u_{\overline i}^{n_{\overline i}})}c^{(\overline i)}_px^{r_p}u^p,$$
where the last sum includes the result of all the previous changes of variables.
 \end{proof}
\begin{remark}\label{inf}\small{\emph{Let us keep the notation $$ H_i=x^{s_i}u_i^{n_i}-\lambda_ix^{r_i}\prod_{k\in E(i)}u_k^{t^{(i)}_k}-g_i-u_{i+1}\in\widehat{k[x, (u_i)_{i\in I}]}.$$
If there is an infinite segment $j_1<j<j_2$ in $I$, since the number of $k\in I$ with $n_k>1$ is finite, for all but finitely many $j$ in the segment we must have $n_j=1$, and when $n_j=1$ the minimality of the system of generators implies $s_j\neq 0$.}}
\end{remark}
Let us consider the valuation $\mu$ of $k[[x, y]]$ which is composed of the $p(y)$-adic valuation and our valuation $\nu$ on the quotient ring $R_1=k[[x, y]]/(p(y))=k[[x]][y]/(p(y))$. Its rank is $r+1$, its group is $\Z\oplus\Phi$ with the lexicographic order, and its associated graded ring is ${\rm gr}_\mu k[[x, y]]={\rm gr}_\nu R_1[P]$, where $P$ is the $\mu$-initial form of $p(y)$ (see \cite{Te1}, Example 3.19).\par By abuse of language, let us continue to denote by $\xi_j$ elements in $k[[x,y]]$, polynomials of degree $<{\rm deg}p(y)$,  which lift the corresponding ones in $R_1$. Similarly, since for elements of $k[[x,y]]$ which are not in the ideal generated by $p(y)$ we may identify the valuation $\mu$ with the value of $\nu$, we will sometimes use the notation $\nu$ instead of $\mu$ for such elements.\Par According to the valuative Cohen theorem we have a continuous surjective map of $k$-algebras
$$\widehat{k[x,(u_i)_{i\in I},v]}\to k[[x, y]]$$ determined by $x_j\mapsto x_j, u_i\mapsto \xi_i\in k[[x,y]]\ (u_1\mapsto y)\  {\rm and}\  v\mapsto p(y)$. In what follows we will deal mostly with elements of $k[[x,y]]$ that are not divisible by $p(y)$, and write $\nu$ instead of $\mu$.\par\noindent
We now eliminate the variables $u_i,\ i>1$ from the equations $H_i$.\par
Given $i\in I$, consider the ideal $\Ff_i$ which is the closure in $\widehat{k[x,(u_i)_{i\in I}]}$ of the ideal generated by $\{(H_j)_{j<i}, H_i+u_{i+1}\}$. Elimination produces generators for each $\Ff_i\cap k[[x,u_1]]$.
 Since the expression of $u_{i+1}$ depends only on the variables $x,u_1,\ldots ,u_i$, the elimination of the variables $u_1,\ldots ,u_i$ consists in successively replacing for $j=1,\ldots ,i$ each $u_j$ by its expression as a polynomial in $y$ in the expression of $u_{i+1}$ given by the equation $H_i$. Then we replace the result of the elimination by the remainder of its division by the unitary polynomial $p(y)$. In this manner we build a sequence of polynomials $Q_i(y)\in k[[x]][y]\subset k[[x,y]]$, with $Q_1=y,Q_2= x^{s_1}y_1^{n_1}-\lambda_1x^{r_1}-g_1(x,y), Q_3=x^{s_2}Q_2^{n_2}-\lambda_2x^{r_2}Q_1^{t^{(2)}_1}-g_2(x,y,Q_1(x,y),Q_2(x,y)),..., Q_{i+1}=x^{s_i}Q_i^{n_i}-\lambda_ix^{r_i}\prod_{k\in E(i)}Q_k^{t^{(i)}_k}-g_i(x,y,\ldots , Q_i(x,y)),\ldots $ which are all of degree $<{\rm deg}p(y)$.\par
 We remark that by construction the polynomial $Q_{i+1}$ obtained by elimination of the variables $u_j, j\geq 2$, between the generators of the ideal $\Ff_i$ has the property that $Q_{i+1}(y)-u_{i+1}$ is in the ideal $F$, so that \emph{the variables $x_j$ and the polynomials $(Q_i)_{i\in I}, p(y)$ form a generating sequence for the valuation $\mu$ of $k[[x,y]]$: every element $q(y)$ of $k[[x,y]]$ is represented by a series in the $Q_i(y)$ and $p(y)$ with coefficients in $k[[x]]$, where every term has $\mu$-valuation $\geq \mu(q(y))$.}\par\noindent
 \begin{proposition}\label{irreducible} The polynomials $Q_{i+1}(y),\ i\geq 0$, are irreducible in $k[[x,y]]$.
 \end{proposition}
 \begin{proof} The $\mu$-initial form of $Q_{i+1}(y)$ in ${\rm gr}_\mu k[[x,y]]$ is $\overline\xi_{i+1}$ which is irreducible since it is part of a minimal set of generators of that algebra, and so $Q_{i+1}(y)$ has to be irreducible. \end{proof}
 We shall later make use of the following:
 \begin{proposition}\label{ineq} For each $i\in I$ the inequality $r_i-s_i>0$ holds.
 \end{proposition}
 \begin{proof} The statement is equivalent to: $\nu(Q_i^{n_i})>\nu (\prod_{k\in E(i)}Q_k^{t^{(i)}_k})$. In fact, remembering the inequalities $0\leq t^{(j)}_k<n_k$, we are going to prove the stronger inequality $\nu(Q_i^{n_i})>\nu (\prod_{k\in E(i)}Q_k^{n_k-1})$.  By the equations of Proposition \ref{nokey}, denoting by $j\setminus 1$ the predecessor of $j$ in the finite set $E(i)\bigcup \{i\}$, we have the inequalities and equalities  $\nu(Q_i^{n_i})\geq \nu(Q_i)>\nu(Q_{i\setminus 1}^{n_{i\setminus 1}})= \nu(Q_{i\setminus 1}^{n_{i\setminus1}-1})+\nu(Q_{i\setminus1})> \nu(Q_{i\setminus1}^{n_{i\setminus1}-1})+\nu(Q^{n_{i\setminus 2}}_{i\setminus 2})$. Note that the second and last inequalities are valid because the successor of $j\setminus 1$ in $I$ is less than or equal to $j$. Now we can again write $\nu(Q^{n_{i\setminus 2}}_{i\setminus 2})>\nu(Q^{n_{i\setminus 2}-1}_{i\setminus 2})+\nu(Q^{n_{i\setminus 3}}_{i\setminus 3})$ and so on. This stops when we have exhausted $E(i)$ and proves the proposition.
 \end{proof}
 \begin{remark}\label{s=0}\small{\emph{The proposition implies that when $r=1$ all the $s_i$ are $0$. It is also closely related to the proposition of E. Garc\'\i a Barroso and A. P\l oski quoted in remark \ref{curve} below.}}
 \end{remark}
 \noindent
 We now resume the proof of the fact that the semigroup is finitely generated.
 \begin{proposition}\label{finhyp} In the situation of proposition \ref{nokey} the semigroup of the Abhyankar valuation $\nu$ on the ring $R_1=k[[x_1,\ldots ,x_r]]/(p(y))$ is finitely generated.
\end{proposition}
\begin{proof}
\emph{We begin by the case where the valuation $\nu$ is of rank one}, so that for any index $i$, there are only finitely many elements between $\nu(y)$ and $\gamma_i$ and $i$ has a predecessor denoted by $i-1$.\par\noindent
 \begin{lemma}\label{value}There is a function $\beta\colon \N\to\N$ with $\beta(i)$ tending to infinity with $i$ and such that for each $i\geq 2$  the coefficients in $k[[x]]$ of the polynomial $Q_i(y)\in k[[x]][y]$ are in $(x)^{\beta(i)}$.
\end{lemma}\begin{proof} By Chevalley's theorem the polynomials $Q_i(y)$ must belong to powers of the maximal ideal of $R_1$ which tend to infinity with $i$. Since the maximal ideal of $R_1$ is generated by $(x,y)$ and the degrees of the $Q_i(y)$ are bounded the powers of the ideal $(x)$ to which the $Q_i(y)$ belong must tend to infinity with $i$.
\end{proof}

To produce a contradiction with the assumption that the semigroup is not finitely generated, we  proceed as follows, using the valuation $\mu$ of $k[[x,y]]$ introduced above:\par\medskip
In the value group $\Z\oplus\Phi$ of the valuation $\mu$ the valuation $\mu(p(y))=(1,0)$ is certainly larger than the weight of $p(u_1)\in \widehat{k[x,(u_i)_{i\in I},v]}$ which belongs to $\{0\}\oplus\Phi$. By an immediate extension of proposition \ref{initial} to this context, since $p(u_1)-v\in F$, we have ${\rm in}_wp(u_1)\in F_0$. Since $p(u_1)$ contains only the variables $x,u_1$, we must have an expression ${\rm in}_wp(u_1)=A^{(1)}_1(x,u_1)(x^{s_1}u_1^{n_1}-\lambda_1x^{r_1})^{e^{(1)}_1}$ with $A^{(1)}_1(x,u_1)\notin F_0$ and $e^{(1)}_1\geq 1$. Applying the same treatment to $p(u_1)-A^{(1)}_1(x,u_1)(x^{s_1}u_1^{n_1}-\lambda_1x^{r_1})^{e^{(1)}_1}$ and continuing in the same manner, we build a series\break $\sum_{i\geq 1} A^{(1)}_i(x,u_1)(x^{s_1}u_1^{n_1}-\lambda_1x^{r_1})^{e^{(1)}_i}$ with $A^{(1)}_i(x,u_1)\notin F_0$ and $e^{(1)}_i\geq 1$, and we see that only two things can happen:\par\medskip\noindent
- Either we never reach a point where the series stops, and then since the weight increases at each step and the group is of rank one, $p(u_1)$ is a series in $x^{s_1}u_1^{n_1}-\lambda_1x^{r_1}$. Factoring out the smallest power of this binomial, we can write it $p(u_1)=A^{(1)}(x,u_1)(x^{s_1}u_1^{n_1}-\lambda_1x^{r_1})^{e^{(1)}}$. Thus $p(u_1)$ is divisible by the binomial and since it is irreducible it means that $p(u_1)=x^{s_1}u_1^{n_1}-\lambda_1x^{r_1}$, necessarily $s_1=0$ and $n_1={\rm deg}p(u_1)$. The semigroup of $R_1$ is generated by $\nu(x_1),\ldots ,\nu(x_r),\nu(y)$.\par\medskip\noindent
- Or such is not the case, and there is an integer $k$ such that
$$p(u_1)-A^{(1)}(x,u_1)(x^{s_1}u_1^{n_1}-\lambda_1x^{r_1})^{e^{(1)}}=B^{(1)}(x,u_1),\ {\rm with}\ {\rm in}_wB^{(1)}(x,u_1)\notin F_0,$$
where $A^{(1)}(x,u_1)(x^{n_1}u_1^{s_1}-\lambda_1x^{m_1})^{e^{(1)}}=\sum_{i=1}^k A^{(1)}_i(x,u_1)(x^{n_1}u_1^{s_1}-\lambda_1x^{m_1})^{e^{(1)}_i}$. We note that $B^{(1)}(x,u_1)$ is a polynomial in $u_1$ of degree $\leq{\rm deg}p(u_1)$ and $A^{(1)}(x,u_1)$ is not a multiple of $x^{n_1}u_1^{s_1}-\lambda_1x^{m_1}$. Moreover, by construction, every term of $B^{(1)}(x,u_1)$ is of weight larger than the weight of any term of $A^{(1)}(x,u_1)(x^{n_1}u_1^{s_1}-\lambda_1x^{m_1})^{e^{(1)}}$. In this case, it is impossible that the valuation of $x^{s_1}u_1^{n_1}-\lambda_1x^{r_1}$ is equal to $(1,0)$ and since it has to be larger than the weight of $x^{s_1}u_1^{n_1}-\lambda_1x^{r_1}$ it is equal to the valuation of $g_1+u_2$. Then, modulo the ideal $F$, we can substitute $g_1+u_2$ for the binomial, in order to increase the weight of the expression of $p(u_1)$, obtaining an expression in $\widehat{k[x, (u_i),v]}$:
$$p(u_1)= A^{(1)}(x,u_1)(g_1(x,u_1)+u_2)^{e^{(1)}}+B^{(1)}(x,u_1)\ {\rm mod.} F,\eqno{(E_1)}$$
or equivalently, taking images in $R_1$,
$$p(y)= A^{(1)}(x,y)(g_1(x,y)+Q_2(y))^{e^{(1)}}+B^{(1)}(x,y).$$
If the $\mu$-value of the right hand side, which is $(1,0)$, is equal to the weight of the expression $(E_1)$ of $p(u_1)$ mod.$F$ given just above, there must be in that expression a term of weight $(1,0)$, and by definition of the generators of the semigroup, this implies that $u_2$ must have weight $(1,0)$ and $e^{(1)}=1$. So $u_1$ is the last of the $u_i$ and in fact $u_2=v$ and equation $(E_1)$ reduces to $p(u_1)-u_2\in F$. This implies that $ A^{(1)}(x,u_1)g_1(x,u_1)+B^{(1)}(x,u_1)=0\ {\rm mod.} F$ and $p(u_1)= A^{(1)}(x,u_1,u_2)u_2\ {\rm mod.} F$ so that $p(y)$ is a multiple of $Q_2(y)$ and they must be equal since $p(y)$ is irreducible, and $ A^{(1)}(x,u_1)=1\ {\rm mod.} F$. In this case we have the equality
$$p(y)= x^{s_1}y^{n_1}-\lambda_1x^{r_1}-g_1(x,y),$$ and the semigroup $\nu(R_1\setminus \{0\})$ is generated by $\nu(x_1),\ldots ,\nu(x_r),\nu(y)$, while the semigroup $\mu(k[[x,y]]\setminus\{0\})$ is generated by $\nu(x_1),\ldots ,\nu(x_r),\nu(y), (0,1)$.\par\noindent
 If the weight of $(E_1)$ is $<(1,0)$, its initial form must be in the ideal generated by $x^{s_1}u_1^{n_1}-\lambda_1x^{r_1}$ and $x^{s_2}u_2^{n_2}-\lambda_2x^{r_2}u_1^{t^{(2)}_1}$ but since modulo $F$ each occurrence of the first binomial can be replaced by $g_1+u_2$, we can modify this right hand side until its initial form is $A^{(2)}_1(x,u_1,u_2)(x^{s_2}u_2^{n_2}-\lambda_2x^{r_2}u_1^{t^{(2)}_1})^{e^{(2)}_1}$ without changing the image in $R_1$. We apply the same treatment to the difference of $p(u_1)$ and this initial form, and continuing in this manner, we find ourselves with the same alternatives as before: either the process never ends and then for the same reason as above  we have $p(u_1)=x^{s_2}u_2^{n_2}-\lambda_2x^{r_2}u_1^{t^{(2)}_1} {\rm mod.}F$ and the semigroup of $R_1$ is generated by $\nu(x_1),\ldots ,\nu(x_r), \nu(y), \nu(x^{s_1}y^{n_1}-\lambda  _1x^{r_1})$, or we reach a form
$$p(u_1)= A^{(2)}(x, u_1, u_2)(g_2(x,u_1,u_2)+u_3)^{e^{(2)}}+B^{(2)}(x,u_1,u_2)\  {\rm mod.}F, \eqno{(E_2)}$$ with $ {\rm in}_wB^{(2)}(x,u_1,u_2)\notin F_0$.\Par
If the weight of the right hand side of $(E_2)$ is $(1,0)$ then this must be the weight of $u_3$ because the weight of all other terms is necessarily smaller than $(1,0)$. By the same argument we used for $u_2$ we have $p(u_1)-u_3\in F$ and $u_3=v$ so that $p(y)=x^{s_2}Q_2(y)^{n_2}-\lambda_2x^{r_2}y^{t^{(2)}_1}-g_2(x,y,Q_2(y))$. Otherwise the initial form of the right hand side is in $F_0$ and we can continue the process to build a sequence of presentations
$$p(u_1)= A^{(i-1)}(x, u_1, u_2,\ldots ,u_{i-1})(g_{i-1}+u_i)^{e^{(i-1)}}+B^{(i-1)}(x,u_1,\ldots ,u_{i-1})\ {\rm mod.}F,$$
with ${\rm in}_wB^{(i-1)}(x,u_1,\ldots u_{i-1})\notin F_0$, ${\rm deg}B^{(i-1)}(x,y,Q_2(y),\ldots ,Q_{i-1}(y))\leq{\rm deg}p(y)$ and ever increasing weights. Assuming that the set $I$ is infinite, when the index $i$ is so large that the degrees of the polynomials $Q_i(u_1)$ are constant and the $s_i$ are $\neq 0$ (see remark \ref{inf}), the only possibility for the weights to increase with bounded degree is that
$B^{(i-1)}(x,u_1,\ldots u_{i-1})$ involves only monomials which contain variables $u_j$ of high index. But  since, as we saw in Lemma \ref{value}, we orders in $(x)$ of the polynomials $Q_{i+1}(y)$ tend to infinity with $i$, this would imply that the coefficients of the polynomial $p(y)$ belong to arbitrarily high powers of the ideal $(x_1,\ldots ,x_r)$.
 This contradiction shows that the set $I$ must be finite in this case.\par\medskip
Assume now that the rank of the group is $>1$. Using the notations of subsection \ref{morestruc}, let $$(0)\subseteq p_1\subseteq p_2\ldots\subseteq p_{h-1}\subseteq p_h=m$$ be the sequence of the centers in $R$ of the valuations with which $\nu$ is composed. The residual valuations on the $R/p_k$ are Abhyankar by the results of subsection \ref{Compo}. By convexity for each $k$ the generators of the semigroup $\Gamma\cap\Psi_k$ are exactly the generators of $\Gamma$ which are in $\Psi_k$ and since $\Gamma\cap\Psi_k$ is the semigroup of values of the residual valuation on $R/p_k$ which is Abhyankar it has to generate the group $\Psi_k$. The indices of the $\Psi_k\bigcap\Phi_0$ in the $\Psi_k$ are divisors of $[\Phi :\Phi_0]$ and so prime to the characteristic of $k$.\par
We may then partition the set of indices $\{1,\ldots ,r\}$ of the variables $x_1,\ldots ,x_r$ as\break $\{1,\ldots ,r\}=T_1\bigcup\ldots\bigcup T_h$ with $T_k=\{i\in \{1,\ldots ,r\}\vert\nu(x_i)\in \Psi_{h-k}\setminus\Psi_{h-k+1}$, and we have the equality $p_j\cap k[[x_1,\ldots ,x_r]]= ((x_j)_{j\in T_1\bigcup\ldots\bigcup T_j})k[[x_1,\ldots ,x_r]]$. Then $R/p_1$, which is a finite module over $R_0/p_1\cap R_0=k[[(x_j)_{j\notin T_1\bigcup\ldots\bigcup T_{h-1}}]]$, has to be of the form $(R_0/p_1\cap R_0)[y]/(p_1(y))$ where $p_1(y)$ is an irreducible factor of the image of $p(y)$ in $(R_0/p_1\cap R_0)[y]$. By induction on the dimension (using the arguments of proposition \ref{fingenn} below) we may assume that the semigroup $\Gamma\cap\Psi_1$ is finitely generated since it is the semigroup of values of the residual valuation on $R/p_1$.\par
The smallest $\gamma_i$ of $\Phi\setminus\Psi_1$ therefore has a predecessor, and if we denote the corresponding polynomial by $Q_{b+1}(y)$,  we have an equation
$$Q_{b+1}=x^{s_b}Q_b^{n_b}-\lambda_bx^{r_b}\prod_{1\leq k\leq b-1}Q_k^{t^{(b)}_k}-g_b.$$
If, as we build successively the polynomials $(Q_{i})_{i>b}$, the corresponding unitary polynomials have increasing valuations in the rank one group $\Phi/\Psi_1$, we find a contradiction exactly as in the rank one case. If not, for some $\gamma_1\in \Phi/\Psi_1$ their values in $\Pp_{\gamma_1}/\Pp_{\gamma_1}^+$  must increase, and applying Chevalley's theorem to this finitely generated complete $R_1/p_1$-module shows by the same argument as in the rank one case, in view of remark \ref{inf}, when $i$ becomes very large this implies that the coefficients of $p(y)$ should belong to arbitrarily high powers of the ideal generated by $(x_j)_{j\in T_1\bigcup\ldots\bigcup T_{h-1}}$ and gives us a contradiction. \par We have tacitly assumed that the ideals $p_i$ are distinct.  A result of Zariski (see \cite{Z-S}, Vol.2, Appendix 3, lemma 4 or \cite{Te1}, 3.17; see also the proof of proposition \ref{fingen}) implies that if two consecutive $p_i$ are equal the conditions of finiteness of generation for the corresponding residual semigroups are equivalent. \par\noindent
 This shows that the semigroup of values of $\nu$ on $R_1$ is finitely generated.\end{proof}

 \begin{proposition}\label{sepsep} Assuming that we are in the situation created by proposition \ref{sepex}, the finite fields extension $K_0\subset K_0(y)$ is separable.
\end{proposition}
\begin{proof} The proof is based on the jacobian interpretation given in remark \ref{tame}, 1) above of the non vanishing ${\rm mod.} p$ of the minors studied above. Since the semigroup of values is finitely generated, the valuative Cohen theorem presents the noetherian ring $R_0[y]/(p(y))$ as a quotient of a power series ring in $N$ variables by an ideal which is an overweight deformation of a prime binomial ideal, indeed the ideal of proposition \ref{nokey}, and one uses, exactly as in the proof of proposition \ref{approx} the $N-r$ overweight deformations of binomials whose jacobian determinant is non zero to show that the projection from the formal space $X$ corresponding to $R_0[y]/(p(y))$ to the affine space $\A^r(k)$ corresponding to $R_0$ is generically \'etale because the jacobian minor of the deformed equations has the nonzero jacobian minor of binomials as its initial form (see \cite{B4}, \S 7, No. 9, Th.3), and this property of being generically \'etale corresponds to the separability of the extension $K_0\subset K(y)$ (\emph{loc.cit.}, \S 7, No. 3, Remarques, 2)).
\end{proof}
\begin{corollary}\label{ext} The extension $K_0\subset K$ of the fields of fractions of $R_0$ and $R$ is separable.
\end{corollary}
\begin{proof} By proposition \ref{sepsep}, each element $z\in K$ is a quotient of two elements of $R$ which are both separable over $K_0$ in view of proposition \ref{sepsep}.
\end{proof}
\begin{corollary} \label{fingenn} Let $\nu$ be a rational Abhyankar valuation of a complete equicharacteristic noetherian local domain. Assume that there are $r={\rm dim}R$ elements $x_i\in R$ whose values are rationally independent and such that the injection $k[[x_1,\ldots ,x_r]]=R_0\subset R$ makes $R$ into a finite $R_0$-module, with a separable fields extension $K_0\subset K$. Then the semigroup of values of the valuation $\nu$ on $R$ is finitely generated.
\end{corollary}
\begin{proof} By corollary \ref{ext}, we can choose a primitive element $y\in R$ for the separable extension $K_0\subset K$ and by proposition \ref{finhyp} the semigroup $\Gamma_1$ of $\nu$ on $R_1=R_0[y]/(p(y))$ is finitely generated. The inclusion $R_1\subset R$ is integral since $R$ is integral over $R_0$, and birational because $y$ is a primitive element, so that there is a conductor, an element $f\in R_1$ such that $fR\subset R_1$. This implies that there is a $\gamma\in \Gamma_1$ such that $\gamma+\Gamma\subset\Gamma_1$ and we can apply proposition \ref{transl} to deduce that $\Gamma$ is finitely generated.
\end{proof}
\begin{remark}\label{keyrem}
\small{\emph{Let $p$ be a prime number. If we consider the numerical semigroup $$\Gamma= \langle p^3,\  p^3+p^2,\ p^4+p^3+p^2+p, \ p^5+p^4+p^3+p^2+p+1\rangle$$ as in (\cite{Te1}, 6.3)\footnote{This reference is a development of the study of Example 3.5.4, p. 114 of \cite{C}.}, and we choose a field $k$ of characteristic $p$, then $\Gamma$ is the semigroup of a plane branch defined parametrically by $x=t^{p^3},\ y=t^{p^3+p^2}+t^{p^3+p^2+p+1}$, and implicitly by a unitary polynomial of degree $p^3$ in $y$ with coefficients in $k[[x]]$. With the current notations, this polynomial can be obtained by eliminating $u_2,u_3$ between three equations which are: $y^p-x^{p+1}-u_2=0, u_2^p-y^{p^2+1}-u_3=0, u_3^p-y^{p^3}u_2-x^{p+1}u_2^{p^2}=0$; compare with \emph{loc.cit.}  From the viewpoint taken here it is the last generator of $\Gamma$ which must be chosen because it is prime to $p$. But then we are no longer in the situation of proposition \ref{nokey} because the corresponding element of the ring of the curve is not part of a minimal system of generators of the maximal ideal, and a new primitive element must be chosen, corresponding to another curve.}\par \emph{The situation may be summed up by saying that once we have re-embedded, using the valuative Cohen theorem, our plane curve in the space spanned by the associated monomial curve, new projections to coordinate axis are available, whose kernel may have high contact with the curve, but which are separable \emph{and tame}, and then a suitable projection of our curve to a new plane is a plane curve for which this axis is a coordinate axis with separable \emph{and tame} projection, and which is birationally equivalent to our original curve. This is not in the spirit of the classical approaches to resolution of singularities since in this operation the multiplicity of the plane curve considered may increase in a way that is not controllable by classical invariants. What we have done above in the proof of the quasi finite generation of semigroups of Abhyankar valuations is a generalization of this.}}

\end{remark}

 \subsection{Abhyankar valuations and quasi monomial valuations}\label{results}
 We have already defined birational $\nu$-modifications and toric $\nu$-modifications $R\to R'$ of our local ring $R$. In the case where $R$ is complete and our valuation is rational and Abhyankar, let us agree to call with the same name the morphism $R\to R"$ obtained by completing $R'$ and extending the valuation to the quotient of $\hat{R'}$ by a minimal prime, with the same semigroup, according to subsection \ref{extcomp}.
 \begin{definition} We say that the semigroup of values of a valuation $\nu$ on a local domain $R$ which is dominated by the valuation ring $R_\nu$ of $\nu$  is \emph{quasi finitely generated} if there exists a birational $\nu$-modification $R\to R'$ such that the semigroup of values of $\nu$ on $R'$ is finitely generated.
\end{definition}
\begin{theorem}\label{AbhFin} The semigroup of a rational valuation of a complete equicharacteristic noetherian local ring $R$ with algebraically closed residue field is quasi finitely generated if and only if the valuation is Abhyankar.
\end{theorem}
\begin{proof} We have seen the "only if" part at the beginning of the section.\par\noindent
By propositions \ref{sepex} and \ref{sepsep} we can after a $\nu$-modification assume that $R$ is a finite extension of $R_0=k[[x_1,\ldots, x_r]]$ with a separable fraction fields extension. We now apply corollary \ref{fingenn}.\end{proof}

\begin{corollary}\label{finfin} The semigroup of values of an Abhyankar valuation $\mu$ of a complete equicharacteristic noetherian local ring $R$ with algebraically closed residue field is quasi finitely generated.
\end{corollary}
\begin{proof} Using \S 3.6 of \cite{Te1}, remark \ref{existab} and proposition \ref{compAb}, which applies because complete noetherian local rings are catenary, we see that there is a rational Abhyankar valuation $\nu$ of $R$ which is composed with $\mu$. To uniformize $\mu$ it is sufficient to uniformize $\nu$ since if the excellent ring $R'$ obtained from $R$ by a birational modification is regular at the center of $\nu$ it must be regular at the center of $\mu$, and once we have made the semigroup of $\nu$ finitely generated, the semigroup of $\mu$ is an image of the semigroup of $\nu$.
\end{proof}
\begin{definition}\label{defqm} A valuation on $R$ is said to be quasi monomial if there is a birational $\nu$-extension $R\to R'$ where $R'$ is a regular local ring and there is a system of generators of its maximal ideal with respect to which $\nu$ is a monomial valuation.
\end{definition}
A rational quasi monomial valuation is obviously Abhyankar and its semigroup is quasi finitely generated. The conjunction of corollary \ref{finfin}, proposition \ref{resoverwght} and proposition \ref{OWD}, with the addition of proposition \ref{quasim}, show the:
\begin{proposition}\label{AQ} Any rational Abhyankar valuation of a noetherian complete equicharacteristic local domain with an algebraically closed residue field is quasi monomial.\hfill\qedsymbol\end{proposition}
\begin{corollary}\label{anyAb} Any Abhyankar valuation $R\subset R_\nu$ of an an excellent equicharacteristic local domain with an algebraically closed residue field can be uniformized by a birational $\nu$-modification: there exists a regular local ring $R'$, essentially of finite type over $R$ and dominated by $R_\nu$. If it is rational, it is also quasi monomial.
\end{corollary}
\begin{proof} Let $R$ be an excellent equicharacteristic local domain with an algebraically closed residue field and let $\mu$ be an Abhyankar valuation of $R$. By the same argument as in the proof of corollary \ref{finfin}, which applies because excellent local domains are catenary, we obtain a rational Abhyankar valuation $\nu$ of $R$ which is composed with $\mu$. Using the fact proved in (\cite{HOST}, Lemma 7.3) that sufficiently far in the tree of $\nu$-modifications of $R$, the local ring $R'$ becomes analytically irreducible and proposition \ref{Abext}, after renaming $R'$ into $R$ we can extend $\nu$ to an Abhyankar valuation $\hat\nu$ of $\hat R^m$ \emph{with the same semigroup}. Now we can apply Theorem \ref{AbhFin} to the valuation $\hat\nu$ on $\hat R^m$ and obtain a complete $\nu$-modification of $R$ with a finitely generated semigroup. The blowing-ups we make are dominated by monomial blowing-ups in the variables generating the maximal ideals of the complete local rings, but they are obtained by completion from the blowing-ups of the same monomial ideals in the rings we consider before completion. So the reduction of the general case of an Abhyankar valuation on $R$ to the case where $\hat R^m$ is a finite module over a power series ring with a separable and tame fraction fields extension is achieved by blowing-up an ideal in $R$, localizing at the point picked by the valuation, and taking the completion.  Similarly, once we have thus reached the situation where the semigroup is finitely generated, the toric map in the coordinates $(\xi_i)$ which uniformizes $\hat\nu$ according to Theorem \ref{LU} is an algebraic map for $R$ since in view of the isomorphism ${\rm gr}_\nu R\simeq {\rm gr}_{\hat\nu}\hat R^{(\nu)}$ we can choose, as representatives of the generators $\overline\xi_i$ of ${\rm gr}_{\hat\nu}\hat R^{(\nu)}$, elements $\xi_i\in R$, and the completion of the local ring of the transform of ${\rm Spec}R$ at the point picked by $\nu$ is the (completion of the) transform of $\hat R^{(\nu)}$. After applying Theorem \ref{LU}, the regularity of the transform of $\hat R^{(\nu)}$ implies the regularity of $\hat R'^{m'}$ and hence the fact that the transform of ${\rm Spec}R$ is regular at that point. Here we use the fact that $R$ and its transforms are excellent, that if $R$ is analytically irreducible and sufficiently far in the tree of $\nu$-modifications, so are its $\nu$-modifications and finally that in this case an Abhyankar valuation extends uniquely to the completion, with the same semigroup. Alternatively, we might have used this last fact and the argument of (\cite{Te1}, \S 7).\end{proof}\noindent
\begin{Remark}\label{weak}
\small{\begin{enumerate} \item Local uniformization for Abhyankar valuations of rank one of algebraic function fields in characteristic zero was proved by Dale Cutkosky (see \cite{ELS}, proposition 2.8 and \cite{JM}, Prop. 3.7) as a consequence of embedded resolution of singularities. It can also be deduced, in arbitrary characteristic and for algebraic function fields separable over the base field and arbitrary Abhyankar valuations, from the main result (Theorem 1.1) of Knaf-Kuhlmann (\cite{KK1}).
\item Assuming that Conjecture 9.1 of \cite{HOST} is true, or even only the "asterisked proposition" 5.19 of \cite{Te1}, for a ring $R$ as in corollary \ref{anyAb}, there is an ideal $H$ in $\hat R^m$ with $H\cap R=(0)$ such that $\nu$ extends to a valuation $\hat\nu_-$ on $\hat R^m/H$ with the same value group. We defined in \cite{Te1}, Remarks 7.3, the fact that $R$ is \emph{weakly Abhyankar} by the equality ${\rm dim}\hat R^m/H={\rm rat.rk.}\nu$. The proof of \cite{Te1}, 7.11 shows that local uniformization of Abhyankar valuations of complete local rings implies local uniformization of weakly Abhyankar valuations.
\end{enumerate}}
\end{Remark}
In the special case of rational Abhyankar valuations we have the ungraded analogue of corollary \ref{grLU}:
\begin{theorem}\label{LUTOR}{\rm (The toroidal nature of rational Abhyankar valuations)} Let $R\subset R_\nu$ be a rational Abhyankar valuation of an excellent equicharacteristic local domain $R$ of dimension $r$ with algebraically closed residue field $k$. Assume that $R$ contains a field of representatives of $k$. Choose $r$ homogeneous elements $x_1^{(0)},\ldots ,x_r^{(0)}\in {\rm gr}_\nu R$ whose valuations are rationally independent. The ring $R_\nu$ is the union of a nested family indexed by $\N$ of local domains $R^{(h)}$ essentially of finite type over $R$, which are regular for $h\geq 1$:
$$R=R^{(0)}\subset R^{(1)}\subset\ldots \subset R^{(h)}\subset R^{(h+1)}\subset \ldots\subset R_\nu ,$$ where each inclusion $R^{(h)}\subset R^{(h+1)}$, including the first one $R^{(0)}\subset R^{(1)}$, is obtained by localizing at the point picked by the valuation a birational map of finite type, which for $h\geq 1$ is monomial with respect to suitable minimal sets of generators of the maximal ideals of $R^{(h)}$ and $R^{(h+1)}$, and the associated graded rings ${\rm gr}_\nu R^{(h)}$ for $h\geq 1$ form a nested system of polynomial rings in $r$ variables over $k$ containing $k[x_1^{(0)},\ldots ,x_r^{(0)}]$, where the inclusion maps
$$k[x_1^{(0)},\ldots ,x_r^{(0)}]\subset \ldots \subset k[x_1^{(h)},\ldots ,x_r^{(h)}]\subset k[x_1^{(h+1)},\ldots ,x_r^{(h+1)}]\subset\ldots\subset {\rm gr}_\nu R_\nu\eqno{(gr)}$$
send each variable to a term, and the inclusions are birational, with the possible exception of the first one, which is composed: $k[x_1^{(0)},\ldots ,x_r^{(0)}]\subset {\rm gr}_\nu R\subset k[x_1^{(1)},\ldots ,x_r^{(1)}]$.
\end{theorem}
\begin{proof} Using corollary \ref{anyAb}, let us begin with a birational $\nu$-modification $R\to R^{(1)}$ such that $R^{(1)}\subset R_\nu$ is regular and $\nu\vert R^{(1)}$ is monomial in coordinates $\tilde x_1^{(1)},\ldots ,\tilde x_r^{(1)}$. Define the $x_i^{(1)}$ as the $\nu$-initial forms of the $\tilde x_i^{(1)}$. Now use the inductive limit presentation of ${\rm gr}_\nu R_\nu$ given by corollary \ref{grLU} and, starting from $h=1$,  inductively choose representatives $\tilde x_i^{(h+1)} \in R_\nu$ of the $x_i^{(h+1)}$ in such a way that they satisfy the same equations $\tilde x_i^{(h)}-\mu_i^{(h)}(\tilde x^{(h+1)})^{b_i^{(h)}}=0$, with $\mu_i^{(h)}\in k^*$ and $b_i^{(h)}\in\N^r$, as their images in ${\rm gr}_\nu R$ satisfy. To achieve this, it suffices to use the fact that the inclusions of polynomial rings are birational for $h\geq 1$ because $\nu\vert R^{(1)}$ is monomial (see also remark \ref{unimod}) so that the matrix of the vectors $(b_i^{(h)})_{1\leq i\leq r}$ is unimodular and the $x^{(h+1)}_i$ are Laurent terms (=constant times a Laurent monomial) in the $x_i^{(h)}$. Since $k$ has a field of representatives in $R$, we can use  the same expression to define the $\tilde x_i^{(h+1)}$ in terms of the $\tilde x_i^{(h)}$. Define inductively $R^{(h+1)}$ for $h\geq 1$ to be the $R^{(h)}$-subalgebra of $R_\nu$ obtained by localizing the subalgebra $R^{(h)}[\tilde x_1^{(h+1)},\ldots ,\tilde x_r^{(h+1)}]$ of $R_\nu$ at the maximal ideal which is its intersection with $m_\nu$. Since $R^{(h+1)}$ is a localization at the origin of $R^{(h)}[X_1,\ldots ,X_r]/(\tilde x_i^{(h)}-\mu_i^{(h)}X^{b_i^{(h)}})_{1\leq i\leq r}$, we see that the $R^{(h+1)}$ are regular local rings with coordinates $\tilde x_1^{(h+1)},\ldots ,\tilde x_r^{(h+1)}$. \par
As a consequence of proposition \ref{finapp}, given $h\in \N\setminus\{0\}$, and two terms in $k[x_1^{(h)},\ldots ,x_r^{(h)}]$, there is an $h'\geq h$ such that the image in $k[x_1^{(h')},\ldots ,x_r^{(h')}]$ of one of the two terms becomes a multiple of the image of the other. Therefore the same is true of any two monomials $(\tilde x^{(h)})^m$, $(\tilde x^{(h)})^n$ of $R^{(h)}$. We can write any element of $R^{(1)}$ as a series in $x_1^{(1)},\ldots ,x_r^{(1)}$ in $\hat R^{(1)}$. The ideal of monomials appearing in this series is finitely generated, and so it becomes principal in some $\hat R^{(h)}$, so that our element of $ R^{(1)}$ can be written $(\tilde x^{(h)})^EU(\tilde x^{(h)})$ in $\hat R^{(h)}$, with $U(\tilde x^{(h)})$ a unit. By (\cite{B3}, Chap. III, \S 3, no.5, Corollaire 4) it follows from this that $U((\tilde x^{(h)})\in R^{(h)}$, and finally that the union of the $R^{(h)}$'s is a valuation ring, and since its value group is $\Phi$ it has to be $R_\nu$.\end{proof}\noindent
\begin{Remark}\label{toroidal}\small{\begin{enumerate} \item We cannot say that ${\rm gr}_\nu R$ stands at the left of the inclusions $(gr)$ because we do not know that its semigroup is finitely generated. Before we can do that we have to make a $\nu$-modification $R\to R'$ of $R$.\Par One may hope that the semigroup of a rational Abhyankar valuation of an equicharacteristic excellent noetherian local domain with an algebraically closed residue field is always finitely generated, so that the valuation has an embedded local uniformization given by a toric map with respect to suitable generators of the maximal ideal, whose $\nu$-initial forms generate the $k$-algebra ${\rm gr}_\nu R$.\par\medskip
\item Even in the case where the ring $R$ is regular, the first inclusion $R\subset R^{(1)}$ may be necessary to make the valuation monomial. For a regular two dimensional local ring, theorem \ref{LUTOR} reduces to the fact that in the sequence of blowing-ups of centers of a rational Abhyankar valuation, after finitely many steps there are only satellite points, which means that  from there on the sequence of blowing-ups is toroidal, and this characterizes rational Abhyankar valuations (see \cite{Favre-Jonsson}, 6.2).\par\medskip
\item This theorem states in a precise way that rational Abhyankar valuations are those which are toroidal in nature, which is somewhat more precise than "quasi monomial". In the toric world, say over a field $k$, the analogues of valuations are (additive) preorders on the lattice $\Z^r$, and they are all "Abhyankar": see \cite{Ewald-Ishida} and \cite{GP-T2}, \S 13. The analogue of the valuation ring is the semigroup algebra over $k$ of the non-negative part $\Z^r_{\geq 0}$ with respect to the preorder. The theorem gives us a relationship between the nature of the valuation as expressed by $\Phi_{\geq 0}$ and specific sequences of birational toroidal modifications.
\end{enumerate}}
\end{Remark}
Going back to Theorem \ref{LUTOR}, if we assume that $R$ is analytically irreducible, we have injections $\hat R^m\subset k[[\tilde x_1^{(h)},\ldots ,\tilde x_r^{(h)}]]$ and we may view this last ring, say for $h=1$, as $k[[t^{\N^r}]]$, which is itself a subalgebra of $k[[t^{\Phi_{\geq 0}}]]$, with $\Phi=\Z^r$. The valuation on $k[[\tilde x_1^{(h)},\ldots ,\tilde x_r^{(h)}]]$ and therefore also the valuation $\hat\nu$ on $\hat R^m$ extending the valuation $\nu$ on $R$, is induced by the $t$-adic valuation of $k[[t^{\Phi_{\geq 0}}]]$. Up to multiplication by a nonzero constant, the image in $k[[t^{\N^r}]]$ or $k[[t^{\Phi_{\geq 0}}]]$ of each $\xi_i\in R$ is of the form $\xi_i (t)=t^{\gamma_i}+\sum_{\delta>\gamma_i}c^{(i)}_\delta t^\delta$ with $c^{(i)}_\delta\in k$, and we may view $\hat R^m$ as the image of the map $\widehat{k[(u_i)_{i\in I}]}\to k[[t^{\Phi_{\geq 0}}]]$ determined by $u_i\mapsto \xi_i (t)$.\Par
\emph{We can say that just like in the case of curves (see \cite{Te0}), the formal space corresponding to $R$ is obtained by deforming the parametrization $u_i\mapsto t^{\gamma_i}$ of the formal space corresponding to $\Gamma$, inside the space with coordinates $(u_i)_{i\in I}$.} \par In particular the valuation $\nu$ is induced by the embedding $R\subset \hat R^m\subset k[[t^{\Phi_{\geq 0}}]]$. According to what we saw before Lemma \ref{rat} or in \S 13 of \cite{GP-T2}, if the order on $\Phi\simeq \Z^r$ is of rank (or height) $h$ it is induced by an embedding $\Phi\subset (\R^h)_{\rm lex}$.\par
Putting together these remarks and the fact that we have the inequality $h\leq {\rm dim}R$ and for $h\leq h'$ natural convex inclusions $(\R^h)_{\rm lex}\subset (\R^{h'})_{\rm lex}$ we see that we have proved the following result\footnote{Belonging to a tradition in valuation theory which goes back to Ostrowski in \cite{O}. A recent result for valuations of rank one on complete regular local rings, including the mixed characteristic case and making significant use of key polynomials, is due to San Saturnino in \cite{SSa}.}, which partially answers a question of D.A. Stepanov.  Here the ring $k[[t^{(\R^{{\rm dim}R})_{\rm lex,+}}]]$ is endowed with its natural $t$-adic valuation.
 \begin{proposition}\label{Kap} Let $R$ be an excellent equicharacteristic local domain with algebraically closed residue field $k$. Every rational Abhyankar valuation $\nu$ of $R$ is induced by an injective map $f_\nu\colon R\longrightarrow k[[t^{(\R^{{\rm dim}R})_{\rm lex,+}}]]$ of local rings such that the rank $h$ of the valuation determines the smallest subring $k[[t^{(\R^h)_{\rm lex,+}}]]$ containing its image.\hfill {\qedsymbol}
 \end{proposition}
Any local injective map defines a rational valuation on $R$ but if it is Abhyankar the same valuation can be defined via an injection satisfying the condition of the proposition.\par\noindent
Stepanov asks\footnote{See \cite{St} for motivation and results.} whether \emph{given a field $k$ with a rank one valuation $\nu$ and a $k$-algebra $R$ there exists an extension $K$ of the field $k$ and a valuation $\mu\colon K\to \R\cup\{\infty\}$ extending $\nu$ such that for any valuation $v$ on $R$ extending $\nu$ there exists a morphism $f_v\colon R\to K$ of $k$-algebras such that the valuation $v$ is induced from the valuation $\mu$ on $K$ via the morphism $f_v$.} \Par
This embedding of $R$ in a power series (or Hahn) ring is of course related to the classical theorem of Kaplansky (see \cite{K}) about embeddings of valued fields in fields of power series. It would be interesting to compare the proof given here with the recent constructive proof of San Saturnino (see \cite{SSa}) for valuations of rank one of complete regular local rings without the equicharacteristic condition, using key polynomials.
\section{An example of G. Rond, the defect, analytic irreducibility and the structure of the semigroups of hypersurfaces}\label{precisions}
In this section, we give applications of proposition \ref{nokey} to the study of the extensions of a rational monomial valuation $\nu_0$ on $k[[x_1,\ldots ,x_r]]$ to $k[[x_1,\ldots ,x_r]][y]/(p(y))$, where $p(y)$ is an irreducible unitary polynomial. We show that some classical results for plane branches can be understood in this context.
\begin{example}\label{rond} {\rm The following example is due to Guillaume Rond}: Let $k$ be an algebraically closed field of characteristic $\neq 2$. Set $R_0=k[[x_1,x_2]]$, equip it with a monomial valuation $\nu_0$ of rank one such that $w_2=\nu_0 (x_2)>\nu_0(x_1)=w_1$, $w_2\notin\langle w_1\rangle$ and $3w_2-w_1\notin\langle w_1,w_2\rangle$.  Consider the irreducible polynomial $p(y)=y^2-x_1^2-x_2^3\in R_0[y]$. If $\nu$ is a valuation on $R_0[y]/(p(y))$ which extends the valuation $\nu_0$ then $\nu(y)=\nu_0(x_1)$ and either the value of $y-x_1$ or the value of $y+x_1$ has to be equal to $3w_2-w_1$, which is in $\Phi_{0,\geq 0}$ but not in $\Gamma_0$. Here we have the equality $\Phi=\Phi_0$ while the degree of $p(y)$ is equal to two. In this case the polynomial $p(y)$ becomes reducible in the completion (or in the henselization) of the valuation ring $R_{\nu_0}\subset k((x_1,x_2))$ of $\nu_0$, with roots $y=\pm x_1\sqrt{1+\frac{x_2^3}{x_1^2}}$, and there are two extensions of $\nu_0$ to $K_0(y)$, so that indeed the extension is defectless in view of Ostrowski's ramification formula (see \cite{Ku}, \cite{Roq} and \cite{V2}):
$$[K_0(y):K_0]=\sum_{\nu_s}e(\nu_s/\nu_0)f(\nu_s/\nu_0)d(\nu_s/\nu_0),\eqno{(O)}$$
where $\nu_s$ runs through the extensions of $\nu_0$, and the integers $e(\nu_s/\nu_0)$, $f(\nu_s/\nu_0)$, and $d(\nu_s/\nu_0)$ are respectively the ramification index $[\Phi_s:\Phi_0]$, the degree of the residual extension, and the defect of the extension of valuations. Here for both extensions we have $e(\nu_s/\nu_0)=f(\nu_s/\nu_0)=1$. The semigroup of values of both valuations on the ring $R_0[y]/(p(y))$ is the semigroup minimally generated by $w_1, w_2, 3w_2-w_1$, the last one being the valuation of $y\pm x_1$ depending on the valuation chosen, so that it is not true in this case that all the $s_i$ are zero (see proposition \ref{semg}). The kernel of the map $\widehat{k[x_1,x_2, u_1,u_2]}\to k[[x_1,x_2,y]]/(p(y))$ provided by the valuative Cohen theorem for the valuation which gives value $3w_2-w_1$ to $y-x_1$ is generated by the overweight deformation $F_1=u_1-x_1-u_2=0,\ F_2= 2x_1u_2-x_2^3+u_2^2$ of the binomial ideal $(u_1-x_1, 2x_1u_2-x_2^3)$. For the valuation which gives value $3w_2-w_1$ to $y+x_1$ (and value $w_1$ to $y-x_1$) , the generators are $u_1+x_1-u_2, 2x_1u_2+x_2^3-u_2^2$. For both extensions of $\nu_0$ we must have $d(\nu_s/\nu_0)=1$: they are \emph{defectless}.
\end{example}
 Let us go back to our general notations, with $R_0=k[[x_1,\ldots ,x_r]]$ and $\nu_0$ a valuation such that the $\nu(x_i)$ are rationally independent. Denote by $R_{\nu_0}\subset K_0=k((x_1,\ldots ,x_r))$ the valuation ring of $\nu_0$ and by $\tilde R_{\nu_0}\subset \tilde K_0$ the valuation ring of a fixed henselization $(\tilde K_0, \tilde{\nu_0})$ of the valued field $(K_0,\nu_0)$. Let $\overline{\tilde K_0}$ be an algebraic closure of $\tilde K_0$ equipped with the unique extension $\overline{\tilde{\nu_0}}$ of $\tilde{\nu_0}$. It is classical \footnote{See \cite{Ri2}, 6.2 for the special case of valuations of rank one; see also \cite{Ri3}, Chapitre F and Chapitre G, proposition 2.} that the extensions of $\nu_0$ to a finite algebraic extension $K_0[y]/(p(y))$ are obtained by restriction of  $\overline{\tilde{\nu_0}}$ via $K_0$-injections of fields $K_0[y]/(p(y))\subset\overline{\tilde K_0}$ given by $y\mapsto\rho(y)\in \overline{\tilde K_0}$ and two injections $\rho, \rho'$ define the same extension of $\nu_0$ if and only if $\rho (y)$ and $\rho' (y)$ are roots of the same irreducible factor of $p(y)$ in $\tilde K_0[y]$. Indeed, if $\rho (y)$ and $\rho' (y)$ determine the same extension of $\nu_0$ to $K_0[y]/(p(y))$, they determine the same henselizations of this field, so that $\tilde K_0(\rho (y))$ and $\tilde K_0(\rho' (y))$ must be isomorphic by a unique $K_0$-isomorphism, which sends $\rho (y)$ to $\rho' (y)$ since they are the images of $y$. Therefore they must also be $\tilde K_0$-isomorphic (we could also use the fact that $K_0$ is scalewise dense in $\tilde K_0$ in the sense of \cite{Te1} according to \cite{Ku3}) and so $\rho (y)$ and $\rho' (y)$ must be roots of the same irreducible polynomial in $\tilde K_0[y]$. It is also classical that the valuation rings of the extensions of $\nu_0$ all contain the integral closure of $R_{\nu_0}$ in $K_0(y)/(p(y))$ and therefore also contain $R_0[y]/(p(y))$.\Par
  \begin{definition}\label{hensel} We say that a polynomial $q(y)\in K_0[y]$ (or $R_{\nu_0}[y]$) is \emph{$\nu_0$-analytically irreducible} if it is irreducible in $\tilde K_0[y]$ (or $\tilde R_{\nu_0}[y]$). By the universal property of henselization, this property does not depend upon the choice of the henselization. If the valuation is of rank one and $p(y)$ defines a separable extension of $K_0$, it is the same as being irreducible over the completion of $(K_0,\nu_0)$. If $r=1$ it is the same as being irreducible in $k((x))[y]$. The example of G. Rond given at the beginning of this section is irreducible in $k((x_1,x_2))[y]$ but not $\nu_0$-analytically irreducible.
 \end{definition}
 Let us now denote by $\mathopen H_i^\circ,\ 1\leq i\leq t,$ the equations of proposition \ref{nokey} where we have omitted the summand $g_i$ for each $i$, that is
 $$H_i^\circ=x^{s_i}u_i^{n_i}-\lambda_i x^{r_i}\prod_{k\in E(i)}u_k^{t^{(i)}_k}-u_{i+1},\eqno{(E^\circ)}$$
 and by $Q^\circ_i$ the polynomials in $y:=u_1$ obtained by elimination of the $u_k$, $2\leq k\leq i$ between the $H_j^\circ$ for $2\leq j\leq i$, forgetting the variables of index $\geq i+1$. By proposition \ref{resoverwght}, the overweight deformation described by the $H_i^\circ$ gives rise to a valuation $\nu^\circ$ on $k[[x]][y]/(Q^\circ_t(y))$.
 \begin{lemma}\label{order} With the notations just defined, we have the following:
 \begin{enumerate}
 \item For $1\leq j\leq t-1$, the coefficient of lowest value of the polynomial $Q^\circ_{j+1}(y)\in R_0[y]$ is the coefficient of the highest power of $y$, which is $y^{n_1\ldots n_j}$ and it is the only one with this value, which is $S_{j+1}(1):=s_j+n_js_{j-1}+n_jn_{j-1}s_{j-2}+\cdots +n_j\ldots n_3s_2+n_j\ldots n_3n_2s_1=s_j+n_jS_j(1)$.
\item For $1\leq j\leq t-1$ the valuation of $Q^\circ_{j+1}(0)\in R_0$ is $s_j+n_js_{j-1}+n_jn_{j-1}s_{j-2}+\cdots +n_j\ldots n_3s_2+n_j\ldots n_3n_2r_1$ and the term with the lowest valuation comes from $x^{s_j}Q^{\circ n_j}_j(y)$.
\end{enumerate}
 \end{lemma}
Given the binomials $x^{s_i}u_i^{n_i}-\lambda_ix^{r_i}\prod_{k\in E(i)}u_k^{t^{(i)}_k}$ we define two sequences:\par\noindent$\bullet$ The sequence of integers $T_i(i-k)$ is defined inductively for $0\leq k\leq i-1$ by:\par\noindent $T_i(i)=n_i, T_i(i-1)=n_in_{i-1}-t^{(i)}_{i-1},\ldots ,T_i(i-k-1)=T_i(i-k)n_{i-k-1}-t^{(i)}_{i-k-1},\ T_i(j)=0 \ {\rm if}\ j\notin\{1,\ldots ,i\}$.\par\noindent $\bullet$ The sequence $L_i(i-k)$ of elements of $Z^r$ is defined inductively for $0\leq k\leq i-1$ by:\par\noindent
$L_i(i)=s_i, L_i(i-1)=T_i(i)s_{i-1}+s_i,\ldots ,L_i(i-k-1)=T_i(i-k)s_{i-k-1}+L_i(i-k),\ L_i(j)=0\  {\rm if}\ j\notin\{1,\ldots ,i\}$.
\begin{lemma}\label{estim} We have the equality
$$\nu^\circ(x^{r_i})=\nu^\circ(x^{L_i(1)}u_1^{T_i(1)})+\sum_{k=0}^{i-2}\nu^\circ\bigl((\frac{u_{i-k}}{x^{s_{i-k-1}}u_{i-k-1}^{n_{i-k-1}}})^{T_i(i-k)}\bigr)$$
\end{lemma}
\begin{proof} Start from the equality $\nu^\circ(x^{s_i}u_i^{n_i})=\nu^\circ (x^{r_i}\prod_{k\in E(i)}u_k^{t^{(i)}_k})$ and write the second term $\nu^\circ (x^{r_i}\prod_{k=1}^{i-1}u_k^{t^{(i)}_k})$ by adding the required number of $t^{(i)}_k$ that are zero. Then replace the first term by $\nu^\circ(x^{s_i}(x^{s_{i-1}}u_{i-1}^{n_{i-1}})^{n_i})$, adding the correction term $\nu^\circ\bigl((\frac{u_i}{x^{s_{i-1}}u_{i-1}^{n_{i-1}}})^{n_i}\bigr)$. We can simplify the powers of $u_{i-1}$ to obtain the equality:
$$\nu^\circ (x^{r_i}\prod_{k=1}^{i-2}u_k^{t^{(i)}_k})=\nu^\circ (x^{s_i+n_is_{i-1}}u_{i-1}^{n_{i-1}n_i-t^{(i)}_{i-1}})+\nu^\circ\bigl((\frac{u_i}{x^{s_{i-1}}u_{i-1}^{n_{i-1}}})^{n_i}\bigr),$$
and repeat the operation with $u_{i-1}^{n_{i-1}n_i-t^{(i)}_{i-1}}$ and so on. The inductive definition of $T_i(i)$ and $L_i(i)$ follows this process.\end{proof}

\begin{corollary}\label{strineq} We have the inequality $r_i>L_i(1)$, which is stronger than the inequality $r_i> s_i$ of proposition \ref{ineq} except if $s_1=\cdots =s_{i-1}=0$.\qed
\end{corollary}
 Let us define for each $i$ two more sequences as follows:\par\noindent
 $\bullet$ The sequence of positive integers $M_i(i-k)=n_in_{i-1}n_{i-2}\ldots n_{i-k}, \ M_i(j)=1\  {\rm if}\ j\notin\{1,\ldots ,i\}$.\par\noindent
 $\bullet$ The sequence of elements of $\N^r$ defined inductively by $S_i(i-1)= s_{i-1}, S_i(i-2)= n_{i-1}s_{i-2}+s_{i-1}, \ldots ,S_i(i-k-1)=M_{i-1}(i-k)s_{i-k-1}+S_i(i-k), \ S_i(j)=0\  {\rm if}\ j\notin\{1,\ldots ,i-1\}$. In particular, $S_1(1)=0$.\par\noindent

One checks easily by induction the identity
  $$M_i(i-k)=T_i(i-k)+\sum_{s=2}^{k}M_{i-s}(i-k)t^{(i)}_{i-s+1} +t^{(i)}_{i-k},$$
  which we can also write, setting $\ell=i-s+1$ and remembering that $M_{i-k-1}(i-k)=1$:
  $$M_i(i-k)=T_i(i-k)+\sum_{\ell=i-k}^{i-1}M_{\ell-1}(i-k)t^{(i)}_\ell .\eqno{(MT)}$$
  We have by definition $L_i(i)=S_{i+1}(i)=s_i$ and $L_i(i-1)=n_is_{i-1}+s_i=S_{i+1}(i-1)$. and now we can prove by induction on $k$ the equality $$L_i(i-k)+\sum_{\ell=2}^{i-1}t^{(i)}_\ell S_\ell(i-k)=S_{i+1}(i-k).$$
Assuming that it is true for $i,i-1,\ldots, i-k$, to prove the same equality for $i-k-1$, by the inductive definition, amounts to proving the equality $$T_i(i-k)s_{i-k+1}+\sum_{\ell=i-k}^{i-1}M_{\ell-1}(i-k)t^{(i)}_\ell s_{i-k+1}=M_i(i-k)s_{i-k+1},$$ which follows from the equality $(MT)$ above. We can continue until we reach $i-k-1=1$. As a consequence we have the equality$$L_i(1)+\sum_{\ell=2}^{i-1}t^{(i)}_\ell S_\ell(1)=S_{i+1}(1).\eqno{(**)}$$\par
  We are now in position to prove that the coefficient of least valuation of each polynomial $Q^\circ_i(y)$ is the coefficient of its highest degree term and that after division by this coefficient the other coefficients of the resulting unitary polynomials have valuations increasing with $i$.\par\noindent Given a polynomial $q(y)\in R_0[y]$, let us define $N(q(y))$ to be the least valuation of its coefficients. Assume that for $k<i+1$ we have the equality $N(Q^\circ_k(y))= \nu^\circ(x^{S_k(1)})$, which is obvious for $i=1,2$, and let us find the least valuation of a coefficient of the polynomial $x^{r_i}y^{t^{(i)}_1}Q_2^{\circ t^{(i)}_2}\ldots Q_{i-1}^{\circ t^{(i)}_{i-1}}$. According to Lemma \ref{estim} and equality $(**)$, we have
  $$ N(x^{r_i}y^{t^{(i)}_1}Q_2^{\circ t^{(i)}_2}\ldots Q_{i-1}^{\circ t^{(i)}_{i-1}})\geq\nu^\circ(x^{S_{i+1}(1)}u_1^{T_i(1)})+\sum_{k=0}^{i-2}\nu^\circ\bigl((\frac{u_{i-k}}{x^{s_{i-k-1}}u_{i-k-1}^{n_{i-k-1}}})^{T_i(i-k)}\bigr),\eqno{(***)}$$
  and this implies that $N(Q^\circ_{i+1}(y))=\nu^\circ(x^{S_{i+1}(1)})$ and we can continue the induction. To complete the proof of the first statement of the lemma, we check that: \begin{lemma}\label{degdeg} If the degree of $Q^\circ_k$ is $n_1\ldots n_{k-1}$ for $k\leq i-1$, the degree of the polynomial $\prod_{k\in E(i)}Q_k^{\circ t^{(i)}_k}(y)$ is $\leq \sum_{k=1}^{i-1} (n_k-1)n_1\ldots n_{k-1}= n_1\ldots n_{i-1}-1$.\qed \end{lemma}   To prove the second statement of lemma \ref{order}, we observe that by proposition \ref{ineq}, the constant term of the polynomial $Q^\circ_{i+1}(y)$ must come from $x^{s_i}Q_i^{\circ n_i}$. Iterating the construction by elimination shows that it must ultimately come from the term $x^{r_1}$ in $Q^\circ_2(y)$ and gives the result.
  \qed  \par

Let us simplify $S_{i+1}(1)$ into $S_{i+1}$. The polynomials $x^{-S_{j+1}}Q^\circ_{j+1}(y)$ are unitary polynomials in $R_{\nu_0}[y]$. Using the presentation of $R_{\nu_0}$ as an inductive limit of regular local rings given in theorem \ref{LUTOR}, we may assume that they are polynomials in $R_0^{(h)}[y]$ for sufficiently large $h$, with $R_0^{(h)}$ a monomial $\nu_0$-modification of $R_0$ as in theorem \ref{LUTOR}. By construction, in view of equality $(**)$ above, they are the result of elimination of the $u'_j$ for $2\leq k\leq j$ between the equations ${u'}_i^{n_i}-\lambda_i x^{r_i-L_i(1)}\prod_{k\in E(i)}{u'}_k^{t^{(i)}_k}-u'_{i+1}$, where $L_i$ is the function defined after the statement of  lemma  \ref{order} and $u'_{i+1}=x^{-S_{i+1}}u_{i+1}$. A direct computation shows that these are the transforms of our overweight equations which by elimination give the unitary polynomials associated to the $Q^\circ_i(y)$.\par\noindent
We note that the semigroup associated with this new overweight deformation is equal to $\langle \Gamma_0^{(h)}, \gamma_1,\gamma_2-S_2,\ldots ,\gamma_i-S_i,\ldots \gamma_t-S_t\rangle$, where $\Gamma_0^{(h)}$ is the semigroup of $\nu_0$ on $R_0^{(h)}$   and if we want a minimal system of generators, all the indices $i$ with $n_i=1$ disappear.\par\noindent Equality $(**)$ above implies that we can rewrite the image in $R_{\nu_0}[y]$, or in $R_0^{(h)}[y]$ for sufficiently large $h$, of equation $H_i^\circ$ as
$$x^{-S_{i+1}}Q_{i+1}^\circ=(x^{-S_i}{Q_i^\circ})^{ n_i}-\lambda_ix^{r_i-L_i(1)}\prod_{k\in E(i)} (x^{-S_k}Q^\circ_k)^{ t^{(i)}_k}.\eqno{(****)}$$ Using corollary \ref{strineq} we see that \emph{$n_i$ is the smallest integer $k>0$ such that $k(\gamma_i-S_i)\in \langle\Gamma_0^{(h)}, \gamma_1,\gamma_2-S_2,\ldots ,\gamma_{i-1}-S_{i-1}\rangle$.}
\par As a result, we can for each index $j$ define a valuation $\nu^{\circ (j)}$ on \break $\hat R_0^{(h)}[y]/(x^{-S_{j+1}}Q^\circ_{j+1}(y))$ by the overweight deformation described by the first $j$ equations and $Q^\circ_{j+1}=0$. These are approximate pseudo-valuations in the sense of Vaqui\'e (see \cite{V1}).\par\noindent
With the usual convention that $n_0=1$, the degree of $Q^\circ_{j+1}(y))$ is $n_1\ldots n_j$, which is the index of $\Phi_0$ in the group of values of $\nu^{\circ (j)}$, so that each $Q^\circ_{j+1}(y))$ is $\nu_0$-analytically irreducible: by the ramification formula there is only one extension of $\nu_0$ to $K_0[y]/(x^{-S_{j+1}}Q^\circ_{j+1}(y))$.\par
We must add to the simple equations $H_i^\circ$ the "perturbations" $g_i$ in order to obtain the real equations $H_i$, but when all the $s_i$ are zero this does not change the semigroup. This has a geometric interpretation (the geometry will be more apparent in corollary \ref{quasio}):

 \begin{proposition}\label{semg} Let $p(y)$ be a unitary irreducible polynomial in $R_0[y]$ and let $\nu$ be an extension of $\nu_0$ to the field $K=K_0[y]/(p(y))$. Let $p_1(y)\in \tilde R_{\nu_0}[y]$ be the irreducible unitary factor of $p(y)$ corresponding to the extension $\nu$. Then we have:\begin{itemize}
 \item  The valuation $\nu$ has a unique extension $\tilde\nu$ to $\tilde R_{\nu_0}[y]/(p_1(y))$.
 \item If $p(y)$ is $\nu_0$-analytically irreducible, all the $s_i$ are $0$, so all the $n_i$ are $>1$. Conversely, if all the $s_i$ are zero, the polynomial $p(y)$ is $\nu_0$-analytically irreducible and its degree is $[\Phi:\Phi_0]$.
  \item Writing, according to definition \ref{ord}, the semigroup  of $\nu$ on $R=R_0[y]/(p(y))$ as $\Gamma=\langle\Gamma_0,\gamma_1,\ldots ,\gamma_t\rangle$, let $i_1,\ldots, i_l$ be the integers among $\{1,\ldots ,t\}$ such that $n_{i_k}>1$, so that $n_{i_1}\ldots n_{i_l}=[\Phi:\Phi_0]$. The semigroup $\tilde\Gamma$ of $\tilde \nu$ on $\tilde R_{\nu_0}[y]/(p_1(y))$ is equal to the semigroup $\langle \Phi_{0,\geq 0},\gamma_{i_1}-S_{i_1},\gamma_{i_2}-S_{i_2},\ldots ,\gamma_{i_l}-S_{i_l}\rangle$.

 \end{itemize}
 \end{proposition}
  \begin{proof} The first statement is classical and follows from what was recalled above on the classification of extensions. The second statement follows from the first and lemma \ref{order} as follows: the polynomial $p(y)$ is equal to the last polynomial $Q_{t+1}(y)$ obtained by elimination. Since it is irreducible in $\tilde R_{\nu_0}[y]$ all its roots in a splitting extension $L$ of $\tilde K_0$ have the same valuation because the valuation of $\tilde K_0$ has a unique extension to $L$. From lemma \ref{order}, and the fact that in view of the overweight condition adding the pertubations $g_i$ to the equations $H^\circ_i$ of $(E^\circ)$ above does not affect the valuation of the constant term of the polynomials $Q^\circ_i(0)$, we see that the valuation of $Q_{t+1}(0)$ is $s_t+n_ts_{t-1}+n_tn_{t-1}s_{t-2}+\cdots +n_t\ldots n_3s_2+n_t\ldots n_3n_2r_1$. Therefore we must have the equality $n\nu(y)=\nu_0(Q_{t+1}(0))$ with $n={\rm deg}p(y)$. On the other hand, in view of the equation $x^{s_1}u_1^{n_1}-\lambda_1x^{r_1}-g_1(x,u_1)=u_2$, we must have $n_1\nu(y)=r_1-s_1$. The equality of the two expressions for $\nu(y)$ implies $s_1=0$ because there is no negative component in $\nu_0(Q_{t+1}(0))$. If the highest degree term $y^n$ does not come from the initial binomial, the first equality of lemma \ref{order} and the overweight condition yield $s_t+n_ts_{t-1}+n_tn_{t-1}s_{t-2}+\cdots +n_t\ldots n_3s_2+n_t\ldots n_2n_1\frac{r_1}{n_1}<n\frac{r_1}{n_1}$, which contradicts the previous equality. Thus, the term $y^n$ must come from the initial binomial, and in view of the structure of the equations, the only possibility is that $n=n_1\ldots n_t$ and all the $s_i$ are zero, and then all the $n_i$ are $>1$ because of the minimality of the system of generators.\Par
 Conversely, if all the $s_i$ are zero, one can choose each $g_i$ to be a polynomial of degree $<n_i$ in $u_i$. The overweight condition implies that the series in $u_i$ obtained from $H_i+u_{i+1}=u_i^{n_i}-\lambda_i x^{r_i}\prod_{k\in E(i)}u_k^{t^{(i)}_k}-g_i (x,u_1,\ldots ,u_i)$ by setting all the variables equal to zero except $u_i$, is of order $n_i$. The Weierstrass preparation theorem makes $H_i+u_{i+1}$ a unitary polynomial of degree $n_i$ in $u_i$, up to multiplication by a unit. The unit can then be absorbed in $u_{i+1}$ by a change of variable. Iteratively replacing each occurrence of $u_k^{n_k}$ by its expression modulo $H_k$ for $1\leq k<i$ and making appropriate changes of representatives, we can then assume that $g_i(x,u_1,\ldots ,u_i)$ is of degree $<n_k$ in each $u_k$.\par
  A computation which we have already used in the proof of lemma \ref{degdeg} then shows that the polynomial $p(u_1)$ which is obtained by elimination of $u_2,\ldots ,u_t$ between the $H_i$ is of degree $n_1\ldots n_t=[\Phi:\Phi_0]$. Since henselization is an immediate extension (see \cite{Ri3}, Chapitre F, Corollaire 1 du Th\'eor\`eme 3, or \cite{Ku3}, Theorem 1, for a more precise result), the extension $\tilde K_0[\rho (y)]/\tilde K_0$ is of degree $\geq [\Phi:\Phi_0]$. Since $\rho (y)$ is a root of $p(y)$ this implies that $p(y)$ is the minimal polynomial of $\rho (y)$ in this extension, and therefore irreducible in $\tilde K_0[y]$. \par
 The ring $\tilde R_{\nu_0}[y]/(p_1(y))$ is a quotient of a henselization of $R_{\nu_0}[y]/(p(y))$, and we use the presentation of $R_{\nu_0}$ as an inductive limit of regular local rings given  in theorem \ref{LUTOR}. Since henselization commutes with filtering inductive limits (see \cite{EGA4}, 18.6.14. ii)), the ring $\tilde R_{\nu_0}$ is the inductive limit of the henselizations $\tilde R_0^{(h)}$ of the regular local rings $R_0^{(h)}$ of theorem \ref{LUTOR}. We may assume that the polynomial $p_1(y)\in \tilde R_{\nu_0}[y]$ has coefficients in some $\tilde R_0^{(h)}$, so that it is irreducible in $\tilde R_0^{(h)}[y]$. Since $\tilde R_0^{(h)}$ is henselian, the polynomial $p_1(y)$ remains irreducible in $\hat{\tilde R}^{(h)}_0[y]$ where $\hat {\tilde R}^{(h)}_0$ is the $\tilde m_0$-adic completion of $\tilde R_0^{(h)}$, which is in fact the completion $\hat R^{(h)}_0$, of the noetherian local ring $R^{(h)}_0$ (see \cite{EGA4}, th\'eor\`eme 18.6.6), and then we can apply the valuative Cohen theorem.\par
     In general, since a finite algebraic extension of a henselian valued field is henselian, the valued field $(\tilde K_0[y]/(p_1(y)),\tilde\nu)$ is a henselization of $(K_0[y]/(p(y)),\nu)$. Assuming that we have taken $h$ large enough so that $p_1(y)\in \hat R^{(h)}_0[y]$, the map $\hat R^{(h)}_0\otimes_{R_0}k[[x,u_1,\ldots ,u_l]]\to \hat R^{(h)}_0[y]/(p_1(y))$ deduced from the valuative Cohen theorem for $R_1$ is surjective, so that the elements of the second ring are expressed as series in the $Q_i(y)$ with coefficients in $\hat R^{(h)}_0$. Denoting by $x^{(h)}$ the coordinates of $\hat R^{(h)}_0$, we see that any cancellation of initial forms between terms in the $Q_i(y)$ with coefficients in $\hat R^{(h)}_0$ must come, after multiplication by a monomial in $x^{(h)}$, from a cancellation of initial forms in $R_0[y]$ and therefore give rise, when the cancellation has taken place, to a term in the $Q_i(y)$ with coefficients in $\hat R^{(h)}_0$. In other words, denoting by $M$ and $M^{(h)}$ the multiplicative parts in ${\rm gr}_\nu R_0[y]/(p(y))$ and ${\rm gr}_{\tilde \nu} R^{(h)}_0[y]/(p_1(y))$ made of monomials in $X$ and $X^{(h)}$, we have that the natural map $M^{-1}{\rm gr}_\nu (R_0[y]/(p(y)))\to (M^{(h)})^{-1}{\rm gr}_{\tilde \nu} (R^{(h)}_0[y]/(p_1(y)))$ is an isomorphism. So we can take as generators of ${\rm gr}_{\tilde \nu} (R^{(h)}_0[y]/(p_1(y)))$ the $X^{(h)}$ and elements of the form $\overline\eta_i=\frac{\overline\xi_i}{(X^{(h)})^{d_i}}$. Now we remember that the equations between these generators must come from the original equations between the $\overline\xi_i$ and be of the form $\overline\eta_i^{n_i}=\ldots$. The only possibility is that $\overline\eta_i=X^{-S_i}\overline\xi_i$. The $\eta_i$ with $n_i=1$, that is $i\notin\{i_1,\ldots ,i_l\}$, disappear if we want a minimal system of generators. If we denote by $\tilde Q_i(y)$ representatives of the $\overline\eta_i$ in $\hat R^{(h)}_0[y]/(p_1(y))$, they must be obtained by elimination from equations $$\tilde u_i^{n_i}-\lambda_i(x^{(h)})^{r_i}\prod_{k\in E(i)} \tilde u_k^{t^{(i)}_k}-\tilde g(x^{(h)},\tilde u_1,\ldots ,\tilde u_i)-\tilde u_{i+1}=0,$$ where if $i_1>1$ the image $\tilde Q_{i_1}$ of $\tilde u_{i_1}$ is of the form $y-\upsilon (x^{(h)})$ with $\upsilon (x^{(h)})\in \hat R_0^{(h)}$. \end{proof} It would be interesting to have a more precise relationship between $Q_i(y)$ and $\tilde Q_i(y)$ in $\hat R^{(h)}_0[y]/(p_1(y))$.

  \begin{Remark} \small{\begin{enumerate}\item As the proof shows, the third statement of the proposition holds if $\tilde R_{\nu_0}$ is replaced by $\tilde R_0^{(h)}$ with a sufficiently large $h$, and then the semigroup is finitely generated since it is $\langle \Gamma_0^{(h)},\gamma_{i_1}-S_{i_1},\gamma_{i_2}-S_{i_2},\ldots ,\gamma_{i_l}-S_{i_l}\rangle$. In G. Rond's example, since the degree of $p_1(y)$ is one, the semigroup is $\Gamma_0^{(h)}$. Indeed, in this case $\nu (y)=w_1\in \Gamma_0$, $\gamma_1=3w_2-w_1\in \Gamma_0^{(h)}$, where according to (\cite{Te1}, 4.2) we have $\Gamma_0^{(h)}=\langle -p_hw_1+q_hw_2,p_{h+1}w_1-q_{h+1}w_2\rangle$, with $\frac{p_h}{q_h},\frac{p_{h+1}}{q_{h+1}}$ two successive approximants, say with $\frac{p_h}{q_h}<\frac{w_2}{w_1}$, from the continued fraction expansion of $\frac{w_2}{w_1}$, the vectors $(q_h,p_h),(q_{h+1},p_{h+1})$ being close enough in direction to the vector $(w_1,w_2)$ for the point $(-2,3)$ corresponding to the monomial $\frac{x_2^3}{x_1^2}$ and the weight $3w_2-2w_1$ to be in the cone generated by the vectors $(-p_h,q_h)$ and $( p_{h+1},-q_{h+1})$.

 \item For the proof of the second statement, we could have used a Newton polygon argument: if the valuation $\nu_0$ is of rank one, since $\tilde K_0$ is henselian the Newton polygon of the $\nu_0$-analytically irreducible polynomial $p(y)$ must have only one compact face (see \cite{Ri2}), which necessarily is a homothetic of the compact face of the Newton polygon  of $Q_2(y)$ in view of the weight conditions on the equations $H_i$. But since $p(y)$ is unitary this imposes that the monomial $y^n$ is on that compact face, which implies that all $s_k$ are zero and the $Q_k(y)$ are unitary polynomials. The semigroup is generated by the $\nu(x_i)$ and the valuations of the $Q_k$, and the degree of $p(y)$ is equal to $\prod_kn_k=[\Phi:\Phi_0]$ so that in this case there is no defect for the extension $K$ of $K_0$. If the rank of $\nu_0$ is $>1$ we use the remarkable work of Vaqui\'e on key polynomials, and in particular \cite{V3}, \S 3, which tells us that again in that case there is a Newton polygon in a suitable sense which has only one compact face and allows us to reach the same conclusion.
  \end{enumerate}}
\end{Remark}
Let us denote by $Q^u_i(y)\in K_0[y]$ the polynomials $Q_i(y)$ made unitary.
  \begin{proposition}\label{Key}  The polynomials $(Q^u_{i_k}(y))_{1\leq k\leq l}$ are key polynomials for the extension $\nu$ to $K_0[y]/(p(y))$ of the valuation $\nu_0$ of $K_0$.
\end{proposition}
 \begin{proof} Recalling the definition of MacLane extended by Vaqui\'e (see \cite{Mac}, \cite{V1}), the three conditions are that they should be $\nu$-minimal, $\nu$-irreducible, and unitary.  In view of proposition \ref{semg} it suffices to prove the result for the polynomials $\tilde Q_j(y)$ associated to $p_1(y)\in \hat R_0^{(h)}[y]$ so that we may assume that $p(y)$ is $\nu_0$-analytically irreducible. Then, the $Q_i(y)$ are all unitary, and according to proposition \ref{irreducible} they are all irreducible. They are also $\nu$-irreducible since their initial forms minimally generate ${\rm gr}_\nu k[[x]][y]/(p(y))$ and are therefore irreducible elements. The only thing left to prove is that if the $\nu$-initial form of a polynomial $q(y)$ is divisible by the $\nu$-initial form of a polynomial $Q_i(y)$ then ${\rm deg}q(y)\geq
{\rm deg}Q_i(y)$ (resp. ${\rm deg}q(y)\geq
{\rm deg}p(y)$).\Par After what we have seen we can, in $K_0[y]/(p(y))$, write $$q(y)=\sum_{t=1}^n a_t(x)y^{t_1}{(Q^u_2)}^{t_2}\ldots {(Q^u_k)}^{t_k}$$ with $t_j\leq n_j-1$ for $1\leq j\leq k$ and the valuation of each term of the sum is $\geq\nu(q(y))$. Now we can follow the proof of corollary 5.4 given in \cite{PP}, since we have reduced our problem to a very similar situation. This corollary states in particular that the degrees in $y$ of the terms of the sum are all distinct. If one of the terms, corresponding to the initial form, contains the polynomial $Q_i(y)$ it must be of degree $\geq {\rm deg}Q_i(y)$, and this cannot be cancelled by any of the other terms, which are of different degrees.  Finally, by the valuative Cohen theorem the $Q_i(y)$ form a generating sequence for $\nu$ so that their values determine it and in fact the $Q^u_i(y)$ generate the ${\rm gr}_{\nu_0}K_0$-algebra ${\rm gr}_{\nu}K_0[y]/(p(y))$ (Compare with \cite{Ma}). \end{proof}
  \begin{proposition}\label{crit} Let us keep the notations introduced in subsection \ref{Kpol}. Given a unitary irreducible polynomial $p(y)\in R_0[y]$ and a rational Abhyankar valuation $\nu$ on $R_0[y]/(p(y))$ as in proposition \ref{nokey}, with its extension $\mu$ to $k[[x,y]]$, let $H_1,\ldots, H_t, p(u_1)-v$ be the generators of the kernel $F$ of the valuative Cohen map $k[[x,u_1,\ldots, u_t,v]]\to k[[x,y]]$ describing the valuation $\mu$. Up to a change of generators of $F$ and of representatives of the $\overline\xi_i$, there are two possibilities for the last equation $H_t$ in the sequence we have built in the proof of proposition \ref{finhyp}: \par\medskip\noindent - Either the polynomial $p(y)$ is $\nu_0$-analytically irreducible, its degree is $[\Phi:\Phi_0]$, we have $l=t$ and the last equation in the sequence is of the form $$u_t^{n_t}-\lambda_tx^{r_t}\prod_{k\in E(t)} u_k^{t^{(t)}_k}-g_t =v.$$
  In this case, and only in this case, we have $s_i=0$ for all $i$, so that all the polynomials $Q_i(y)$ are unitary, with ${\rm deg}Q_i= n_1\ldots n_{i-1}$ (setting $n_0=1$), we can assume that for each $i$ the series $g_i$ has only terms of degree $<n_k$ in each $u_k$, and $\nu$ is the unique extension of $\nu_0$ to $K_0(y)$. \par\medskip\noindent - Or the degree of $p(y)$ is $>[\Phi:\Phi_0]$ and $p(y)$ is divisible in $\tilde R_{\nu_0}[y]$ by a unitary polynomial $p_1(y)$ of degree $[\Phi:\Phi_0]$. The last equation is of the form
    $$x^{s_t}u^{n_t}_t-\lambda_tx^{r_t}\prod_{k\in E(t)}   u_k^{t^{(t)}_k}-g_t=v.$$
    where the highest degree term $y^n$ of $p(y)$ comes by elimination from a term of one of the $g_j(x,y, Q_2(y),\ldots ,Q_j(y))$.\par
In all cases, the separable extension of valued fields $(K_0,\nu_0)\subset (K_0(y),\nu)$ is defectless.
\end{proposition}
 \begin{proof} The process we have described in the proof of proposition \ref{finhyp} stops when for some index $t$ we reach a situation where $p(u_1)-u_{t+1}\in F$, which means that $$p(y)=x^{s_t}Q_t(y)^{n_t}-\lambda_tx^{r_t}\prod_{k\in E(t)}Q_k^{t^{(t)}_k}(y)-g_t(x,y,Q_2(y),\ldots, Q_t(y)).$$  The first case follows directly from proposition \ref{semg}, where we have seen that $p(y)$ is $\nu_0$-analytically irreducible if and only if all the $s_i$ are zero, and the degree of $p(y)$ is then $n_1\ldots n_t=[\Phi :\Phi_0]$.  \par In the second case, the fact that $y^n$ comes from a  $g_j$ follows from the fact that otherwise all the $s_i$ must be zero since $r_i-s_i>0$ and then we are in the first case.\par
  Since $[K_0(y):K_0]$, which is the degree of $p(y)$, is the sum of the degrees of the irreducible factors $p_s(y)$ of $p(y)$ in $\tilde R_{\nu_0}[y]$, the Ostrowski ramification formula $(O)$ shows that the extension from $K_0$ to $K$ has no defect with respect to $\nu_0$.\Par
 In this case too the semigroup is finitely generated by the valuations of the $x,y,Q_k(y)$ but the $Q_k(y)$ are not necessarily unitary polynomials.  \end{proof}
 \begin{example} In the case where $r>1$, let us assume that the valuation $\nu_0$ is of rank one and consider two equations such as $x^{s_1}y-x^{r_1}=Q_2(y),\ Q_2(y)^{n_2}-x^{r_2}-y^n=p(y)$ satisfying the conditions we have observed above, in particular that $(x^{s_1}u_1-x^{r_1}, u^{n_2}_2-x^{r_2})$ generate a prime binomial ideal with $n_2$ prime to the characteristic of $k$,  that $r_1-s_1$ is in $\Phi_{0+}$, that we have $n\nu(y)=n(r_1-s_1)>r_2$, and $r_2>n_2r_1>n_2s_1$. According to corollary \ref{seeAb}, the polynomial $p(y)$ is irreducible in $k[[x]][y]$ because it corresponds to an overweight deformation of the prime binomial ideal generated by $x^{s_1}u_1-x^{r_1}, u^{n_2}_2-x^{r_2}$, but in $\tilde R_{\nu_0}[y]$ it is divisible by a polynomial of degree $n_2$, of the form $p_1(y)= (y-x^{r_1-s_1})^{n_2}-x^{r_2-n_2s_1}+{\rm terms\  of \ higher\  value}$, as one sees using the Newton polygon of $p(y)$ according to \cite{Ri2}, 5.1, D, E. The semigroup $\Gamma$ of $R_0[y]/(p(y))$ is $\langle \Gamma_0, r_1-s_1, \frac{r_2}{n_2}\rangle$ and the semigroup $\tilde\Gamma$ of $\tilde R_{\nu_0}[y]/(p_1(y))$ is $\langle\Phi_{0,\geq 0}, \frac{r_2}{n_2}-s_1\rangle$.
 \end{example}
\begin{Remark}\label{qroot} \small{\begin{enumerate}
\item In the example of G. Rond, we have $t=2=l+1$ and Vaqui\'e's method builds, for each of the two extensions of $\nu_0$, an infinite continuous family of key polynomials which has the polynomial $p(y)$ as limit key polynomial.\par Each family corresponds to the successive truncations $$T_e(x_1,x_2)=x_1(1+\frac{1}{2}\frac{x_2^3}{x_1^2}+\cdots +c_e\frac{x_2^{3e}}{x_1^{2e}})\in k((x_1,x_2)),\ {\rm with}\  c_e\in k,$$ of the power series for $ x_1\sqrt{1+\frac{x_2^3}{x_1^2}}$.\Par Using the equation $(1+c_1\beta+\cdots +c_e\beta^e+\cdots )^2=1+\beta$ to compute inductively the coefficients $c_e$ we see that if $c_{e+1}=0$ in the field $k$, the smallest $e'>e+1$ such that $c_{e'}\neq 0$ is certainly $\leq 2e$. Otherwise all the $c_f,\ f\geq e+1$, are zero, and this makes $\sqrt{1+\beta}$ a polynomial in $\beta$, which would have to be of degree $\frac{1}{2}$.\Par Given a choice of extension of  the valuation $\nu_0$, to each truncation corresponds a polynomial $Q_e(y)$ of $k[[x_1,x_2]][y]/(y^2-x_1^2-x_2^3)$, which is $x_1^{2e-1}(y\pm T_e(x_1,x_2))$. These polynomials  satisfy relations of the form $Q_{e+1}=x_1^2Q_e-\lambda_e x_2^{3(e+1)}$ with $\lambda_e=\pm c_{e+1}$. Their valuations remain in the semigroup generated by $w_1,w_2,3w_2-w_1$: if $\lambda_e\neq 0$ they are $\nu (Q_e)=(e+1)(3w_2-w_1)+(e-1)w_1$ for $e\geq 1$. If $\lambda_e=0$ and $e'$ is the least integer $>e+1$ such that $c_{e'}\neq 0$, we compute easily that the value of the polynomial $Q_e$ is $3e'w_2+(2e-2e')w_1= e'(3w_2-w_1)+(2e-e')w_1$ and this is in the semigroup since $e'\leq 2e$. The unitary polynomials associated to the $Q_e$, after the removal of repetitions due to the fact that some $c_e$ may be zero in $k$, constitute the continuous family of key polynomials describing the given extension of $\nu_0$.
\item By proposition \ref{semg}, when the polynomial $p(y)$ is $\nu_0$-analytically irreducible, the semigroup $\Gamma$ is contained in the cone generated by $\Gamma_0$ so that the map ${\rm Spec}k[t^\Gamma]\to \A^r(k)$ corresponding to the injection $k[X_1,\ldots ,X_r]\subset k[t^\Gamma]$ is finite.\end{enumerate}}
\end{Remark}\noindent
In fact, with proposition \ref{semg}, remark \ref{tame} and the results of subsection \ref{Kpol} we have proved the following result:
\begin{corollary}\label{quasio} Let $p(y)$ be a unitary irreducible polynomial in $k[[x_1,\ldots ,x_r]][y]$, let $\nu_0$ be a monomial valuation on $k[[x_1,\ldots ,x_r]]=R_0$ such that the $\nu_0(x_i)$ are rationally independent, and let $\nu$ be an extension of $\nu_0$ to $R=R_0[y]/(p(y))$. Let $\Gamma$ be the semigroup of $\nu$ on $R$ and let $\varpi\colon{\rm Spec}k[t^\Gamma]\to \A^r(k)$ be the map corresponding to the injection $k[X_1,\ldots ,X_r]\subset k[t^\Gamma]$ determined by $X_i\mapsto t^{\nu_0(x_i)}$ (or the map ${\rm Spec}{\rm gr}_\nu R \to \A^r(k)$ corresponding to the natural inclusion ${\rm gr}_{\nu_0}R_0\subset{\rm gr}_\nu R$). The valuation $\nu$ defines a separable and tame extension of valued fields $(K_0,\nu_0)\subset (K,\nu)$ if and only if the index $[\Phi:\Phi_0]$ is prime to the characteristic of $k$, and if that is the case the polynomial $p(y)$ is $\nu_0$-analytically irreducible if and only if the map $\varpi$ is finite and thus makes ${\rm Spec}k[t^\Gamma]$ (or ${\rm Spec}{\rm gr}_\nu R$) a quasi-ordinary singularity.\hfill\qedsymbol
\end{corollary}

As a consequence of proposition \ref{nokey} and the second statement of proposition \ref{semg}, apart from the fact that one must replace "irreducible" by "$\nu_0$-analy-\break  tically irreducible", the situation is quite similar to the plane branch case where $r=1$. Given $q(y)\in R_0[y]$ one can replace the intersection number of $q(y)=0$ with our unitary polynomial $p(y)=0$ by the valuation $\nu_0({\rm Res}_y(p(y),q(y))$ of the resultant of the two polynomials, and if we assume that the degree of $p(y)$ is prime to the characteristic of $k$, we can define approximate roots, for which we refer to \cite{PP}. We only recall here that if $n$ is the degree of $p(y)$ and $d$ is an integer dividing $n$ the approximate root of $p(y)$ of degree $n/d$ is the unique polynomial $q(y)\in  k[[x_1,\ldots x_r]][y]$ such that ${\rm deg}(p(y)-q(y)^d)< {\rm deg}p(y)-{\rm deg}q(y)$. It must be unitary. It is also the unique unitary polynomial such that the $q(y)$-adic expansion of $p(y)$ has the form:
$$p(y)=q(y)^d+a_1q(y)^{d-1}+\cdots +a_d,$$
with ${\rm deg}a_i<{\rm deg} q(y)$, and the coefficient $a_1$ is zero (see \cite{PP}, proposition 6.1).\par
We are going to relate this to the expansion $$q(y)=\sum a_t(x)y^{t_1}Q_2(y)^{t_2}\ldots Q_k(y)^{t_k}p(y)^{t_\infty}$$ given by the valuative Cohen theorem, which \emph{in the case where the $Q_i(y)$ are unitary} can be obtained by successive divisions as in \cite{PP}.\par\medskip
Given $p(y)\in R_0[y]$, of degree $n$ prime to the characteristic, and a rational Abhyankar valuation of $R_0$, we will now give an analogue of Abhyankar's irreducibility criterion for plane curves, as found in \cite{A}, in \cite{GB-P}, theorem 10.9 and in \cite{GP3}, theorem 4.2; see also the analogous result for quasi-ordinary polynomials in \cite{As}. We need some preliminaries. First, we saw in the paragraph following definition \ref{ord} that we could assume that for any extension $\nu$ of $\nu_0$ to $R_0[y]/(p(y))$ we had $\nu (y)\notin\Gamma_0$. We present the same fact a little differently:
\begin{lemma}\label{resres} If $p(y)\in R_0[y]$ is $\nu_0$-analytically irreducible, there is a series $\upsilon(x)\in R_0$ such that $\frac{\nu_0({\rm Res}_y(p(y),y-\upsilon (x))}{{\rm deg}p(y)}\notin\Gamma_0$.
\end{lemma}
\begin{proof} Let $\nu$ be the unique extension of $\nu_0$ to $R_0[y]/(p(y))$ (proposition \ref{semg}). Recall that ${\rm Res}_y(p(y),y)=p(0)$ and that it follows from lemma \ref{order} (statement 2, with all the $s_i=0$) that $\frac{\nu_0(p(0))}{n}=\nu(y)$.\Par If $\frac{\nu_0(p(0))}{n}=\nu(y)\in\Gamma_0$, we have $\lambda_0\in k^*,\gamma_0\in \Gamma_0$ such that $\nu(y-\lambda_0x^{\gamma_0})>\nu(y)$. If $\nu(y-\lambda_0x^{\gamma_0})\in \Gamma_0$ we can repeat the construction and if there does not exist a series $\upsilon (x)$ as in the statement, we obtain, possibly by transfinite summation, a series such that ${\rm Res}_y(p(y),y-\upsilon (x))=0$, which means that $y-\upsilon (x)$ divides $p(y)$ and contradicts the irreducibility of the polynomial $p(y)$.\Par
The reader is encouraged to produce another proof, valid if the degree $n$ of $p(y)$ is not divisible by the characteristic of $k$, by verifying that if the coefficient of $y^{n-1}$ in $p(y)$ is zero, one can take $\upsilon(x)=0$. \end{proof} \par\medskip
 Now, assuming that the degree of $p(y)$ is not divisible by the characteristic of $k$,  we can build a sequence of values, numbers, and polynomials as follows:\par\noindent
begin with $\gamma_1=\frac{\nu_0({\rm Res}_y(p(y),y))}{{\rm deg}p(y)}$, take $n_1$ to be the smallest integer such that $n_1\gamma_1\in\Phi_0$; it divides $n$ and we can define $Q_2$ to be the approximate root of degree $n_1$ of the polynomial $p(y)$.\par\noindent
Assuming that the $\gamma_k, n_k, Q_{k+1}(y)$ have been defined for $k\leq j-1$, define $\gamma_j =\frac{\nu_0({\rm Res}_y(p(y),Q_j(y)))}{{\rm deg}p(y)}$, then define $n_j$ to be the least integer such that $n_j\gamma_j$ is in the group $\Phi_{j-1}$ generated by $\Phi_0,\gamma_1,\ldots \gamma_{j-1}$ and $Q_{j+1}$ to be the approximate root of degree $n_1\ldots n_j$ of $p(y)$. With this construction we have:
\begin{proposition}\label{Abhirr} Let $p(y)\in k[[x_1,\ldots x_r]][y]=R_0[y]$ be a unitary polynomial of degree $n$ prime to the characteristic of $k$. Let $\nu_0$ be a rational valuation of $k[[x_1,\ldots x_r]]$ such that the $\nu (x_i)$ are rationally independent.\par\noindent The following are equivalent:
\begin{enumerate}\item The polynomial $p(y)$ is $\nu_0$-analytically irreducible.
\item \begin{itemize}\item There exists a series $\upsilon(x)\in R_0$ such that $\frac{\nu_0({\rm Res}_y(p(y+\upsilon (x)),y)}{{\rm deg}p(y)}\notin\Gamma_0$ and after replacing $p(y)$ by $p(y+\upsilon (x))$, in the construction described above we have for each $j$ that $n_j>1$ and $n_j\gamma_j\in\langle\Gamma_0,\gamma_1,\ldots ,\gamma_{j-1}\rangle$, and:\item  Giving variables $u_j$ the weight $\gamma_j$, the map of complete $k$-algebras\\ $k[[x,u_1,\ldots ,u_t]]\to  R_0[y]/(p(y))$ determined by $x_i\mapsto x_i$ and $u_j\mapsto Q_j(y)$ is an overweight deformation of a prime binomial ideal corresponding to the relations $n_j\gamma_j\in\langle\Gamma_0,\gamma_1,\ldots ,\gamma_{j-1}\rangle$ (in particular, the inequalities $n_j\gamma_j<\gamma_{j+1}$ hold).\end{itemize}
\end{enumerate}
 \ \     If these conditions are satisfied, the unique extension of the valuation $\nu_0$ to $R_0[y]/(p(y))$ is given by $$\nu(q(y))=\frac{\nu_0({\rm Res}_y(p(y),q(y))}{{\rm deg}p(y)}$$ and any presentation of $R_0[y]/(p(y))$ by the valuative Cohen theorem can be modified in such a way that the polynomials $Q_j(y)$ are the approximate roots of $p(y)$.
\end{proposition}
\begin{proof} Here the "overweight deformation" condition of $2)$ means that the $Q_j$ satisfy the relations obtained by elimination: $$Q_{j+1}=Q_j^{n_j}-\lambda_j x^{r_j}\prod_{k\in E(i)}Q_k^{t^{(i)}_k}-g_j(x,y,Q_2(y),\ldots ,Q_i(y))$$ with the weight condition, so that $p(y)$ is the end result of the elimination process. Since $p(y')$ is $\nu_0$-analytically irreducible if and only if $p(y'+\upsilon(x))$ is, it follows from lemma \ref{resres} that $2)$ implies $1)$. The converse is a consequence of propositions \ref{semg} and \ref{crit} provided that we can show that the approximate roots are eligible as polynomials $Q_j(y)$ in the sense of these propositions. But if, as we may after proposition \ref{crit}, we view the expansion $$u_{j+1}=u_j^{n_j}-\lambda_jx^{r_j}\prod_{k\in E(i)}u_k^{t^{(i)}_k}-g_j(x,u_1,\ldots ,u_j)\ {\rm mod}.F$$ as giving rise to the $Q_j$-adic expansion of $Q_{j+1}$, we need only to show that we can avoid the appearance of a term with $u_j^{n_j-1}$. Should such a term occur, it must come from $g_j$ and by the overweight condition must appear as $g_{j,n_j-1}(x,u_1,\ldots ,u_{j-1})u_j^{n_j-1}$ with $w(g_{j,n_j-1})>w(u_j)$. But then we can make a change of representatives $Q_j\mapsto Q_j-\frac{g_{j,n_j-1}(x,y,\ldots ,Q_{j-1}(y))}{n_j}$ (an avatar of the Tschirnhausen transformation, permissible since $n_j$ must be prime to the characteristic) to make this term disappear and transform our $Q_j$ into approximate roots of $p(y)$. \Par
Since we assume $p(y)$ to be analytically irreducible, the fact that the valuation is given by the resultant is due to the uniqueness of the extension of $\nu$ to a splitting field of $p(y)$ and the fact that ${\rm Res}_y(p(y),q(y))=\prod_{p(\alpha_i)=0}q(\alpha_i)$ (For this usage of the resultant, which is classical for curves, see \cite{GB-GP}, \cite{PP2}, D\'efinition 5.7 and \cite{Gr}, proposition 3.1).
\end{proof}
\begin{Remark} \small{\begin{enumerate}
\item The idea behind Abhyankar's criterion is that one knows that if $p(y)$ is analytically irreducible the extension of the valuation $\nu_0$ to $R_0[y]/(p(y))$ is given by the resultant. So one begins to compute, with the resultant, the would-be valuations of the approximate roots which, as we show, should give an overweight deformation with all the $s_i=0$ (proposition \ref{semg}). If they do, and only if they do, then $p(y)$ was indeed analytically irreducible.
\item The reader who is familiar with Abhyankar's criterion will remark that the conjunction of the condition $n_j\gamma_j<\gamma_{j+1}$ and the "straight line condition" which appear in that criterion has become an overweight condition with respect to binomials. This can be compared to (\cite{GB-P}, \S\S 7 and 8 and \cite{GP3}) for plane curves. It should also be compared with the proof in \cite{CM} where in dimension 2 the existence of a generating sequence for a curve valuation replaces the valuative Cohen theorem, as key polynomials and approximate roots do in \cite{GB-P}.
\item Using the change of variables $y\mapsto y-x_1$, we can rewrite G. Rond's example in such a way that $\nu(y)\notin\Gamma_0$ as $p(y)=y^2+2x_1y-x_2^3$. Then the value $\gamma_1$ of $y$ given by the resultant is $\frac{3w_2}{2}$, so $n_1=2={\rm deg}p(y)$. Our overweight deformation would have to be $y^2-x_2^3+2x_1y$, but the weight of $2x_1y$ is $w_1+\frac{3w_2}{2}<3w_2$. The overweight deformation condition fails and the criterion tells us that the polynomial is not analytically irreducible. From the viewpoint of proposition \ref{semg}, it \emph{is} an overweight deformation, written as $2x_1y-x_2^3+y^2$ with $y$ of weight $3w_2-w_1$, but we know it is reducible because $s_1\neq 0$.\Par These two verifications correspond to the two ways of expressing that a Newton polygon is \emph{not} the segment joining the point $(0,{\rm deg}p(y))$ to some point $(m,0)$ on the horizontal axis: the first is to say that there is an exponent below that segment, and the second is to say that the point $(0,{\rm deg}p(y))$ is strictly above the line supporting the compact face of the Newton polygon ending at the point $(m,0)$.
\end{enumerate}}
\end{Remark}
The valuative Cohen theorem allows us to highlight the close relationship between this result and the next one, a relationship which is also made apparent in \cite{GB-P} for plane curves, from a different viewpoint.
\begin{proposition} Let $k$ be a field, let $r$ be an integer, and let $\Gamma$ be a torsion free commutative semigroup whose associated group is $\Z^r$, which we assume to be equipped with a total monomial order $\prec$ such that $\Gamma\subset\Z^r_{\succeq 0}$. The following two conditions are equivalent:\par\medskip\noindent
\hbox{\rm \textbf{(1)}} The semigroup $\Gamma$ is finitely generated, hence by \cite{Ne} well ordered, and contains a free subsemigroup $\Gamma_0\simeq\N^r$ such that if we write, after definition \ref{ord}, a minimal (with respect to $\Gamma_0$ and $\prec$) system of generators as $\Gamma =\langle\Gamma_0,\gamma_1,\ldots ,\gamma_l\rangle$ and define $\Phi_{i-1}$ to be the group generated by the semigroup $\Gamma_{i-1}= \langle\Gamma_0,\gamma_1,\ldots ,\gamma_{i-1}\rangle$, the following holds:\Par If $n_i$ is the least positive integer $k$ such that $k\gamma_i\in\Phi_{i-1}$,  then for each $i$, $1\leq i\leq l$, we have $n_i\gamma_i\in \Gamma_{i-1}$ and for $1\leq i\leq l-1$ we have $n_i\gamma_i\prec \gamma_{i+1}$.\par\medskip\noindent
\hbox{\rm \textbf{(2)}} The ordered semigroup $(\Gamma,\prec)$ is the semigroup of values of a rational Abhyankar valuation $\nu$ on a ring of the form $R=k[[x_1,\ldots ,x_r]][y]/(p(y))$ where $\nu(x_1),\ldots ,\nu(x_r)$ are rationally independent and the polynomial $p(y)$ is unitary and $\nu_0$-analytically irreducible with respect to the restriction $\nu_0$ of $\nu$ to $k[[x_1,\ldots ,x_r]]$.

\end{proposition}
\begin{proof} We have just seen how \textbf{(2)} implies \textbf{(1)}, since the fact that all $s_i=0$ implies the inclusions $n_i\gamma_i\in \Gamma_{i-1}$, and the inequalities $n_i\gamma_i\prec \gamma_{i+1}$ follow from the equations as normalized in proposition \ref{nokey} (compare with the introduction of \cite{GB-P}). To prove the converse, using proposition \ref{expression} and the hypothesis, write each $n_i\gamma_i=r_i+\sum_{1\leq k\leq i-1}t^{(i)}_k\gamma_k$ with $r_i\in\Gamma_0$. Then, build the corresponding equations $H_i^\circ=u_i^{n_i}-x^{r_i}\prod_{1\leq k\leq i-1}u_k^{t^{(i)}_k}-u_{i+1}=0$ for $1\leq i\leq l-1$ and $ H^\circ_l=u_l^{n_l}-x^{r_l}\prod_{1\leq k\leq l-1}u_k^{t^{(l)}_k}=0$, which, in view of the inequalities $n_i\gamma_i\prec \gamma_{i+1}$, determine an overweight deformation of a prime binomial ideal (the proof that the binomial ideal is prime appears in the proof of a) of proposition \ref{nokey}). By proposition \ref{resoverwght} the weights $w(x_j)=\nu(x_j)$ and $w(u_i)=\gamma_i$  determine a valuation on the quotient $k[[x_1,\ldots ,x_r, u_1,\ldots ,u_l]]/(H^\circ_1,\ldots ,H^\circ_l)$, whose semigroup is obviously $\Gamma$. There only remains to eliminate the variables $u_2,\ldots ,u_l$ to obtain a $\nu_0$-analytically irreducible unitary polynomial (proposition \ref{crit}) of degree $n_1\ldots n_l$ in $y:=u_1$.\end{proof}\noindent
\begin{Remark}\label{curve}\small{ \begin{enumerate}
\item When $\Gamma$ generates the group $\Z^r$, to give a total monomial order on $\Z^r$ with $\Gamma\subset\Z^r_{\succeq 0}$ is equivalent to giving a total monomial order on $\Gamma$, so the proposition concerns totally ordered affine semigroups. The notation $\prec$ has been chosen here to emphasize the role of the order in the statement.
\item  The special case $r=1$ (numerical semigroups of plane branches) of this result is due to Bresinsky (\cite{Br}) in characteristic zero, using the Puiseux expansion. It was rediscovered in (\cite{Te0}, Remark 2.2.2) with a proof quite close to the one given above, but still relying on Puiseux expansions, and extended to positive characteristic by Angerm\"uller in \cite{An}. Recently in \cite{GB-P} Garc\'\i a Barroso and P\l oski have given a more general version for plane branches in arbitrary characteristic with a proof based on intersection theory. Among other things they show that the inequalities $n_i\gamma_i<\gamma_{i+1}$ imply that in the expression $n_i\gamma_i=\phi^{(i)}_0+\sum_{1\leq k\leq i-1}t^{(i)}_k\gamma_k$ of proposition \ref{expression} we must have $\phi^{(i)}_0>0$ (this is generalized in proposition \ref{ineq} above). Thus, when $r=1$ the second condition of $\mathbf{(1)}$ implies the first.\item As was already implicit in the statement of proposition 3.2.1 of \cite{Te0}, from the viewpoint of the valuative Cohen theorem, for an Abhyankar valuation the classical inequalities $n_i\gamma_i<\gamma_{i+1}$ appear as consequences of the fact that the ring has dimension $r$ and embedding dimension $\leq r+1$. If the ring $R$ is regular of dimension 2, then the inequalities hold also for non-Abhyankar valuations; see \cite{C-T} and \cite{C-V}.  Related results, with different motivations, also valid in arbitrary characteristic and expressed in the language of valuations on a polynomial ring in two variables, are found much earlier in an article\footnote{I am grateful to Arkadiusz P\l oski and Evelia Garc\'\i a Barroso for bringing it to my attention.} of A. Seidenberg; see \cite{S}. In this paper, Seidenberg in particular shows the existence of generating sequences (or key polynomials) for rational valuations centered in a polynomial ring in two variables and makes use of the inequalities $n_i\gamma_i<\gamma_{i+1}$.

\end{enumerate}}
\end{Remark}
In the same spirit we have an analogue of the Abhyankar-Moh irreducibility theorem as it is stated in (\cite{GB-P}, corollary 8.3); it is a direct consequence of proposition \ref{nokey} and what we have seen in this subsection. We keep the notations of subsection \ref {Kpol} and in particular propositions \ref{nokey} and \ref{semg}:
\begin{proposition}\label{A-M} Let $\nu_0$ be a valuation of $k[[x_1,\ldots ,x_r]]$ such that the $\nu(x_i)$ are rationally independent. Let $p(y)\in k[[x_1,\ldots ,x_r]][y]$ be a $\nu_0$-analytically irreducible unitary polynomial and $\nu$ the unique extension of $\nu_0$ to $k[[x_1,\ldots ,x_r]][y]/(p(y))$. Denote by $\mu$ the valuation on $k[[x,y]]$ with values in $(\Z\times \Z^r)_{lex}$ which is composed with $\nu$ and gives value $(1,0)$ to $p(y)$. Let $Q_1=y, Q_2(y),\ldots ,Q_l(y)$ be the unitary polynomials obtained by elimination from the $H_i$ of proposition \ref{nokey} and whose valuations, together with those of the $x_i$, minimally generate the semigroup $\Gamma$ of $\nu$. If $q'(x,y)\in k[[x,y]]$ is a series such that $q'(0,y)$ is of order $n={\rm deg}p(y)$ in $y$ and $\mu(q'(x,y))>n_l\gamma_l$, then $q'(x,y)$ can be written as the product of a unit of $k[[x,y]]$ by a unitary polynomial which is $\nu_0$-analytically irreducible in $k[[x]][y]$ and is of the form:
$$q(y)=Q_l(y)^{n_l}-\lambda_lx^{r_l}\prod_{k\in E(l)} Q_k(y)^{t^{(l)}_k}-g'_l(x,y,Q_2(y),\ldots ,Q_l(y)),$$
where $g'(x,u_1,\ldots ,u_l)$ is a series of weight $>n_l\gamma_l$ and $g'_l(x,y,Q_2(y),\ldots ,Q_l(y))$ is a polynomial of degree $<n$ in $y$. The semigroup of the unique extension of $\nu_0$ to $k[[x_1,\ldots ,x_r]][y]/(q(y))$ is $\Gamma$.
\end{proposition}
\begin{proof} First, by the Weierstrass preparation theorem, up to multiplication by a unit of $k[[x,y]]$, we may replace $q'(y)$ by a unitary polynomial $q(y)$ of the same degree as $p(y)$. If we divide $q(y)$ by the unitary polynomial $Q_l(y)$ we therefore obtain $q(y)=Q_l(y)^{n_l}+A_{l-1}Q_l(y)^{n_l-1}+\cdots +A_0$ with ${\rm deg}_yA_i<{\rm deg}_yQ_l(y)$. If $\mu(q(y))\geq (1,0)$, then by reason of degree we must have $q(y)=p(y)$ and the result is proved. So we assume $\mu(q(y))< (1,0)$. By the valuative Cohen theorem $q(y)$ is the image of a series $\tilde q(x,u_1,\ldots ,u_l,v)=u_l^{n_l}+\cdots\in \widehat{k[x,u_1,\ldots, u_l,v]}$ which is of weight $\leq n_l\gamma_l$ since it contains $u_l^{n_l}$. Now we use proposition \ref{initial} much as in the proof of proposition \ref{finhyp}. By our assumption the weight of $\tilde q$ is less than the valuation of $q(y)$ so that ${\rm in}_w\tilde q$ belongs to the binomial ideal of the initial forms of the equations $H_i$. But any binomial $u_i^{n_i}-\lambda_ix^{r_i}\prod_{k\in E(i)} u_k^{t^{(i)}_k}$ which is not $u_l^{n_l}-\lambda_lx^{r_l}\prod_{k\in E(l)} u_k^{t^{(l)}_k}$ can be replaced, modulo the equations $H_i$, by $g_i+u_{i+1}$, thus increasing the weight and making it disappear from the initial form. Finally this produces a series which, since its image contains $Q_l^{n_l}$, must be of the form stated in the proposition: the only possibility for its initial form, by reason of degree and because it must be in the binomial ideal, is to be $u_l^{n_l}-\lambda_lx^{r_l}\prod_{k\in E(l)} u_k^{t^{(l)}_k}$. The polynomial $g'_l(x,y,Q_2(y),\ldots ,Q_l(y))$ is of degree $<n$ because $p(y)-q(y)$ is of degree $<n$. Finally our polynomial $q(y)$ is obtained by elimination from an overweight deformation of the same binomial ideal as the polynomial $p(y)$ and must have the same semigroup.
\end{proof}
\begin{remark} \small{\emph{Applying the proposition to the $\nu_0$-analytically irreducible polynomials $Q_j(y)$ gives a generalization of Theorem 8.2 of \cite{GB-P}. The Abhyankar-Moh irreducibility criterion is in fact an equisingularity criterion: if a polynomial has the same degree as a $\nu_0$-analytically irreducible unitary polynomial and sufficiently high contact with it, then it is not only $\nu_0$-analytically irreducible but in fact "$\nu_0$-equisingular" with that polynomial, in the sense that it determines the same semigroup.}}
\end{remark}
\section{Key polynomials and the valuative Cohen theorem}\label{keyCohen}
The classical theory of key polynomials uses division by unitary polynomials of $K_0[y]$ and can at best produce polynomials with coefficients in $R_{\nu_0}[y]$ while the valuative Cohen theorem produces, at least when $R_0$ is complete, polynomials in $R_0[y]$. Also, when the semigroup $\Gamma$ of $\nu$ on $R_0[y]/(p(y))$ is finitely generated, the valuative Cohen theorem produces finitely many polynomials while the classical theory may produce infinitely many polynomials, as we have seen. In this section we indicate references which may help the reader to analyze the differences in viewpoints.\par Concerning our construction of key polynomials, the basic mechanism appears in the equations $H_i$ used above in the study of Abhyankar valuations of $k[[x]][y]/(p(y))$ and our key polynomials are the images in $k[[x]][y]/(p(y))$ of the variables $u_i$ by the valuative Cohen map. Assuming that all the $s_i=0$, these polynomials $Q_i(y)$ determine the valuation by the following construction (compare with \cite{Te1}, Example 4.20, and \cite{M1}): given a polynomial $q(y)$, we can replace in it every occurrence of $y^{n_1}$ by $x^{r_1}+g_1(x,y)+u_2$, and then continue inductively, replacing each occurrence of $u_i^{n_i}$ by $\lambda_ix^{r_i}\prod_k u_k^{t^{(i)}_k}+g_i(x,y,\ldots ,Q_i(y))+u_{i+1}$. In this way one produces in a finite number of steps a polynomial $\tilde q(y,Q_2(y),\ldots ,Q_s(y))$ in $y$ and the $(Q_k)_{k\geq 2}$ with coefficients in $k[[x]]$, the image of $\tilde q(u_1,u_2,\ldots ,u_s)\in\widehat{k[x, (u_i)_{i\in I}]}$ which has the virtue that its value is now the minimum of the values of its terms, which is also the weight of $\tilde q(u_1,u_2,\ldots ,u_s)$. Indeed, all the cancellations of initial forms which complicate the computation of valuations have been transmuted into variables with weights.\par If the $s_i$ are not all zero, we use for $q(y)$ the construction used for $p(y)$ in the proof of proposition \ref{finhyp}, with the same result as above except that now the polynomials $Q_i$ used in the substitutions are not unitary. In the complete case, this expansion in terms of the $u_i$ is exactly the one coming from the valuative Cohen Theorem.\par The inverse operation, which we have seen above, is the elimination of the variables $u_i$ to recover the polynomials in one variable whose valuations generate the semigroup, which amounts to defining explicitly the map of the valuative Cohen theorem. Indeed, if we replace each $u_i$ by the polynomial $Q_i(y)$ we get the expansion in terms of key polynomials which comes from that theory. These polynomials are key polynomials by proposition \ref{Key}, but here they are defined for rings, not fields as in the classical case, and the proof of proposition \ref{nokey} shows that there is indeed a difference.\par
The classical theory of key polynomials does not work in this way. Key polynomials were developed by Ostrowski (see \cite{Roq}) and MacLane (see \cite{Mac}) in special cases and in full generality by Vaqui\'e (see \cite{V1}-\cite{V5}) to determine all the extensions of a given valuation of the field $K$ to the field $K(y)$, where $y$ may be algebraic over $K$ or not. Given such an extension of valuations, the key polynomials are defined as "milestones" of the cancellations of initial forms which one has to take into account when computing the valuation of an element $Q(y)\in K[y]$; this systematic recording of cancellations ultimately provides a sequence, in general indexed by an ordinal, of unitary polynomials in $K[y]$ whose values determine completely the valuation. Each element of $K[y]$ is a polynomial in these polynomials in such a way that now its valuation is the least of the values of the terms.\par In a sense the theory of key polynomials constructs simultaneously, by an inductive process, a generating sequence for the valuation on $K_0[y]$ and the relations between its members. But it uses the structure of $K_0[y]$ in an essential way. For an arbitrary equicharacteristic noetherian complete local  domain the valuative Cohen theorem does the same with series but in a different, non constructive, way. Also, Vaqui\'e's theory has the advantage that it can be used to build all the extensions of a valuation such as $\nu_0$ while the valuative Cohen theorem starts from a given extension. Thus, Vaqui\'e's theory may build infinitely many polynomials whose values belong to a finitely generated semigroup, as we saw in section \ref{precisions}. The point here is that \emph{key polynomials record new cancellations of initial forms whether their result augments the semigroup of the valuation on a given subring of $K_0(y)$ or not}. There is no requirement of them corresponding to a minimal set of generators of the graded algebra of a noetherian subring. \par In (\cite{Te1}, examples 4.20 and 4.22) the author explained in the special case of valuations of $k(x)[y]$, where $k$ is algebraically closed of characteristic zero, the relationship between key polynomials, approximate roots, the valuative Cohen theorem, and the embedding of the field $k(x)(y)$ in a field of generalized power series \emph{\`a la} Hahn-Kaplansky (proposition 5.48 of \emph{loc.cit.}). Pedro Gonz\'alez P\'erez showed among other things in \cite{GP2} that the same mechanism is at work for quasi-ordinary irreducible  hypersurface singularities, still in characteristic zero.  In \cite{M1}, M. Moghaddam has extended this mechanism again to generalized quasi-ordinary series in $X_1,\ldots, X_d$, possibly non-algebraic over the field $k( X_1,\ldots, X_d)$.\par\noindent It is perhaps notable that although the construction given above differs from that described in the work of Vaqui\'e, the polynomials $Q_i(y)\in R_0[y]\subset  K_0[y]$ are indeed key polynomials for the valuation $\nu$. This is due mostly to part $(1)$ of proposition \ref{nokey} which gives a structure to the equations: this structure is automatically given by the classical construction of key polynomials but is in general sorely lacking in the equations given by the valuative Cohen theorem. The basic reason is that the rings to which we apply the valuative Cohen theorem are not assumed to be regular.  As we noted, seeking polynomials in $R_0[y]$ instead of $K_0[y]$ also makes a difference.  \par Other approaches, which also construct key polynomials step by step (in a transfinite sense) are found in the work of M. Spivakovsky and his collaborators, who have developed (see \cite{H-O-S}) a general theory of key polynomials from the viewpoint of generating sequences for valuations of rank one on complete regular local rings, generalized to higher rank by W. Mahboub (see \cite{Ma2}) and in the work of M. Moghaddam (see \cite{M2}, \cite{M3}) who has generalized the constructions of Favre and Jonsson (see \cite{Favre-Jonsson}) for $\C\{x,y\}$\footnote{A description of the valuation semigroups and residue field extensions for valuations of \emph{any} two dimensional regular local ring, containing also an algorithmic construction of a generating sequence for the valuation on $R$ itself, is given by Cutkosky and Vinh in \cite{C-V}.}. The relationship of the first approach with Vaqui\'e's is explained in \cite{Ma} but again there, the key polynomials are sought in $K_0[y]$ and not $R_0[y]$ although Mahboub corrects a lapsus in the definition of key polynomials in \cite{H-O-S} where one seeks generators of the graded ring with respect to the valuation which is associated to $K_0(y)$ and not $K_0[y]$.\par It should be noted that in \cite{C-V}, Cutkosky and Vinh build the equivalent of a sequence of key polynomials for a valuation of any noetherian regular local ring of dimension two. Their method is different from what we do in subsection \ref{Kpol} where we heavily use the hypothesis that the valuation is Abhyankar. Our proposition \ref{Abext} can be viewed as a generalization of cases 1), 3) and 4) of their proposition 3.4. There is a difference in case 3) since we allow ourselves to move in the tree of $\nu$-modifications and so can assume $\beta=0$ in their formula in view of the results of \cite{HOST} quoted in our proof.\par The possibility of defining a generalization of key polynomials for valuations on rings that are not regular has not yet been established. It would mean essentially a structuring of the equations for the ring $R$ given by the valuative Cohen theorem allowing the systematic elimination of variables $u_i$ that are not necessary to generate the maximal ideal of $R$.

\section{The Artin-Schreier example}\label{AS}
We continue the study we have begun in \cite{Te1} of the difference between the overweight deformation method for local uniformization and the ramification theoretic method of \cite{KK1} by revisiting example 4.23 of \cite{Te1} from the viewpoint of key polynomials. The main points here are that the construction of key polynomials does not commute with base change (or field extension) and that in positive characteristic the price to pay to write a fractional power series parametrization of a curve may be high in terms of key polynomials.\par
Let $k$ be a perfect field of finite characteristic $p$ and $x$ an indeterminate; set $K=\bigcup_{n\geq 1}k(x^{\frac{1}{p^n}})$, the perfect closure of $k(x)$.
There is a unique extension $\nu$ to $K$ of the $x$-adic valuation of $k(x)$, and its valuation ring is \textit{not} n\oe therian. The value group of this valuation is $\frac{1}{p^\infty}\Z$. Consider the series
$$y=\sum_{i=1}^{\infty}x^{1-\frac{1}{p^i}}\in k[[x^{{\mathbf Q}_+}]];\eqno{(NP)}$$
it is a solution of the polynomial equation
$$y^p-x^{p-1}(1+y)=0.$$ This equation is an Artin-Schreier equation\footnote{Of course this is a "baby case" compared to the general Artin-Schreier equation\break $y^p-g(x_1,\ldots ,x_r)^{p-1}y+f(x_1,\ldots ,x_r)=0$ for which local uniformization was proved for $r\leq 3$ in \cite{CP}.}: it is obtained from the standard Artin-Schreier
$z^p-z=\frac{1}{x}$ by replacing $z$ by $\frac{y}{x}$. Note that if $p=2$ it is non singular.\par
 If we set $L=K(y)$, it is shown in \cite{Ku} that the extension $L/K$ has degree $p$ and
defect $p$. More precisely, the unique extension $\nu$ to $K$ of the $x$-adic valuation of $k(x)$ has a unique extension $\nu'$ to $L$, with the same
group of values, so that the ramification index
$e=[\Phi':\Phi]$ is equal to one, and no residual extension so that the inertia degree $f=[\kappa(\nu'):\kappa (\nu)]$ is also equal to one. The extension is
of degree $p$ so that the Ostrowski ramification formula (see \cite{Ku}, \cite{Roq} and \cite{V2}), which is $[L:K]=def$, where $d$ is the defect, gives $d=p$. This defect
complicates the parametrization but does not make it more difficult to create a non-singular model of the affine plane curve defined by the same equation in $k[x,y]$. We remark that our curve is a deformation of the
monomial curve
$y^p-x^{p-1}=0$, and apply to this monomial curve the toric embedded resolution process of \cite{GP-T1} and \cite{Te1}: it gives us a proper toric map of non singular surfaces $\Z(\Sigma)\to\A^2(k)$ and a chart $Z(\sigma)$ of $Z(\Sigma)$ where the map is described by
$x=y_1^py_2,y=y_1^{p-1}y_2$. Our equation then becomes $y_1^{p(p-1)}y_2^{p-1}(y_2-1-y_1^{p-1}y_2)$, so that the strict transform
$y_2-1-y_1^{p-1}y_2=0$ is non singular. It can be parametrized in the Zariski neighborhood $y_1\neq 1$ of the point $y_1=0,\ y_2=1$ of the exceptional divisor $y_1y_2=0$ by
$y_2=\frac{1}{1-y_1^{p-1}}$, so that we have the following rational parametrization of our curve:
$$  x=\frac{y_1^p}{1-y_1^{p-1}};\ \ y=\frac{y_1^{p-1}}{1-y_1^{p-1}}.\eqno{(1)}$$
Remark that in this case the blowing-up of the origin also gives an embedded resolution.\par\noindent
It is with this parametrization that we get the embedding of proposition \ref{Kap} of $R=k[x,y]_{(x,y)}/(y^p-x^{p-1}(1+y))$ into $k[[y_1^\N]]\subset k[[y_1^{\R_{\geq 0}}]]$. To seek a Newton-Puiseux type embedding where $y$ is a series in $x$ makes things much more complicated as in the expansion $(NP)$ above.\par The fact that the extension $K\subset K(v)$ has defect seems to be related to the fact that while the extension of fields $k(x)\to k(x)[y]/(y^p-x^{p-1}(1+y))$
is separable, the extension of graded rings associated to the $x$-adic valuation of $k[x]$ and its extension to $k[x,y]/(y^p-x^{p-1}(1+y))$,
which is
$$k[X]\to k[X,Y]/(Y^p-X^{p-1})$$ is purely inseparable of degree $p$ (the projection to the $X$-axis does not make the monomial curve quasi-ordinary; see corollary \ref{quasio}) and in addition the binomial $y^p-x^{p-1}$ becomes reducible when we extend $k(x)$ to $K$ so that there is no hope for our ring to be, after this extension, an overweight deformation of such a simple binomial.\par
Let us now illustrate on this example the difference in positive characteristic between a system of key polynomials and our method of specialization to the associated graded ring, first for the extension $k(x)\to k(x)(y)$ and then for the extension $K\to K(y)$ where the key polynomials correspond to truncations of the solution $y=\sum_{i=1}^{\infty}x^{1-\frac{1}{p^i}}$ at each exponent where the denominator of the exponent increases, which in this case means truncating successively at every exponent.\par

Here, making more explicit what is described in \cite{Te1}, Example 4.20, we must distinguish between "natural coordinates" which belong to affine space in which the singularity corresponding to $R$ is determined by "natural equations" and the "key polynomials" which are obtained by eliminating most of the natural coordinates between the natural equations in order to obtain polynomials in one variable.\par

Let us first build a system of "key polynomials without root extraction" for the pseudo-valuation of the ring $k(x)[y]$ defined by the parametrization $(1)$, whose kernel is the prime ideal generated by $y^p-x^{p-1}(1+y)$. We normalize the valuation on $k(x)$ by setting $\nu(x)=1$.\par\noindent
The first key polynomial has to be $Q_0=y$, then we have $Q_0^p-x^{p-1}=Q_1$ and here, since we work over $k$, we cannot say that $Q_1$ is a $p$-th power and therefore not of minimal degree. So we keep $Q_1$ as our key polynomial. In view of the equation the value of $Q_1$ has to be that of $yx^{p-1}$, that is $p-\frac{1}{p}$. In fact the equation tells us that $Q_1-yx^{p-1}=0$ and we have to stop. So our system of polynomials consists of $y, Q_1$, and we find again the presentation of our curve as an overweight deformation of a curve defined by binomials, namely $Q_0^p-x^{p-1}=0\ , Q_1-yx^{p-1}=0$.\par We are no longer dealing with plane curves in the coordinates $x,y$ but are working in the space with "key coordinates" $x, y=Q_0, Q_1$, where our curve is a non transversal intersection of two non-singular surfaces: $Q_0^p-x^{p-1}-Q_1=0\ , Q_1-yx^{p-1}=0$. However, in view of the form of the second equation, this presentation is isomorphic to the deformation $y^p-x^{p-1}=yx^{p-1}$ of the binomial equation $y^p-x^{p-1}=0$ which we used above.\par\noindent Here we have \textit{not} tried to solve the equation with some series $y(x)$ \textit{\`a la} Newton-Puiseux, but to present our curve as a deformation of a binomial variety over $k$, namely $y^p-x^{p-1}=0$, which is a reduced binomial curve and as such has a toric embedded resolution of singularities. Over the field $K$ the same equation is a $p$-th power, namely $(y-x^{1-\frac{1}{p}})^p$. \par
Now let us build a sequence of key polynomials for the extension of the valuation $\nu_x$ on $K$ to the pseudo valuation on $K[y]$ whose kernel is the prime ideal generated by $y^p-x^{p-1}(1+y)$. Building the key polynomials amounts to writing systems of equations of plane curves of the form $F_j(x,y)=0$ whose solutions $y^{(j)}(x)$ are better and better finite approximations to our infinite series. We notice that on the curve we have $Q_1=(y-x^{1-\frac{1}{p}} )^p=yx^{p-1}$ so that the degree one polynomial $t_1=y-x^{1-\frac{1}{p}}$ has valuation $1-\frac{1}{p^2}>1-\frac{1}{p}=\nu(y)$ and we notice that $Q_1^p=t_1^{p^2}=x^{p^2-1}(1+y)$; its initial form is $Q_1^p-x^{p^2-1}=(Q_1-x^{p-\frac{1}{p}})^p=(t_1^p-x^{p-\frac{1}{p}})^p=(t_1-x^{1-\frac{1}{p^2}})^{p^2}$ and this produces a "generating sequence" polynomial $t_2=t_1-x^{1-\frac{1}{p^2}}$ whose $p^2$-th power, on the curve, is equal to $yx^{p^2-1}$ and which has valuation $1-\frac{1}{p^3}>\nu (t_1)$. Continuing in this manner we build a system of "key coordinates" $(t_i)_{i\geq 0}$ in $K[y]$, where $t_0=Q_0=y, t_1=y-x^{1-\frac{1}{p}}$ and which are subjected to the "key equations" $t_{i+1}=t_i-x^{1-\frac{1}{p^{i+1}}}$, with $\nu(t_i)=1-\frac{1}{p^{i+1}}$.  After eliminating $t_i$ between the first $i$ equations, which means interpreting $t_i$ as $y-\sum_{k=1}^ix^{1-\frac{1}{p^k}}$, they are the key polynomials for the extension $\nu$ from $K$ to $K[y]/(y^p-x^{p-1}(1+y))$ of the valuation $\nu_x$. The limit key polynomial is the result of the successive elimination of the variables $(t_i)_{i\geq 1}$ in the infinite sequence of degree one key polynomials. If the series $\sum_{i=1}^{\infty}x^{1-\frac{1}{p^i}}$ converged in the field $K$ for the topology given by the valuation $\nu_x$, the result would be the degree one polynomial $y-\sum_{i=1}^{\infty}x^{1-\frac{1}{p^i}}$, but it does not, so the result is the Artin-Schreier equation.\par
The $T_j=y-\sum_{k=1}^jx^{1-\frac{1}{p^k}}$ form a continuous admissible family of key polynomials of degree one in the sense of Vaqui\'e \cite{V4}, and their limit key polynomial is $y^p-x^{p-1}(1+y)=0$, which is of degree $p$. One verifies in this example the result of Vaqui\'e in \cite{V2}: \emph{the jump in degree between the members of the continuous admissible family and the limit key polynomial is equal to the defect.} \par
We have built a sequence $(\nu_i)_{i\geq 1}$ of pseudo-valuations of $K[y]$, beginning with the Gauss valuation $\nu_1$ which gives $y$ the value $1-\frac{1}{p}$ and corresponds to defining the value of a polynomial $P(y)\in K[y]$ as the order in $x$ of its restriction to the curve $y=x^{1-\frac{1}{p}}$, and in general $\nu_i(P)$ computes the order in $x$ of the restriction of the polynomial $P(y)$ to the curve $C^{(i)}$ given parametrically  by $y=\sum_{k=1}^ix^{1-\frac{1}{p^k}}$. This is essentially the same as what is done by Vaqui\'e at the end of \cite{V1}.\par\noindent
According to the theory of "approximate root" polynomials (see \cite{PP}), the valuations $\nu_i$ converge to the valuation $\nu$ as $i\to\infty$. In analogy to what is explained in \cite{Te1}, Example 4.20, in the characteristic zero case, if one keeps only the first $i$ equations $t_{k+1}=t_k-x^{1-\frac{1}{p^{k+1}}},\ 1\leq k\leq i$ and sets $t_{i+1}=0$, the system obtained defines, by elimination of the $t_k$ for $k\geq 1$, the equation of the curve $C^{(i)}$. The difference with \cite{Te1} is that here the limit curve as $i\to\infty$ is algebraic. The natural valuation on the curve $y^p-x^{p-1}(1+y)=0$ appears, as we have seen, as a very simple overweight deformation of  its associated graded ring $k[X,Y]/(Y^p-X^{p-1})$. The method of "approximate root" polynomials gives us an approximation process of this valuation, seen as a pseudo-valuation on the $(x,y)$ plane, by pseudo-valuations corresponding to the finite expansions of $y$ in rational powers of $x$ parametrizing the curves $C^{(i)}$.\par
Finally, the same story can be told in the language of overweight deformations: let $R_{\nu_x}\subset K$ be the valuation ring of the valuation $\nu_x$ and let ${\rm gr}_{\nu_x}R_{\nu_x}$ be its associated graded ring. Since the positive semigroup of $\frac{1}{p^\infty}\Z$ is $\frac{1}{p^\infty}\N$, we have a presentation ${\rm gr}_{\nu_x}R_{\nu_x}\simeq k[(U_i)_{i\geq 0}]/((U_i-U_{i+1}^p)_{i\geq 0})$, corresponding as in (\cite{Te1}, 4.3, corollary 4.13) to the presentation according to proposition \ref{finapp} of $\frac{1}{p^\infty}\N$ as the limit of the inductive system indexed by $\N$:
$$\N_{(0)}\subset\N_{(1)}\subset\cdots \subset \N_{(i)}\subset\cdots\subset\frac{1}{p^\infty}\N,$$ where each map is multiplication by $p$ and the map from the $i$-th copy of $\N$ to $\frac{1}{p^\infty}\N$ is $a\mapsto\frac{a}{p^i}$. The variable $U_i$ corresponds to the generator $\frac{1}{p^i}$ of the semigroup of values of $\nu_x$.\par The ring $({\rm gr}_{\nu_x}R_{\nu_x})[y]/(y^p-U_0^{p-1}(1+y))$ can be seen as an overweight deformation of the binomial scheme corresponding to the ring $({\rm gr}_{\nu_x}R_{\nu_x})[ (T_j)_{j\geq 0}]/(\{T_{i-1}-U_i^{p^i-1}\}_{i\geq 1})$, with the $T_i$ denoting the initial forms of the $t_i$ we introduced above, and the $U_i$ denoting, by abuse of notation, the images of the variables $U_i$ in ${\rm gr}_{\nu_x}R_{\nu_x}$. The deformation is given, exactly as in \cite{Te1}, 4.20, by changing the $i$-th equation $T_{i-1}-U_i^{p^i-1}=0$ to $t_{i-1}-U_i^{p^i-1}=t_i$ for all $i\geq 1$. The first equation $T_0-U_1^{p-1}$, read as $T_0-U_0^{1-\frac{1}{p}}$, is a factor of the initial form $T_0^p-U_0^{p-1}$, which is no longer irreducible over ${\rm gr}_{\nu_x}R_{\nu_x}$, of the equation $y^p-U_0^{p-1}(1+y)$. The first $j$ deformed equations, $(t_{i}=t_{i-1}-x^{1-\frac{1}{p^i}})_{1\leq i\leq j}$, once transformed by elimination into unitary polynomials of $K[y]$ in the variable $t_0=y$ as we have seen above, are the key polynomials for the valuation corresponding by overweight deformation (see proposition \ref{resoverwght}, a)) to the natural weight on the ring $({\rm gr}_{\nu_x}R_{\nu_x})[ (T_j)_{j\geq 0}]$ making the binomial equations homogeneous with $w(x)=1$.\par\medskip
We remark finally that the construction used in the proof of the quasi finiteness of the semigroup of Abhyankar valuations uses a different presentation for the complete local ring $$R=k[[x,y]]/(y^p-x^{p-1}(1+y)).\eqno{(A)}$$Indeed it suggests to write it as $$R=k[[y,x]]/(x^{p-1}-y^p(1+y)^{-1})\eqno{(B)}$$ since if we provide $R$ with the valuation induced from the $y_1$-adic valuation via the inclusion $R\subset k[[y_1]]$ given by the parametrization $(1)$, in the presentation $(A)$ we have used above the index of the group extension $\Phi_0\subset \Phi$ is not prime to $p$ (compare with remark \ref{tame}, 2)); although the extension of valued fraction fields is separable, it is not tame, but in $(B)$ it is. While in characteristic $p\neq 2$ the series $x^{p-1}-y^p(1+vy)^{-1}$ defines a non trivial deformation with parameter $v$ of its initial binomial, this deformation is equisingular in the sense of simultaneous resolution. If $p=2$ it gives a simultaneous power series parametrization of the curves of the family.\par As was mentioned at the end of the previous paragraph, the construction of key polynomials by elimination of key coordinates between key equations is more complicated in general and not yet understood.
\section{Conclusion}
If one sets aside the question of the finite generation of semigroups before any $\nu$-modification, the results of this paper complete the program of \cite{Te1} in the special case of Abhyankar valuations. The extension to all rational valuations might lead to a proof of local resolution of singularities along the following lines:\par
Let $k$ be an algebraically closed field and $X$ a closed reduced subscheme of a proper non singular algebraic scheme $W$ over $k$. One may ask (see \cite{Te1}, \cite{Te2} and \cite{Te3}) the following question:\par\noindent
\emph{Does there exist a closed embedding of $W$ into a non singular toric variety $Z$ such that the intersection of $X$ (resp. $W$) with the torus $T$ of $Z$ is dense in $X$ (resp. $W$) and there exist toric proper birational maps $Z'\to Z$ such that $Z'$ is non singular and the strict transforms $X'$ and $W'$ of $X$ and $W$ are also non singular and transversal to the toric boundary in $Z'$?}\par Moreover, the induced map $X'\to X$ should be a resolution of singularities (i.e., an isomorphism over the non singular part). One may ask further that the map $Z'\to Z$ is a composition of blowing-ups with non singular centers.\par
There are obvious local and formal versions of this question. The first formulation of a question of this type goes back to \cite{G-T} where such a re-embedding result was proved for germs of complex analytic plane branches. Pedro Gonz\'alez P\'erez and the author then proved an embedded resolution theorem for an affine toric variety equivariantly embedded in a normal one. This is a necessary step for the proof of embedded local uniformization by deformation of the embedded resolution of a toric variety associated to the valuation.\par A recent result of Tevelev shows that if one assumes embedded resolution of singularities to be true, most of the question above has a positive answer :\par\noindent
\begin{theorem}{\rm (Tevelev, \cite{Tev})} Let $k$ be an algebraically closed field of characteristic zero. Let $X\subset \P^n$ be an irreducible algebraic variety. For a sufficiently high order Veronese re-embedding $X\subset \P^N$ one can choose homogeneous coordinates $z_0,\ldots , z_N$, a smooth toric variety $Z'$ of the algebraic torus $T=\P^N\setminus \bigcup\{z_i=0\}$ and a toric birational morphism $Z'\to \P^N$ such that the following conditions are satisfied: $X\cap T$ is non-empty, the strict transform of $X$ in $Z'$ is smooth and intersects the toric boundary transversally, and $Z'\to\P^N$ is a composition of blowing-ups with smooth torus-invariant centers.
\end{theorem}
Tevelev's proof starts from an embedded resolution of singularities $Y\to\P^n$  of $X\subset \P^n$ and the re-embedding is constructed from sections of invertible sheaves on $\P^n$ built from the images of special classes of divisors on $Y$ involving the exceptional divisor and the pull back of the hyperplane class of $\P^n$. In fact Tevelev shows a stronger result, without assumption on the characteristic:\Par\textit{ Given an embedding of $X$ in a non singular irreducible projective variety $S$, an embedded resolution of singularities $Y\to S$ for $X\subset S$ and  an ample invertible sheaf $L$ on $S$, the projective embedding of $X$ obtained through the projective embedding of $S$ by any sufficiently large multiple of $L$ will have the desired property in such a way that the strict transform of $S$ by the toric embedded resolution will be $Y$.}\Par It shows that such toric embedded resolutions are in a sense "universal" among embedded resolutions, so that embedded resolution by a single toric modification of a larger ambient space is indeed an alternative to embedded resolution by sequences of blowing-ups with non singular centers. Since by \cite{DC-P} and \cite{Mo} birational toric maps of non singular toric varieties can be dominated by sequences of blowing-ups of non singular invariant centers, one can hope for a future unification of the two viewpoints, obtained by generalizing what is achieved for quasi-ordinary hypersurfaces in characteristic zero by (\cite{GP2}, theorem 3).\par The approach which led to the program proposed in \cite{Te1} is of course different since one of its goals is to prove embedded resolution in any characteristic. It is to prove first a "local uniformization" version of the result, obtaining for each (rational) valuation a re-embedding after which it is uniformized by a toric map of the ambient space in suitable coordinates.\par After the results of this paper one can formulate the following problem to summarize the remaining part of the program of \cite{Te1} in that direction:\par\noindent
\emph{1) Give a combinatorial proof of toric embedded local uniformization for rational Abhyankar valuations (possibly by showing that their semigroup is finitely generated).\Par 2) Show that, for every valuation $\nu$ of an excellent equicharacteristic noetherian local domain with algebraically closed residue field, there exist rational Abhyankar valuations $\nu'$ such that certain toric embedded uniformizations of $\nu'$ uniformize $\nu$.}\par\noindent
 Then, one should prove using the quasi-compactness of the Zariski-Riemann manifold that there are finitely many valuations whose uniformizations suffice to uniformize all valuations, and finally to glue up the corresponding re-embeddings into a single one in which a toric modification uniformizes those valuations and thus locally resolves singularities..\par\noindent After that one can attack the problem of globalization.\par\medskip\noindent.\par\noindent
\textbf{Acknowledgments:} I am grateful to Dale Cutkosky for detecting imperfections in a first version of this paper and for several stimulating comments and suggestions. I am also grateful to Charles Favre, Evelia Garc\'\i a Barroso, Mohammad Moghaddam and Hussein Mourtada for interesting discussions, and to Olivier Piltant, Guillaume Rond, Mark Spivakovsky and the referee for useful comments on preliminary versions. I also thank Jenia Tevelev whose theorem gives encouragement to proceed with the program, and Patrick Popescu Pampu who communicated to me the question of D.A. Stepanov. Part of this work was completed in the excellent environment of the MSRI in Berkeley in the context of the Commutative Algebra year 2012-2013, with the support of the National Science Foundation under Grant No. 0932078 000.
\section{Appendix: On Hironaka's division theorem and flattening}\label{appendix}\par\medskip
Since the references \cite{Hi} and \cite{H-L-T} are not easily available, we present here for the convenience of the reader a brief summary of Hironaka's flattening adapted to our situation.\par
Let $A$ be a quotient of a power series ring over the field $k$ and let $(t_1,\ldots, t_n)$ be indeterminates. We are going to study submodules of $A$-modules of the form $B=\bigoplus_{i=0}^nA[[t_1,\ldots ,t_i]]^{a_i}$, where $a_i\in\N$ and the summand for $i=0$ is $A$. A homomorphism $\Delta \colon B\to B'$ of $A$-modules between two such sums is said to be \emph{natural} if all the homomorphisms $A[[t_1,\ldots ,t_i]]\to A[[t_1,\ldots ,t_j]]$ induced from $\Delta$ by composing injections and projections of summands are homomorphisms of $A[[t_1,\ldots ,t_{{\rm min}(i,j)}]]$-modules. One of Hironaka's versions of the division theorem goes as follows in our context:\Par
\begin{proposition}\label{Divis}{\rm (See \cite{Hi}, lemma 4.9)} Let $q$ be an integer and $J$ be a $A[[t_1,\ldots ,t_n]]$-submodule of $S=A[[t_1,\ldots ,t_n]]^q$. After a suitable invertible $k$-linear transformation of the $t_i$ one can find a system of non negative integers $a_i$ and a natural homomorphism of $A$-modules $$\Delta\colon B=\bigoplus_{i=0}^nA[[t_1,\ldots ,t_i]]^{a_i}\to S=A[[t_1,\ldots ,t_n]]^q$$ such that
\begin{enumerate}
\item $\Delta$ induces a surjective homomorphism of $A$-modules $\delta\colon B\to S/J$,
\item We have ${\rm Ker}\delta\subset m_AB$, where $m_A$ is the maximal ideal of $A$.
\end{enumerate}
\end{proposition}
\begin{proof}The proof is by induction on $n$ using the classical Weierstrass division theorem, which corresponds to the case $A=k$ and $q=1$, with $J$ a principal ideal of $k[[t_1,\ldots ,t_n]]$. Then we have $B=k[[t_1,\ldots ,t_{n-1}]]^a$ and the kernel of $\delta$ is zero. One can also view it as a consequence of the formal version of the Grauert-Hironaka division theorem (see \cite{Hi2}, \cite{Galligo}).\end{proof}
Taking now $q=1$ so that $J$ is an ideal of $S=A[[t_1,\ldots ,t_n]]$, we take a presentation $\Delta\colon B\to S$ as above and, setting $t(i)=(t_1,\ldots ,t_i)$, consider the unique expansion of elements $g$ of $\Delta^{-1}(J)$ as $$g=\sum_{i=0}^n\sum_{j=1}^{a_i}\sum_{\alpha\in\N^i}g_{ij\alpha}t(i)^\alpha ,$$
with $g_{ij\alpha}\in A$.\Par
The ideal $\Ff$ of $A$ generated by the coefficients $g_{ij\alpha}$ as $g$ runs through $\Delta^{-1}(J)$ is the \emph{universal flattener} of the map $A\to S/J$. It means that $S/J\otimes_A A/\Ff$ is a flat $A/\Ff$-module and the map $A\to A/\Ff$ is minimal for this property. The first statement is easy to verify since by construction we have $S/J\otimes_A A/\Ff= B\otimes_A A/\Ff= \bigoplus_{i=0}^nA/\Ff[[t_1,\ldots ,t_i]]^{a_i}$.\par
If the ideal $\Ff$ is invertible, there exists an element $h\in S/J$ which is not contained in the ideal $m_A.S/J$ but satisfies $h\Ff .S/J=0$. Indeed, assuming that a generator of $\Ff$ is $g_{ij\alpha}$, we can write $g=g_{ij\alpha}g'$ with $g'\in B$ and $g'\notin m_AB$ since one of its coefficients is equal to $1$. The image $h=\delta(g')\in S/J$ is not in $m_AS/J$ since ${\rm Ker}\delta\subset m_AB$, and by the definition of $\Ff$ we have $\delta (g)=h\delta (g_{ij\alpha})=0$, which means $h\Ff .S/J=(0)$.\Par
So when we blow up the flattener of $A\to S/J$, localize at a point and complete to obtain a map $A\to A'$ of quotients of power series rings, the flattener of\break $S/J\otimes_AA'$, which is $\Ff.A'$ by the universal property, is an invertible ideal, and the strict transform $ (S/J)'$ of $ S/J$, which is $S/J\otimes_AA'$ divided by its $\Ff.A'$-torsion, is a quotient of $(S/J\otimes_AA')/h'$ with $h'\notin m_{A'}(S/J\otimes_AA')$. The inclusion of the fiber of the space corresponding to $(S/J)'$ over the (point corresponding to the) maximal ideal $m_{A'}$ into the fiber of the  space corresponding to $S/J\otimes_AA'$ is strict since $h'\notin m_{A'}(S/J\otimes_AA')$. Now we take the flattener of the $A'$-module $(S/J)'$, which is an ideal of $A'$, blow it up and localize the corresponding space at a point lying over the maximal ideal of $A'$. At each iteration of this process, the inclusion of the fiber of the space corresponding to the strict transform into the previous fiber is strict, so that it has to stop after finitely many iterations, and then the flattener of the corresponding strict transform must be zero, which means that the corresponding $\delta$ for this strict transform is an isomorphism and the map is flat. In the proof of proposition \ref{sepex} we have a valuation of $S/J$ which picks a point in each of the strict transforms, and we use the strict inclusion of fibers at that point.


\providecommand{\bysame}{\leavevmode\hbox to3em{\hrulefill}\thinspace}
\providecommand{\MR}{\relax\ifhmode\unskip\space\fi MR }
\providecommand{\MRhref}[2]{%
  \href{http://www.ams.org/mathscinet-getitem?mr=#1}{#2}
}
\providecommand{\href}[2]{#2}
\begin{thebibliography}{}

\end{thebibliography}


\frenchspacing
 \begin{thebibliography}{87}

\bibitem{A} S.S. Abhyankar,  Irreducibility Criterion for Germs of Analytic Functions of Two Complex
Variables. \textit{Advances in Mathematics} \textbf{74}, (1989), 190-257.

\bibitem{ACH} M.E. Alonso, F.J. Castro-Jim\'enez, and H. Hauser, Effective algebraic power series. \textit{Manuscript 2012}, available at http://homepage.univie.ac.at/herwig.hauser/index.html/Publikationen

\bibitem{An} G. Angerm\"uller, Die Wertehalbgruppe einer ebenen irreduziblen algebroiden Kurve. \textit{Math. Z.} \textbf{53} (1977), 267-282.

\bibitem{As} A. Assi, Irreducibility criterion for quasi-ordinary polynomials. \textit{Journal of singularities} \textbf{4} (2012), 23-34.

\bibitem{B3}
N.~Bourbaki, \textit{\'{E}l\'ements de math\'ematique. Alg\`ebre Commutative, Chap. I-VIII}. Masson, Paris

\bibitem{B1} N.~Bourbaki, \textit{\'{E}l\'ements de math\'ematique. Alg\`ebre commutative, Chap. VIII et IX}. Masson, Paris 1983.

\bibitem{B4} N.~Bourbaki, \textit{\'{E}l\'ements de math\'ematique. Alg\`ebre commutative, Chap. X}. Masson, Paris 1998.

\bibitem{B2}
N.~Bourbaki,  \textit{\'{E}l\'ements de math\'ematique. {VII}.
  {A}lg\`ebre, {C}hap. {III}: {A}lg\`ebre multilin\'eaire}. Actualit\'es
  Sci. Ind., No. 1044, Hermann et Cie., Paris, 1948.

\bibitem{Br} H. Bresinsky, Semigroups corresponding to algebroid branches in the plane.  \textit{Proc. Amer. Math.
Soc.} \textbf{32} (1972), 38I-384.

\bibitem{C} A. Campillo, Algebroid Curves in Positive Characteristic. \textit{Springer Lecture Notes in Math., No. 813}, Springer 1980.

\bibitem{C-G} A. Campillo, C. Galindo, On the graded algebra relative to a valuation. \textit{Manuscripta Math.} \textbf{92} (1997), 173-189.

\bibitem{CCD} E. Cattani, R. Curran and A. Dickenstein, Complete intersections in toric ideals.  \textit{Proc. Amer. Math. Soc.} \textbf{135} (2007), 329-335.

\bibitem{CM} V. Cossart and G. Moreno-Soc\'\i as, Racines approch\'ees, suites g\'en\'eratrices, suffisance des jets. \textit{Annales Fac. Sci. Toulouse} \textbf{14} (3) (2005), 353-394.

\bibitem{CP} V. Cossart and O. Piltant, Resolution of singularities of threefolds in positive characteristic. II. \textit{J. of Algebra} \textbf{321} (7) (2009), 1836-1976.

 \bibitem{C-T} S. D. Cutkosky and B. Teissier, Semigroups of valuations on local rings. \textit{Michigan Math. J.} \textbf{57} (2008), 173-193.

\bibitem{C-V} S. D. Cutkosky and P. A. Vinh, Valuation semigroups of two dimensional local rings. To appear in the \textit{Proceedings of the London Math. Soc.} \textit{ArXiv:1105.1448v1.}

\bibitem{DC-P} C. De Concini, C. Procesi, Complete symmetric varieties. II. Intersection theory. In \textit{Algebraic groups and related topics, (Kyoto/Nagoya, 1983)}. 481-513,  Adv. Stud. Pure Math., 6, North-Holland, Amsterdam, 1985.

\bibitem{E-GZ} W. Ebeling,  S. M. Gusein-Zade, On divisorial filtrations associated with Newton diagrams. \textit{J. Singul.} \textbf{3} (2011), 1-7.

\bibitem{Ei-S}
D. Eisenbud and B. Sturmfels, Binomial ideals. \textit{Duke Math. J.} \textbf{84} (1) (1996), 1-45.

\bibitem{ELS} L.~Ein, R. Lazarsfeld, and K. Smith, Uniform approximation of valuation ideals in smooth function fields. \textit{Amer. J. of Math.} \textbf{125} (2) (2003), 409-440.

 \bibitem{EL} George A. Elliott, On totally ordered groups. In \textit{Ring Theory, Waterloo, 1978}  (ed. by D. Handelman and J. Lawrence). Springer Lecture Notes in Mathematics, No. 734, 1-49.

\bibitem{Ewald}
G.~Ewald, \textit{Combinatorial convexity and Algebraic Geometry}. Graduate Texts in Mathematics No. 168. Springer 1996.

\bibitem{Ewald-Ishida}
G.~Ewald and M-N. Ishida, Completion of real fans and Zariski-Riemann spaces.
  \textit{Toh\^oku Math. J.} \textbf{58} (2) (2006), 189-218.

\bibitem{Favre-Jonsson} C.~Favre and M.~Jonsson, \textit{The valuative tree}. Lecture Notes in Mathematics, Vol.
1853, Springer-Verlag, Berlin, 2004.

\bibitem{Galligo} A. Galligo, Th\'eor\`eme de division et stabilit\'e en g\'eom\'etrie analytique locale. \textit{Ann. Institut Fourier} \textbf{29} (2) (1979), 109-184.

\bibitem{GB-P} E. Garc\'\i a Barroso and A. P\l oski, An approach to plane algebroid branches.  \textit{arXiv:1208.0913,} to appear.

\bibitem{GB-GP} E. Garc\'\i a Barroso and P. Gonz\'alez-P\'erez, Decomposition in bunches of the critical locus of a quasi-ordinary map. \textit{Compos. Math.}  \textbf{141},  (2) (2005) 461√ê486.

\bibitem{Gi} R. Gilmer, \textit{Commutative semigroup rings}. University of Chicago Press, Chicago, 1984.

\bibitem{G-T} R. Goldin and B. Teissier, Resolving singularities of plane analytic branches with one toric morphism. In
\textit{Resolution of Singularities, a research textbook in tribute to Oscar Zariski} (ed. by H. Hauser, J. Lipman, F. Oort, and A. Quir\'os). Progress in Math. Vol. 181,
Birkh\"auser, Basel, 2000, 315-340.

 \bibitem{GP1}
P.~Gonz\'alez P\'erez, Singularit\'es quasi-ordinaires toriques et poly\`edre de Newton du discriminant. \textit{Canad. J. Math.} \textbf{52} (2) (2000),  348-368.

\bibitem{GP2}
P.~Gonz\'alez P\'erez, Toric embedded resolutions of quasi-ordinary hypersurfaces.
\textit{Annales Inst. Fourier (Grenoble)} \textbf{53} (2003), 1819-1881.

\bibitem{GP3}
P.~Gonz\'alez P\'erez, Approximate roots, toric resolutions and deformations of a plane branch. \textit{Journ. Math. Soc. of Japan} \textbf{62} (3) (2010), 975-1004.

\bibitem{GP-T1}
P.~Gonz\'alez P\'erez and B. Teissier, Embedded resolutions of non necessarily normal affine toric varieties.
\textit{Comptes-rendus Acad. Sci. Paris, Ser.1} \textbf{(334)} (2002), 379-382. Available at http://people.math.jussieu.fr/$_{\tilde{\ }}$teissier/

\bibitem{GP-T2}
P.~Gonz\'alez P\'erez and B. Teissier, Toric geometry and the Semple-Nash modification.
\textit{Revista de la Real Academia de Ciencias Exactas, Fisicas y Naturales, Serie A, Matem\'aticas} \textbf{108} (1) (2014), 1-46.  Available at http://people.math.jussieu.fr/$_{\tilde{\ }}$teissier/

\bibitem{Gr} A. Granja, M.C. Mart\'inez, and C. Rodr\'iguez, Analytically irreducible polynomials with coefficients in a real-valued field. \textit{Proceedings of the AMS} \textbf{138} (10) (2010), 3449-3454.

 \bibitem{H-O-S} F.J. Herrera Govantes, M. A. Olalla Acosta and Mark Spivakovsky, Valuations in algebraic field extensions. \textit{Journal of Algebra} \textbf{312}  (2007), 1033-1074.

\bibitem{HOST} F. J. Herrera Govantes, M. A. Olalla Acosta, M. Spivakovsky and B. Teissier. Extending a valuation centered in a local domain to the formal completion. \textit{Proc. of the London Math. Soc.}  \textbf{105} (3) (2012), 571-621. doi:10.1112/plms/pds002.

\bibitem{EGA4.1} A. Grothendieck, J. Dieudonn\'e, \textit{El\'ements de G\'eom\'etrie Alg\'ebrique, Chap. IV, premi\`ere partie}, Pub. Math. I.H.E.S, No.20. 1964.

\bibitem{EGA4} A. Grothendieck, J. Dieudonn\'e. \textit{El\'ements de G\'eom\'etrie Alg\'ebrique, Chap. IV, 4\`eme partie}, Pub. Math. IHES, No.32, 1967.

\bibitem{Hi} H. Hironaka, subanalytic sets and subanalytic maps. \textit{Publ. Instituto Leonida Tonelli}, Pisa, 1973.

\bibitem{Hi2} H. Hironaka, Idealistic exponents of singularities. In \textit{Algebraic geometry (J. J.
Sylvester Sympos., Johns Hopkins Univ., Baltimore, Md., 1976)}, Johns Hopkins Univ. Press, Baltimore, Md.,1977, 52-125.

\bibitem{H-L-T} H. Hironaka, M. Lejeune-Jalabert and B. Teissier, Platificateur local en g\'eom\'etrie analytique et aplatissement local. In  \textit{Singularit\'es \`a Carg\`ese (Rencontre "Singularit\'es en G\'eom\'etrie Analytique", Institut d'Etudes Scientifiques de Carg\`ese, 1972)},  Ast\'erisque, Nos. 7 et 8 (1973), 441-463.

\bibitem{JM} M. Jonsson, M. Mus\c ta\u ta, Valuations and asymptotic invariants for sequences of ideals. ArXiv 1011.3699v.3. To appear in \textit{Annales de l'Institut Fourier}.

\bibitem{K} I. Kaplansky, Maximal fields with valuations. \textit{Duke Math. J.} \textbf{9} (1942), 303-321.

\bibitem{Ka} O. Kashcheyeva, A construction of generating sequences of valuations centered in dimension 3 regular local rings. \textit{Preprint, 2012}.

\bibitem{Ka-Kh1} K. Kaveh, A. Khovanskii, Newton-Okounkov bodies, semigroups of integral points, graded algebras and intersection theory. \textit{Annals of Math.} \textbf{176} (2) (2012), 925-978 .

\bibitem{KK1} H. Knaf,  F.-V. Kuhlmann, Abhyankar places admit local uniformization in any characteristic. \textit{ Ann. Sci. \'Ecole Norm. Sup.} \textbf{38}  (4)  (2005),   833--846.

\bibitem{KK2} H. Knaf, F.-V. Kuhlmann, Every place admits local uniformization in a finite extension of the function field,  \textit{Adv. Math.} \textbf{221} (2) (2009), 428-453.

\bibitem{Kr} W. Krull, Allgemeine Bewertungstheorie. \textit{J.f. d. Reine u. angew. Math.} \textbf{167} (1932), 160-196.

\bibitem{Ku} F.-V.\ Kuhlmann, Valuation theoretic and model theoretic aspects of local
uniformization.  In \textit{Resolution of Singularities, a research textbook in tribute to Oscar Zariski} (ed. by H. Hauser, J. Lipman, F. Oort, and A. Quir\'os). Progress in Math. Vol. 181,
Birkh\"auser, Basel 2000, 381-456.

\bibitem{Ku2}   F.-V.\ Kuhlmann, Maps on ultrametric spaces, Hensel's lemma, and differential equations over valued fields. \textit{Comm. Algebra} \textbf{39} (5) (2011) 1730-1776.

\bibitem{Ku3}   F.-V.\ Kuhlmann, Approximation of elements in henselizations. \textit{Manuscripta math.} \textbf{136} (2011), 461-474.

\bibitem{KKMS}
G. Kempf, F. Knudsen, D. Mumford, B. Saint-Donat. \textit{Toroidal embeddings I},
Springer Lecture Notes in Math. No. 339, Springer 1973.

\bibitem{Mac} S. MacLane, A construction for Absolute Values in Polynomial Rings. \textit{Transactions of the A.M.S.} \textbf{40} 3 (1936), 363-395.

\bibitem{Ma} W. Mahboub, Key polynomials. \textit{Journ. Pure and applied Alg.} \textbf{217} (6) (2013), 989-1006.

\bibitem{Ma2} W. Mahboub, Th\`ese, Universit\'e de Toulouse, Nov. 2013.

\bibitem{M1} M. Moghaddam, A construction for a class of valuations of the field $k(X_1,\ldots ,X_d,Y)$ with large value group. \textit{J. of Algebra} \textbf{319} 7 (2008) 2803-2829.

\bibitem{M2} M. Moghaddam, On Izumi's theorem on comparison of valuations, \textit{Kodai Math. J.}
\textbf{34} (2011), 16-30.

\bibitem{M3} M. Moghaddam, Realization of a certain class of semigroups as value semigroups of valuations. \textit{Bulletin of the Iranian Mathematical Society} \textbf{35}, (1) (2009), 61-95.

\bibitem{M4} M. Moghaddam, Embedding of valuations. \textit{Manuscript}.

\bibitem{Mo} R. Morelli, The birational geometry of toric varieties. \textit{J. Algebraic Geom} \textbf{5} (4) (1996), 751-782.

\bibitem{Ne} B.H. Neumann, On ordered division rings. \textit{Trans. Amer. Math. Soc.} \textbf{66} (1949), 202-252.

\bibitem{O} A. Ostrowski, Untersuchung zur arithmetische Theorie der K\"orper,
part III. \textit{Math. Zeit.} \textbf{39} (1934).

\bibitem{Pi} H. Pinkham, \textit{Deformations of algebraic varieties with ${\mathbf G}_m$ action}. Ast\'erisque No. 20. (1974)

\bibitem{PP} P. Popescu Pampu, Approximate roots. In \textit{Valuation theory and its applications, Vol. 1} Proceedings of the Saskatoon Conference and Workshop on valuation theory, Saskatoon 1999 (ed. by F-V. Kuhlmann, S. Kuhlmann, M. Marshall). Fields Institute Communications, Vol. 33, 2003, 285-321.

\bibitem{PP2} P. Popescu Pampu, sur le contact d'une hypersurface quasi-ordinaire avec ses hypersurfaces polaires. \textit{Journ. Inst. Math. Jussieu} \textbf{3} (1) (2004), 105-138.

\bibitem{R} L. R\'edei, \textit{The theory of finitely generated commutative semigroups}. Internat. Ser. Monogr. Pure Appl. Math., 82, Pergamon, Oxford,1965.

\bibitem{Ri} P. Ribenboim, Corps maximaux et complets pour des valuations de Krull. \textit{Math. Zeitschr.} \textbf{69} (1958), 466-479.

\bibitem{Ri2} P. Ribenboim, \textit{The theory of classical valuations}, Springer 1999.

\bibitem{Ri3} P. Ribenboim, \textit{Th\'eorie des Valuations}, Presses de l'Universit\'e de Montr\'eal, Novembre 1964.

\bibitem{Roq} P. Roquette, History of Valuation Theory, Part I. In \textit{Valuation theory and its applications, Vol. 1} Proceedings of the Saskatoon Conference and Workshop on valuation theory, Saskatoon 1999 (ed. by F-V. Kuhlmann, S. Kuhlmann, M. Marshall). Fields Institute Communications, Vol. 33, 2003, 291-355.

\bibitem{S} A. Seidenberg, Valuation ideals in polynomial rings, \textit{Trans. Amer. Math. Soc.} \textbf{57} (1945), 387-425.

\bibitem{SSa} J.-C. San Saturnino, Th\'eor\`eme de Kaplansky effectif pour des valuations de rang 1 centr\'ees sur des anneaux locaux r\'eguliers et complets. \textit{Ann. Inst. Fourier} \textbf{63} (2013). Available at ArXiv:1203.4283

\bibitem{St} D. A. Stepanov, Universal valued fields and lifting points in local tropical varieties. \textit{ArXiv:1304 7726v1}.

\bibitem{Stu}  B. Sturmfels, \textit{Gr\"obner bases and convex polytopes}. University Lecture Series Vol. 8, AMS Providence, 1996.

\bibitem{Te0} B. Teissier, Appendix to \textit{The Moduli problem for plane branches}, by O. Zariski, translated by Ben Lichtin, University Lecture Series, Volume 39, AMS, 2006. In French: Le probl\`eme des modules pour les branches plane, Hermann, Paris, 1986, et Ecole Polytechnique, Centre de Math\'ematiques, 1973.

\bibitem{Te1}
B. Teissier, Valuations, deformations, and toric geometry.
In \textit{Valuation Theory and its applications, Vol. II}, Fields Inst. Commun. 33, AMS., Providence, RI., 2003, 361-459.  Available at http://people.math.jussieu.fr/$_{\tilde{\ }}$teissier/

\bibitem{Te2}
B. Teissier,  Monomial ideals, binomial ideals, polynomial ideals.
\textit{Trends in Commutative Algebra}, MSRI publications, Cambridge University Press 2004, 211-246. Available at http://people.math.jussieu.fr/$_{\tilde{\ }}$teissier/

\bibitem{Te3} B. Teissier, A viewpoint on local resolution of singularities. Oberwolfach Workshop on Singularities, September 2009. \textit{Oberwolfach Reports}, Vol. 6, No. 3, 2009. European Math. Soc. Publications. Available at http://people.math.jussieu.fr/$_{\tilde{\ }}$teissier/

\bibitem{Tem} M. Temkin, Inseparable local uniformization. \textit{J. of Algebra} \textbf{373} (2013), 65-119.

\bibitem{Tev} J. Tevelev, On a question of B. Teissier. \textit{Collectanea Math.} \textbf{65} (1) (2014), 61-66. (Published on line February 2013. DOI 10.1007/s13348-013-0080-9)

\bibitem{V0} M.~Vaqui\'e, Valuations. In \textit{Resolution of Singularities, a research textbook in tribute to Oscar Zariski} (ed. by H. Hauser, J. Lipman, F. Oort, and A. Quir\'os). Progress in Math. Vol. 181,
Birkh\"auser, Basel 2000, 539-590.

\bibitem{V1} M.~Vaqui\'e, Extension d'une valuation. \textit{Trans. Amer. Math. Soc.} \textbf{359} (7) (2007), 3439-3481 (electronic).

\bibitem{V2} M.~Vaqui\'e, Famille admissible de valuations et d\'efaut d'une extension. \textit{J. of Algebra} \textbf{311} (2) (2007), 859-876.

\bibitem{V3}M.~Vaqui\'e, Extensions de valuation et polygone de Newton. \textit{Ann. Inst. Fourier} \textbf{58} (7) (2008), 2503-2541.

\bibitem{V4} M.~Vaqui\'e, Famille admise associ\'ee \`a une valuation de $K[x]$. In \textit{Singularit\'es Franco-Japonaises}, S\'emin. Congr., 10, Soc. Math. France, Paris, 2005, 391-428.

\bibitem{V5} M.~Vaqui\'e, Alg\`ebre gradu\'ee associ\'ee \`a une valuation de $K[x]$. In \textit{Singularities in geometry and topology 2004}, Adv. Stud. Pure Math., 46, Math. Soc. Japan, Tokyo, 2007, 259-271.

\bibitem{Z} O.~Zariski, Local Uniformization on Algebraic Varieties. \textit{Annals of Math., second series,} \textbf{41} (4) (1940), 852-896.

\bibitem{Z2} O. Zariski, Theory and applications of holomorphic functions on algebraic varieties over arbitrary ground fields.
\textit{Mem. Amer. Math. Soc.},  \textbf{5} (1951), 1-90. Reprinted in: \textit{Collected papers}, Vol. II, MIT Press,
1973, 72-161.

\bibitem{Z-S} O. Zariski and P. Samuel, \textit{Commutative Algebra I, II}. Van
Nostrand 1958, 1960. Reprints: \textit{Graduate Texts in Mathematics, Vols. 28, 29}, Springer 1975.

\end{thebibliography}
\end{document}